\documentclass[11pt]{amsart}

\usepackage{a4wide}
\usepackage[T1]{fontenc}
\usepackage{amsmath}
\usepackage{amsfonts}
\usepackage{amssymb}
\usepackage{amsthm}
\usepackage{esint}
\usepackage{graphicx}
\usepackage{enumerate}
\usepackage{mathrsfs}
\usepackage{cite}
\usepackage{wasysym}
\usepackage{xr}
\usepackage{constants}
\usepackage[footnotesize]{caption}

\resetconstant

\newconstantfamily{R}{
symbol=R,
format=\arabic,
}

\newtheorem{thm}{Theorem}[section]
\newtheorem*{thm*}{Theorem}
\newtheorem{prop}[thm]{Proposition}
\newtheorem{lem}[thm]{Lemma}
\newtheorem{cor}[thm]{Corollary}

\newtheorem{introthm}{Theorem}

\newtheorem{introprop}{Proposition}

\theoremstyle{definition}
\newtheorem{defin}[thm]{Definition}
\newtheorem{rem}[thm]{Remark}
\newtheorem{ex}[thm]{Example}
\newtheorem{fact}[thm]{Observation}

\DeclareMathOperator{\dist}{dist}
\DeclareMathOperator{\diam}{diam}
\DeclareMathOperator{\opspan}{span}
\DeclareMathOperator{\opint}{int}
\DeclareMathOperator{\oplk}{lk_2}
\DeclareMathOperator{\opang}{\sphericalangle}
\DeclareMathOperator{\opid}{id}
\DeclareMathOperator{\aff}{aff}
\DeclareMathOperator{\simp}{\bigtriangleup}
\DeclareMathOperator{\perm}{Perm}
\DeclareMathOperator{\opim}{im}

\DeclareMathOperator{\BAP}{BAP_m}
\DeclareMathOperator{\dgras}{\varangle}

\newcommand{\A}{\mathcal{A}}
\newcommand{\F}{\mathcal{F}}
\newcommand{\E}{\mathcal{E}}
\newcommand{\G}{\mathscr{G}}
\newcommand{\Dxy}{\mathcal{D}}
\newcommand{\CDxy}{\overline{\mathcal{D}}}
\newcommand{\Dom}{\mathcal{D}}

\newcommand{\Reg}{\mathcal{V}}
\newcommand{\Rn}{\mathbb{R}^n}
\newcommand{\N}{\mathbb{N}}
\newcommand{\Z}{\mathbb{Z}}
\newcommand{\Sphere}{\mathbb{S}}
\newcommand{\Cone}{\mathbb{C}}
\newcommand{\Ball}{\mathbb{B}}
\newcommand{\CBall}{\overline{\mathbb{B}}}
\newcommand{\Disc}{\mathbb{D}}
\newcommand{\CDisc}{\overline{\mathbb{D}}}
\newcommand{\Shell}{\mathbb{A}}
\newcommand{\R}{\mathbb{R}}
\newcommand{\HM}{\mathscr{H}}
\newcommand{\HD}{d_{\mathcal{H}}}
\newcommand{\K}{\mathcal{K}}
\newcommand{\Kphi}{\mathcal{K}_{\varphi}}

\newcommand{\face}{\mathfrak{fc}}
\newcommand{\height}{\mathfrak{h}}
\newcommand{\hmin}{\height_{\min}}
\newcommand{\Slk}{\mathscr{S}}
\newcommand{\red}{\rho\varepsilon\delta}

\newcommand{\nbeta}{\bar{\beta}_m}
\newcommand{\ntheta}{\bar{\theta}_m}

\title[Integral Menger curvature]
{Integral Menger curvature for sets of arbitrary dimension and codimension}
\author{S{\l}awomir Kolasi{\'n}ski}
\address{Institute of Mathematics\\
  University of~Warsaw\\
  Banacha~2, 02-097 Warsaw\\
  Poland}
\email{s.kolasinski@mimuw.edu.pl}
\urladdr{http://www.mimuw.edu.pl/~skola}
\thanks{The author wishes to thank P. Strzelecki for his advice and many
  inspiring conversations.}
\date{\today}

\newcommand{\mysection}[1]{}
\newcommand{\mysubsection}[1]{}
\newcommand{\mysubsubsection}[1]{}
\let\mysection=\section
\let\mysubsection=\subsection
\let\mysubsubsection=\subsubsection

\begin{document}

\begin{abstract}
  We propose a notion of integral Menger curvature for compact, $m$-dimensional
  sets in $n$-dimensional Euclidean space and prove that finiteness of this
  quantity implies that the set is $C^{1,\alpha}$ embedded manifold with the
  H{\"o}lder norm and the size of maps depending only on the curvature. We
  develop the ideas introduced by Strzelecki and von der Mosel
  [Adv. Math. 226(2011)] and use a similar strategy to prove our results.
\end{abstract}


\maketitle
\tableofcontents

\mysection*{Introduction}

Menger curvature is a notion defined for triples of points in an Euclidean
space. Let $R(x,y,z)$ be the radius of the smallest circle passing through $x$,
$y$ and $z$. Then the \emph{Menger curvature} is just the inverse of
$R(x,y,z)$. This notion can be used to define many different types of curvatures
for $1$-dimensional sets in $\Rn$ and there are several contexts in which
curvatures of this kind occur.

First, there are works motivated by natural sciences and the search for good
models of DNA molecules, protein structures or polymer chains; see for example
the paper by Banavar et al.~\cite{MR1966322} or the book by Sutton and
Balluffi~\cite{SB97}. Long, entangled objects are usually modeled as
$1$-dimensional curves embedded in $\R^3$. The goal is to find analytical tools
catching their physical properties like thickness and lack of
self-intersections. There are several approaches towards this problem. One can
impose a lower bound on the \emph{global radius of curvature} defined as the
infimum of $R(x,y,z)$ over all points $x$, $y$ and $z$ lying on a curve. Such
constraints were studied e.g. by Gonzalez, Maddocks, Schuricht and von der
Mosel~\cite{MR1883599}, by Cantarella, Kusner and Sullivan~\cite{MR1933586} or
by Gonzalez and de la Llave~\cite{MR1953628}. The existence of minimizers of
curvature in a given isotopy class has been proven as well as the existence of
so called \emph{ideal knots}, i.e. knots which minimize the ratio of the length
to the thickness. There are also results considering the shape and regularity of
ideal knots; see Cantarella, Kusner and Sullivan~\cite{MR1933586}, Cantarella et
al.~\cite{MR2284052}, Durumeric~\cite{MR2355512} or Schuricht and von der
Mosel~\cite{MR2033143}. This list of publications is, of course, not complete.
For more information on these topics we refer the reader to the cited articles.

Quite different approach was suggested by Strzelecki, Szuma{\'n}ska and von der
Mosel in~\cite{MR2489022} and~\cite{MR2668877}, where the authors studied
''soft'' knot energies defined as the integral of Menger curvature in some
power. They proved self-avoidance effects and $C^{1,\alpha}$ regularity of knots
with finite energy. Furthermore they showed some analogues of the Sobolev
imbedding theorem, which suggests that Menger curvature is a good replacement
for the second derivatives in a non-smooth setting. Strzelecki and von der Mosel
in~\cite{MR2197957} and~\cite{MR2214619} were also able to apply their ''soft''
potentials to prove the existence of minimizers of some constrained variational
problems in a given isotopy class.

Yet another context, mathematically probably the deepest one, in which
curvatures of non-smooth objects occur is harmonic analysis. Independently of
physical motivations, the research on removability of singularities of bounded
analytical functions led to the study of integral curvatures. Surveys of
Mattila~\cite{MR1648114} and Tolsa~\cite{MR2275656} explain the connection
between these subjects. L{\'e}ger~\cite{MR1709304} proved that $1$-dimensional
sets with finite integral Menger curvature are $1$-rectifiable, which was a
crucial step in the proof of Vitushkin's conjecture.

Intensive research is being done on generalizations of Menger curvature for sets
of higher dimension. It occurs that one cannot define $k$-dimensional Menger
curvature using integrals of the radius of a circumsphere of
$(k+2)$-points. This ''obvious'' generalization fails because of examples
(see~\cite[Appendix B]{0911.2095}) of very smooth embedded manifolds for which
this kind of curvature would be unbounded.

Lerman and Whitehouse in~\cite{0805.1425} and in~\cite{MR2558685} suggested a
whole class of different high dimensional Menger-type curvatures basing on so
called polar sine function. They proved \cite[Theorems 1.2 and 1.3]{MR2558685}
that their integral curvatures can be used to characterize $d$-dimensional
rectifiable measures. This established a connection between the theory of
non-smooth curvatures and uniform rectifiablility in the sense of David and
Semmes~\cite{MR1251061}.

Similar but different notion of integral Menger-type curvature for surfaces in
$\R^3$ was introduced by Strzelecki and von der Mosel~\cite{0911.2095}. They
proved that finiteness of their functional implies H{\"o}lder regularity of the
normal vector. They also applied their own results to prove existence of area
minimizing surfaces in a given isotopy class under the constraint of bounded
curvature. Our work is focused on generalizing these results to sets of
arbitrary dimension and codimension.

For any set of $m+2$ points $\{ x_0, x_1, \ldots, x_{m+1} \} \subseteq \Rn$ we
define the discrete curvature
\begin{displaymath}
  \K(x_0,\ldots,x_{m+1})
  := \frac{\HM^{m+1}(\simp(x_0,\ldots,x_{m+1}))}{\diam(\{ x_0, x_1, \ldots, x_{m+1} \})^{m+2}} \,,
\end{displaymath}
where $\simp(x_0,\ldots,x_{m+1})$ denotes the convex hull of the set $\{ x_0,
\ldots, x_{m+1} \}$, which in a typical case will be an $(m+1)$-dimensional
simplex. For $m = 2$ one can easily prove that the above discrete curvature $\K$
is always smaller than the one defined in~\cite{0911.2095} but for tetrahedrons
which are roughly regular both quantities are comparable. This comes from the
fact that the area of a tetrahedron is always bounded from above by $4\pi$
times the square of the diameter.

Let $\Sigma \subseteq \Rn$ be any $m$-dimensional, compact set and let $p > 0$.
We introduce the \emph{$p$-integral Menger-type curvature} (abbreviated as the
\emph{$p$-energy}) of $\Sigma$
\begin{displaymath}
  \E_p(\Sigma) 
  := \int_{\Sigma^{m+2}} \K(x_0,\ldots,x_{m+1})^p\ d\HM^m_{x_0} \cdots d\HM^m_{x_{m+1}} \,,
  \quad
  \Sigma^{m+2} = \underbrace{\Sigma \times \cdots \times \Sigma}_{(m+2) \text{ times}} \,.
\end{displaymath}
This kind of energy is finite if $\Sigma \subseteq \Rn$ is a compact $C^2$
manifold (cf. Proposition~\ref{prop:beta-curv} and
Corollary~\ref{cor:C2mani}). In a forthcoming, joint paper with Marta
Szuma{'n}ska~\cite{SKMS}, we prove that graphs of a $C^{1,\nu}$ functions also
have finite integral Menger curvature whenever $\nu > \nu_0 = 1 -
\frac{m(m+1)}{p}$ and we construct examples of $C^{1,\nu_0}$ functions with
graphs of infinite $p$-energy.

In~\cite{0911.2095} the authors define a similar energy functional
$\mathcal{M}_p$, which satisfies $\E_p(\Sigma) \le \mathcal{M}_p(\Sigma)$ when
$m=2$ and $n=3$. Next, they prove that whenever $\mathcal{M}_p(\Sigma)$ is
finite for some $p > 8$, then there is a fixed scale $R > 0$ which depends only
on the energy $\mathcal{M}_p$ such that for any $r < R$ and any $x \in \Sigma$
we have
\begin{displaymath}
  \HM^2(\Sigma \cap \Ball(x,r)) \ge \frac{\pi}2 r^2 \,.
\end{displaymath}
What is significant in this theorem, is that the scale $R$ below which we have
the above inequality depends only on the energy bounds of $\Sigma$. This result
is crucial for the rest of the proofs. After establishing this uniform Ahlfors
regularity, the authors prove the existence of tangent planes and estimate their
oscillation. This gives $C^{1,\alpha}$ regularity for $\Sigma$, with $\alpha = 1
- \frac 8p$ and with H{\"o}lder constant depending only on the energy bounds.

This paper is devoted to proving analogues of above theorems in the case of sets
of arbitrary dimension and codimension. It is a part of an ongoing research
aimed establishing properties of Menger-type curvatures, their regularizing
effects and applications in variational and geometric problems.

Our results consider two classes of sets: the class $\A(\delta,m)$ of
\emph{$(\delta,m)$-admissible} sets and the class $\F(m)$ of \emph{$m$-fine}
sets. These classes contain compact, $m$-dimensional subsets of $\Rn$ satisfying
some mild and quite general conditions (see Definition~\ref{def:adm} and
Definition~\ref{def:fine}). The definition of $\A(\delta,m)$ is more topological
and uses the notion of the \emph{linking number} while the definition of $\F(m)$
is purely metric. Examples of sets that fall into one of these classes include
e.g. compact, smooth manifolds immersed in $\Rn$ and all finite sums of such
immersions and even their bilipschitz images. For any set $\Sigma$ in one of the
classes $\A(\delta,m)$ or $\F(m)$ such that $\E_p(\Sigma)$ is finite for some $p
> m(m+2)$ we prove that $\Sigma$ is locally a graph of a $C^{1,\alpha}$ function
with $\alpha = 1 - \frac{m(m+2)}{p}$. Our first meaningful result is
\begin{introthm}[cf. Theorem~\ref{thm:uahlreg}]
  \label{intro:thm:uahlreg}
  Let $E < \infty$ be some positive constant and let $\Sigma \in \A(\delta,m)$
  be an admissible set, such that $\E_p(\Sigma) \le E$ for some $p > m(m+2)$.
  There exist a radius $R = R(E,m,p,\delta)$, such that for each $\rho \le R$
  and each $x \in \Sigma$ we have
  \begin{displaymath}
    \HM^m(\Sigma \cap \Ball(x,\rho)) \ge (1 - \delta^2)^{\frac m2} \omega_m \rho^m \,.
  \end{displaymath}
\end{introthm}

The backbone of the proof of Theorem~\ref{intro:thm:uahlreg} is
Proposition~\ref{prop:big-proj-fat-simp}, which states that at almost every
point $x \in \Sigma$ and for all radii $r > 0$ less then some positive stopping
distance $d(x)$, one can find an $m$-plane $H$ such that the projection of
$\Sigma \cap \Ball(x,r)$ onto $x + H$ contains the ball
$\Ball(x,\sqrt{1-\delta^2}r) \cap (x + H)$. It also ensures the existence of a
''quite regular'' (see Definition~\ref{def:regsimp}) simplex with $x$ as one of
its vertices and dimensions comparable to $d(x)$. The proof of
Proposition~\ref{prop:big-proj-fat-simp} is based on an algorithmic procedure
similar to that presented in~\cite{0911.2095} but is more general and
simpler. It catches the essential difficulty encountered by Strzelecki and von
der Mosel and deals with it considering only two cases instead of their
five. The essence of this algorithm can be summarized as follows. We look at
$\Sigma$ in increasingly larger scales. If $\Sigma$ is almost flat at some
scale, then we have to increase the scale. Otherwise, we find a point $y \in
\Sigma$ which is far from some affine $m$-plane spanned by $m+1$ points of
$\Sigma$ and this way we construct a ''quite regular'' simplex.

Next we show that any $(\delta,m)$-admissible set $\Sigma$ with finite
$p$-energy is also $m$-fine (cf. Theorem~\ref{thm:adm-fine}). The proof is
rather technical. It uses the following
\begin{introprop}[cf. Corollary~\ref{cor:beta-est}]
  \label{intro:prop:beta-est}
  Let $\Sigma \subseteq \Rn$ be some $m$-Ahlfors regular set such that
  $\E_p(\Sigma)$ is finite for some $p > m(m+2)$. Then there exist constants $C
  > 0$ and $\tau \in (0,1)$ such that for any $x \in \Sigma$ and any $r > 0$
  small enough we have
  \begin{displaymath}
    \beta(x,r) \le C r^{\tau} \,,
  \end{displaymath}
  where $\beta(x,r)$ denote the P. Jones' $\beta$-numbers of $\Sigma$.
\end{introprop}
This proposition plays a key role in \S\ref{sec:tangent-planes} where we
establish the following
\begin{introthm}[cf. Theorem~\ref{thm:C1tau}]
  \label{intro:thm:C1tau}
  Let $\Sigma \in \F(m)$ be an $m$-fine set such that $\E_p(\Sigma) \le E <
  \infty$ for some $p > m(m+2)$. Then there exist constants $R > 0$ and $\tau
  \in (0,1)$ such that for each $x \in \Sigma$ the set $\Sigma \cap \Ball(x,R)$
  is a graph of some function $F_x \in C^{1,\tau}(T_x\Sigma,
  T_x\Sigma^{\perp})$. Moreover the radius $R$ and the H{\"o}lder norm of $DF_x$
  depend only on $E$, $m$ and $p$.
\end{introthm}
The proof employs a technique similar to the one used by David, Kenig and Toro
in the proof of~\cite[Proposition~9.1]{MR1808649}. It is technical but with the
Proposition~\ref{intro:prop:beta-est} it becomes rather straightforward. Bounds
on the $\beta$-numbers together with the properties of $m$-fine sets imply that
$\Sigma$ is Reifenberg flat with vanishing constant (see
Definition~\ref{def:rfvc}) and let us prove $C^{1,\tau}$ regularity. Our proof
is independent of the result by David, Kenig and Toro~\cite{MR1808649} and the
outcome is slightly stronger. We show that the scale $R$ and the H{\"o}lder norm
of $DF_x$ do not depend on $\Sigma$ but only on the energy bound $E$. We believe
that this will be crucial when we apply our results in variational problems.

It is worth mentioning that our technique does not use any concept of a
\emph{trapping box} which was introduced in~\cite[\S5.1]{1102.3642}. Instead we
exploit the fact that $(\delta,m)$-admissible sets with finite $p$-energy are
$m$-fine, which gives a bound on the Reifenberg's $\theta$-numbers of $\Sigma$
(also called \emph{bilateral $\beta$-numbers}).

In \S\ref{sec:improved-holder} we improve the exponent $\tau$ to the optimal
value $\alpha = 1 - \frac{m(m+2)}{p}$. This is done employing the method
developed by Strzelecki, Szuma{\'n}ska and von der
Mosel~\cite[\S6.1]{MR2668877}. Again, we were able to simplify things a little
bit. We introduce only two sets of \emph{bad parameters} $\Sigma_0$ and
$\Sigma_1(x_0,\ldots,x_m)$ and we employ good properties of the metric on the
Grassmannian gathered in \S\ref{sec:grass}.

The proof of $C^{1,\alpha}$ regularity boils down to estimating the oscillation
of the tangent planes. The angle between two tangent planes
$\dgras(T_x\Sigma,T_y\Sigma)$ is estimated by the angle $\dgras(X,Y)$, where $X$
and $Y$ are ''secant'' $m$-planes through some appropriately chosen points in
$\Sigma$. First we choose a very big natural number $N \in \N$. The points
$x_0,\ldots,x_m$ and $y_0,\ldots,y_m$ of $\Sigma$ which span $X$ and $Y$
respectively are chosen so that they form almost orthogonal systems and so that
the distances from $x$ to any of $x_0,\ldots,x_m$ or from $y$ to any of
$y_0,\ldots,y_m$ is $N$ times smaller than the distance from $x$ to $y$.
Applying the fundamental theorem of calculus, we estimate the angle between
$T_x\Sigma$ and $X$ by the oscillation of the tangent planes on a set of
diameter $\frac{|x-y|}N$. The same applies to $T_y\Sigma$ and $Y$. Then using
the bound $\E_p(\Sigma) \le E$ we prove that $\dgras(X,Y) \lesssim
|x-y|^{\alpha}$. Next we use a method drawn from the theory of PDE and iterate
our estimates to show that the error made when passing from $T_x\Sigma$ to $X$
and from $T_y\Sigma$ to $Y$ is negligible.

We expect that theorems obtained here can be used in proving further results. We
plan to study other energy functionals and their relations with regularity of
compact subsets of $\Rn$. We believe that our work can also be applied in
variational problems with topological constraints. Furthermore we want to pursue
the connections of this theory with the theory of Sobolev spaces.



\mysection{Preliminaries}
\label{sec:prelim}

\mysubsection{Some notation}
Throughout this paper $m$ and $n$ are two fixed positive integers satisfying $0
< m < n$. The symbol $\Rn$ stands for the $n$-dimensional Euclidean space with
the standard scalar product. We write $\Sphere$ for the unit $(n-1)$-dimensional
sphere centered at the origin and we write $\Ball$ for the unit $n$-dimensional
open ball centered at the origin. We also use the symbols 
\begin{displaymath}
  \Sphere_r := r \Sphere\,, \quad 
  \Ball_r := r \Ball\,, \quad
  \Sphere(x,r) := x + \Sphere_r \quad
  and \quad
  \Ball(x,r) := x + \Ball_r \,.
\end{displaymath}

Let $H$ be an $m$-dimensional linear subspace of $\Rn$ and let $x_0$, \ldots,
$x_k$ be some points in $\Rn$. We use the symbol $\pi_H$ to denote the
orthogonal projection onto $H$ and $Q_H := I - \pi_H$ to denote the orthogonal
projection onto the orthogonal complement $H^{\perp}$. We write $\aff\{x_0,
\ldots, x_m \}$ for the smallest affine subspace of $\Rn$ containing points
$x_0$, \ldots, $x_m$, i.e.
\begin{displaymath}
  \aff\{x_0, \ldots, x_m \} := x_0 + \opspan\{x_1-x_0, \ldots, x_m-x_0\} \,.
\end{displaymath}
We use the notation $\simp(x_0, \ldots, x_k)$ for the convex hull of the set
$\{x_0,\ldots,x_k\}$, which in a typical case is a $k$-dimensional simplex with
vertices $x_0$, \ldots, $x_k$. The symbol $\HM^k$ stands for the $k$-dimensional
Hausdorff measure.

\begin{rem}
  We assume that every simplex $T = \simp(x_0, x_1, \ldots, x_k)$ is
  equipped with appropriate ordering of its vertices, so e.g. $T' = \simp(x_1,
  x_0, x_2, \ldots, x_k)$ is \emph{not} the same as $T$.
\end{rem}

\begin{defin}
  \label{def:face-hmin}
  Let $T = \simp(x_0, \ldots, x_k)$. We define
  \begin{itemize}
  \item $\face_i T := \simp(x_0, \ldots, \widehat{x_i}, \ldots, x_k)$ - the $i$-th face of $T$,
  \item $\height_i(T) := \dist(x_i, \aff\{ x_0, \ldots, \widehat{x_i} , \ldots, x_k \}$ - the height lowered from $x_i$,
  \item $\hmin(T) := \min\{ \height_i(T) : i = 0,1,\ldots,k \}$ - the minimal height of $T$.
  \end{itemize}
\end{defin}

In the course of the proofs we will frequently use cones and ''conical caps'' of
different sorts.
\begin{defin}
  \label{def:cones}
  We define
  \begin{itemize}
  \item the \emph{cone} with ''axis'' $H^{\perp}$ and ''angle'' $\delta$ as the set
    \begin{displaymath}
      \Cone(\delta,H) := \{ x \in \Rn : |Q_H(x)| \ge \delta |x| \} \,,
    \end{displaymath}
  \item the \emph{shell} (or the \emph{$n$-annulus}) of radii $r$ and $R$ as the
    open set
    \begin{displaymath}
      \Shell(r,R) := \Ball_R \setminus \overline{\Ball}_r \,,
    \end{displaymath}
  \item the \emph{conical cap} with ''angle'' $\delta$, ''axis'' $H^{\perp}$
    and radii $r$ and $R$ as the intersection of a cone with a shell
    \begin{displaymath}
      \Cone(\delta,H,r,R) := \Cone(\delta,H) \cap \Shell(r,R) \,.
    \end{displaymath}
  \end{itemize}
\end{defin}

\begin{rem}
  We have the identity
  \begin{displaymath}
    \Cone(\sqrt{1 - \delta^2},H^{\perp}) = \overline{\Rn \setminus \Cone(\delta,H)} \,.
  \end{displaymath}
\end{rem}

We write $G(n,m)$ to denote the Grassmann manifold of $m$-dimensional linear
subspaces of $\Rn$. Whenever we write $U \in G(n,m)$ we identify the point $U$
of the space $G(n,m)$ with the appropriate $m$-dimensional subspace of $\Rn$. In
particular any vector $u \in U$ is treated as an $n$-dimensional vector in the
ambient space $\Rn$ which happens to lie in $U \subseteq \Rn$.

All the subscripted constants $C_1$, $C_2$, \ldots, $R_1$, $R_2$, \ldots have
global meaning and we never use the same subscripted name for two different
constants. We use the notation $C = C(x,y,z)$ to denote that $C$ depends only on
the values of $x$, $y$ and $z$.


\mysubsection{Degree of a map and the linking number}
In this paragraph we briefly present known facts about the degree of a map. We
also state some simple propositions about the linking number in the setting
suitable for our purposes. These notions come from algebraic topology. As a
reference we use the book by Hirsch~\cite{MR0448362}. A clear and detailed
presentation of degree modulo $2$ can be also found in e.g. Blat's
paper~\cite{MR2537023}.

The contents of this paragraph is based on a paper by Strzelecki and von der
Mosel~\cite{1102.3642}. We list here some results from~\cite{1102.3642} which
will be needed later on.

The following fact summarizes of a few lemmas and theorems proved
in~\cite[Chapter 5, \S 1]{MR0448362}.
\begin{fact}
  Let $M$ and $N$ be compact manifolds of class $C^1$ and of the same dimension
  $k$. Assume that $N$ is connected. There exists a map
  \begin{displaymath}
    \deg_2 : C^0(M,N) \to \Z_2 := \{ 0, 1 \}
  \end{displaymath}
  such that
  \begin{enumerate}[(i)]
  \item If $\deg_2 g = 1$, then $g \in C^0(M,N)$ is surjective;
  \item If $H : M \times [0,1] \to N$ is continuous, $f(x) := H(x,0)$ and $g(x) := H(x,1)$, then 
    \begin{displaymath}
      \deg_2 f = \deg_2 g \,;
    \end{displaymath}
  \item If $f : M \to N$ is of class $C^1$ and $y \in N$ is a regular value of $f$, then
    \begin{displaymath}
      \deg_2 f = \# f^{-1}(y) \mod 2 \,.
    \end{displaymath}
  \end{enumerate}
\end{fact}

We introduce the following definition for brevity in stating
Lemmas~\ref{lem:lk-htp-inv}-\ref{lem:intpoint}. We shall use it only in this
paragraph.
\begin{defin}
  Let $I$ be any countable set of indices. We say that $\Sigma \subseteq \Rn$ is
  a \emph{good set} if there exist $m$-dimensional manifolds $M_i$ of class
  $C^1$ and continuous maps $f_i \in C^0(M_i,\Rn)$, such that
  \begin{displaymath}
    \Sigma = \bigcup_{i \in I} f_i(M_i) \cup Z \,,
  \end{displaymath}
  where $\HM^m(Z) = 0$.
\end{defin}

Now we can define the linking number modulo $2$ in the setting appropriate for
our needs.
\begin{defin}
  \label{def:lk2-single}
  Let $M$ and $N$ be compact manifolds of class $C^1$ of dimension $m$ and
  $n-m-1$ respectively. Assume $N$ is embedded in $\Rn$ and assume we have a
  continuous mapping $f : M \to \Rn$ such that $(\opim f) \cap N = \emptyset$. We
  define the following function
  \begin{align*}
    F : M \times N &\to \Sphere^{n-1} \,,\\
    F(w,z) &:= \frac{f(w) - z}{|f(w) - z|} \,,
  \end{align*}
  and set
  \begin{displaymath}
    \oplk(f,N) := \deg_2 F \,.
  \end{displaymath}
\end{defin}
In our applications $N$ will usually be a true round sphere.

\begin{defin}
  \label{def:lk2-good}
  Let $\Sigma \subseteq \Rn$ be a good set and let $N \subseteq \Rn$ be a
  compact manifold of class $C^1$ of dimension $n-m-1$. Assume that $\Sigma \cap
  N = \emptyset$. For each $i \in I$ we define
  \begin{align*}
    F_i : M_i \times N &\to \Sphere^{n-1} \,,\\
    F_i(w,z) &:= \frac{f_i(w) - z}{|f_i(w) - z|} \,,
  \end{align*}
  and we set
  \begin{displaymath}
    \oplk(\Sigma,N) :=
    \left\{
      \begin{array}{ll}
        1 & \text{if there exists an $i \in I$ such that } \deg_2(F_i) = 1 \,, \\
        0 & \text{otherwise} \,.
      \end{array}
    \right.
  \end{displaymath}
  We say that \emph{$\Sigma$ is linked with $N$} if $\oplk(\Sigma,N) = 1$.
\end{defin}

\begin{lem}[\!\!\cite{1102.3642}, Lemma 3.2]
  \label{lem:lk-htp-inv}
  Let $A \subseteq \Rn$ be a good set and let $N$ be a compact, closed
  $(n-m-1)$-dimensional manifold of class $C^1$, and let $N_j = h_j(N)$ for
  $j=0,1$, where $h_j$ is a $C^1$ embedding of $N$ into $\Rn$ such that $N_j
  \cap \Sigma = \emptyset$. If there is a homotopy
  \begin{displaymath}
    G : N \times [0,1] \to \Rn \setminus \Sigma \,,
  \end{displaymath}
  such that $G(-,0) = h_0$ and $G(-,1) = h_1$, then 
  \begin{displaymath}
    \oplk(\Sigma,N_0) = \oplk(\Sigma,N_1) \,.
  \end{displaymath}
\end{lem}

\begin{lem}[\!\!\cite{1102.3642}, Lemma 3.4]
  \label{lem:lk-far}
  Let $\Sigma \subseteq \Rn$ be a good set. Chose $y \in \Rn$ and $\varepsilon
  \in \R$ such that $0 < \varepsilon < r < 2 \varepsilon$ and $\dist(y,\Sigma)
  \ge 3 \varepsilon$. Then
  \begin{displaymath}
    \oplk(\Sigma,\Sphere(y,r) \cap (y+V)) = 0
  \end{displaymath}
  for each $V \in G(n,n-m)$.
\end{lem}

\begin{lem}[\!\!\cite{1102.3642}, Lemma 3.5]
  \label{lem:intpoint}
  Let $\Sigma \subseteq \Rn$ be a good set. Assume that for some $y \in \Rn$, $r
  > 0$ and $V \in G(n,n-m)$ we have
  \begin{displaymath}
    \oplk(\Sigma,\Sphere(y,r) \cap (y+V)) = 1 \,.
  \end{displaymath}
  Then the disk $\Ball(y,r) \cap (y+V)$ contains at least one point of $\Sigma$.
\end{lem}


\mysubsection{The Grassmannian as a metric space}
\label{sec:grass}

In this paragraph we gather some facts about the metric $\dgras$ on the
Grassmannian. These facts can be summarized as follows: having two linear
subspaces $U = \opspan\{ u_1, \ldots, u_m \}$ and $V = \opspan\{ v_1, \ldots,
v_m \}$ in $\Rn$ such that the bases $(u_1,\ldots,u_m)$ and $(v_1,\ldots,v_m)$
are roughly orthonormal and such that $|u_i - v_i| \le \varepsilon$, we derive
the estimate $\dgras(U,V) \lesssim \varepsilon$. This will become especially
useful in \S\ref{sec:improved-holder}.

Recall that the symbol $G(n,m)$ stands for the Grassmann manifold of
$m$-dimensional linear subspaces of $\Rn$. Formally, $G(n,m)$ is defined as the
homogeneous space
\begin{displaymath}
  G(n,m) := O(n) / (O(m) \times O(n-m)) \,,
\end{displaymath}
where $O(n)$ is the orthogonal group; see e.g. Hatcher's book~\cite[\S4.2,
Examples 4.53, 4.54 and 4.55]{MR1867354} for the reference. We treat $G(n,m)$ as
a topological space with the standard quotient topology.
\begin{defin}
  Let $U,V \in G(n,m)$. We introduce the following function on $G(n,m)$
  \begin{displaymath}
    \dgras(U,V) := \| \pi_U - \pi_V \| = \sup_{w \in \Sphere} | \pi_U(w) - \pi_V(w) | \,.
  \end{displaymath}
\end{defin}

\begin{rem}
  Let $I : \Rn \to \Rn$ denote the identity mapping. Note that
  \begin{displaymath}
    \dgras(U,V) = \| \pi_U - \pi_V \| = \| I - Q_U - (I - Q_V) \| = \| Q_V - Q_U \| \,.
  \end{displaymath}
\end{rem}

\begin{rem}
  \label{rem:orth-angle}
  If $\dgras(U,V) < 1$ then $U^{\perp} \cap V = \{ 0 \}$ and $U \cap V^{\perp} =
  \{ 0 \}$. Indeed if there is a unit vector $v \in U^{\perp} \cap V$, then
  $|\pi_U(v) - \pi_V(v)| = |\pi_V(v)| = |v| = 1$, so $\dgras(U,V) \ge 1$. In
  particular, if $\dgras(U,V) < 1$ then both mappings $\pi_U|_V : V \to U$ and
  $Q_U|_{V^{\perp}} : V^{\perp} \to U^{\perp}$ are linear isomorphisms.
  Therefore we can define the inverse mappings
  \begin{displaymath}
    L_U := (\pi_U|_V)^{-1} : U \to V
    \qquad \text{and} \qquad
    K_U := (Q_U|_{V^{\perp}})^{-1} : U^{\perp} \to V^{\perp} \,.
  \end{displaymath}

  To be precise, we treat $U$, $U^{\perp}$, $V$ and $V^{\perp}$ as subsets of
  $\Rn$, so the domains of $L_U$ and $K_U$ contain those $n$-dimensional vectors
  which lie in $U \subseteq \Rn$ and $U^{\perp} \subseteq \Rn$
  respectively. Also the values $L_U(u)$ and $K_U(u)$ are $n$-dimensional. Let
  $I : \Rn \to \Rn$ be the identity. It makes sense to define the mapping $P :=
  L_U - I$, which maps $U \subseteq \Rn$ to $U^{\perp} \subseteq \Rn$. This
  will be used in \S\ref{sec:tangent-planes} where we construct a
  parameterization for $\Sigma$.
\end{rem}

\begin{fact}
  The function $\dgras$ defines a metric on the Grassmannian $G(n,m)$ and the
  topology induced by this metric agrees with the standard quotient topology
  (cf. Remark~\ref{rem:alt-grass-metric}).
\end{fact}

\begin{fact}
  \label{fact:ang-dist}
  We have
  \begin{align*}
    \forall v \in V \quad
    |Q_U(v)| &= \dist(v,U) \le |v| \dgras(V,U) \\
    \text{and} \quad
    \forall v \in V^{\perp} \quad
    |\pi_U(v)| &= \dist(v,U^{\perp}) \le |v| \dgras(V,U) \,.
  \end{align*}
\end{fact}

\begin{proof}
  For $v \in V$ a straightforward calculation gives
  \begin{displaymath}
    |v| \dgras(V,U) = |v| \| Q_V - Q_U \| \ge | Q_V(v) - Q_U(v) | = |Q_U(v)| \,.
  \end{displaymath}
  If $v \in V^{\perp}$ then
  \begin{displaymath}
    |v| \dgras(V,U) = |v| \| \pi_V - \pi_U \| \ge | \pi_V(v) - \pi_U(v) | = |\pi_U(v)| \,.
  \end{displaymath}
\end{proof}

\begin{cor}
  if $\dgras(U,V) \le \alpha < 1$, then for all $v \in V$ we have $(1 - \alpha)
  |v| \le |\pi_U(v)| \le \alpha |v|$. Analogous estimate holds also for $v \in
  V^{\perp}$ and $Q_U(v)$, hence
  \begin{displaymath}
    \| L_U \|_U \le \frac{1}{1 - \alpha}
    \qquad \text{and} \qquad
    \| K_U \|_{U^{\perp}} \le \frac{1}{1 - \alpha} \,.
  \end{displaymath}
\end{cor}

\begin{prop}
  \label{prop:close-bases}
  If $U,V \in G(n,m)$ have orthonormal bases $(e_1,\ldots,e_m)$ and
  $(f_1,\ldots,f_m)$ respectively and if $|e_i - f_i| \le \vartheta$ for each $i
  = 1,\ldots,m$, then $\dgras(U,V) \le 2m\vartheta$.
\end{prop}

\begin{proof}
  Let $w \in \Sphere$ be a unit vector in $\Rn$. We calculate
  \begin{align*}
    |\pi_U(w) - \pi_V(w)| 
    &= \left| \sum_{i=1}^m \langle w, e_i \rangle e_i - \langle w, f_i \rangle f_i \right| \\
    &= \left| \sum_{i=1}^m \langle w, e_i \rangle (e_i - f_i) + \langle w, (e_i - f_i) \rangle f_i \right| \\
    &\le \sum_{i=1}^m |e_i - f_i| + |e_i - f_i| \le 2 m \vartheta \,.
   \end{align*}
\end{proof}

\begin{defin}
  Let $V \in G(n,m)$ and let $(v_1,\ldots,v_m)$ be the basis of $V$. Fix some
  radius $\rho > 0$ and two small constants $\varepsilon \in (0,1)$ and $\delta
  \in (0,1)$.
  \begin{itemize}
  \item We say that $(v_1,\ldots,v_n)$ is a \emph{$\red$-basis} with constants
    $\rho$, $\varepsilon$ and $\delta$ if the following conditions are satisfied
    \begin{align*}
      (1-\varepsilon) \rho \le |v_i| &\le (1+\varepsilon) \rho \quad \text{for } i = 1, \ldots, m \\
      \text{and} \quad
      |\langle v_i, v_j \rangle| &\le \delta \rho^2 \quad \text{for } i \ne j \,.
    \end{align*}

  \item We say that $(v_1,\ldots,v_n)$ is an \emph{ortho-$\rho$-normal basis} if
    \begin{align*}
      |v_i| &= \rho \quad \text{for } i = 1, \ldots, m \\
      \text{and} \quad
      \langle v_i, v_j \rangle &= 0 \quad \text{for } i \ne j \,.
    \end{align*}
  \end{itemize}
\end{defin}

\begin{defin}
  \label{def:GS-process}
  Let $(v_1, \ldots, v_m)$ be an ordered basis of some $m$-plane $H \in G(n,m)$.
  \begin{itemize}
  \item We say that \emph{an orthonormal basis $(\hat{v}_1,\ldots,\hat{v}_m)$
      arises from $(v_1, \ldots, v_m)$ by the Gram-Schmidt
      process}\footnote{Note that all the bases considered here are ordered and
      the result of the Gram-Schmidt process depends on that ordering.} if
    \begin{displaymath}
      \hat{v}_1 = \frac{v_1}{|v_1|} 
      \quad \text{and for $k = 2,\ldots,m$} \quad
      \hat{v}_k = \frac{w_k}{|w_k|} 
      \quad \text{where} \quad 
      w_k = v_k - \sum_{i = 1}^{k-1} \langle v_k, \hat{v}_i \rangle \hat{v}_i \,.
    \end{displaymath}

  \item We say that \emph{an ortho-$\rho$-normal basis
      $(\bar{v}_1,\ldots,\bar{v}_m)$ arises from $(v_1, \ldots, v_m)$ by the
      Gram-Schmidt process} if the orthonormal basis
    \begin{displaymath}
      (\hat{v}_1,\ldots,\hat{v}_m) := (\rho^{-1}\bar{v}_1,\ldots,\rho^{-1}\bar{v}_m)
    \end{displaymath}
    arises from $(v_1, \ldots, v_m)$ by the Gram-Schmidt process.
  \end{itemize}
\end{defin}

\begin{prop}
  \label{prop:gs-red}
  Let $\rho > 0$, $\varepsilon \in (0,1)$ and $\delta \in (0,1)$ be some
  constants. Let $(v_1,\ldots,v_m)$ be a $\red$-basis of $V \in G(n,m)$ and let
  $(\hat{v}_1,\ldots,\hat{v}_m)$ be an ortho-$\rho$-normal basis of $V$ which
  arises from $(v_1,\ldots,v_m)$ by the Gram-Schmidt process. There exist two
  constants $\Cl{gs-eps} = \Cr{gs-eps}(m)$ and $\Cl{gs-del} = \Cr{gs-del}(m)$
  such that
  \begin{align*}
    |v_i - \hat{v}_i| \le (\Cr{gs-eps} \varepsilon + \Cr{gs-del} \delta) \rho
    \quad \text{for } i = 1,\ldots,m \,.
  \end{align*}
\end{prop}

\begin{proof}
  For $i = 1,\ldots,m$ set $e_i := v_i / \rho$. Let $(f_1,\ldots,f_m)$ be an
  orthonormal basis of $V$ obtained from $(e_1,\ldots,e_m)$ by the Gram-Schmidt
  process. Note that
  \begin{displaymath}
    1 - \varepsilon \le |e_i| \le 1 + \varepsilon
    \qquad \text{and} \qquad
    |\langle e_i, e_j \rangle| \le \delta \,.
  \end{displaymath}
  
  We will show inductively that for each $i = 1,\ldots,m$ there exist constants
  $A_i$ and $B_i$ such that $|f_i - e_i| \le A_i \varepsilon + B_i \delta$. For
  the first vector we have
  \begin{displaymath}
    f_1 := \frac{e_1}{|e_1|}
    \qquad \text{hence} \qquad
    |f_1 - e_1| \le \varepsilon \,,
  \end{displaymath}
  so we can set $A_1 := 1$ and $B_1 := 0$.

  Assume we already proved that $|f_i - e_i| \le A_i \varepsilon + B_i \delta$
  for $i = 1, \ldots, k-1$. The Gram-Schmidt process gives
  \begin{displaymath}
    \tilde{f}_k = e_k - \sum_{i=1}^{k-1} \langle e_k, f_i \rangle f_i
    \qquad \text{and} \qquad
    f_k = \frac{\tilde{f_k}}{|\tilde{f_k}|} \,.
  \end{displaymath}
  Let us first estimate $|\langle e_k, f_i \rangle|$ for $i = 1, \ldots, k-1$.
  \begin{align*}
    |\langle e_k, f_i \rangle|
    &\le |\langle e_k, e_i \rangle| + |\langle e_k, (f_i - e_i) \rangle| 
    \le |\langle e_k, e_i \rangle| + |e_k| |f_i - e_i| \\
    &\le \delta + (1 + \varepsilon)(A_i \varepsilon + B_i \delta)
    \le (1 + 2B_i) \delta + 2A_i \varepsilon \,.
  \end{align*}
  Here we used the fact that $\varepsilon,\delta \in (0,1)$, so $\varepsilon
  \delta \le \delta$ and $\varepsilon^2 \le \varepsilon$. Set $\tilde{A}_k := 2
  \sum_{i=1}^{k-1} A_i$ and $\tilde{B}_k := \sum_{i=1}^{k-1} (1 + 2B_i)$. We
  then have
  \begin{displaymath}
    \left| \sum_{i=1}^{k-1} \langle e_k, f_i \rangle f_i \right|
    \le \sum_{i=1}^{k-1} |\langle e_k, f_i \rangle|
    \le \tilde{A}_k \varepsilon + \tilde{B}_k \delta \,.
  \end{displaymath}
  Hence
  \begin{displaymath}
    |\tilde{f}_k| 
    \ge |e_k| - \left| \sum_{i=1}^{k-1} \langle e_k, f_i \rangle f_i \right|
    \ge 1 - (\varepsilon + \tilde{A}_k \varepsilon + \tilde{B}_k \delta) 
  \end{displaymath}
  and
  \begin{align*}
    |e_k - f_k| 
    &\le |e_k - \tilde{f}_k| + |\tilde{f}_k - f_k| \\
    &\le  \tilde{A}_k \varepsilon + \tilde{B}_k \delta + \varepsilon + \tilde{A}_k \varepsilon + \tilde{B}_k \delta 
    = (1 + 2 \tilde{A}_k) \varepsilon + 2\tilde{B}_k \delta 
  \end{align*}
  This gives
  \begin{displaymath}
    A_k := 1 + 2 \tilde{A}_k = 1+  4 \sum_{i=1}^{k-1} A_i
    \quad \text{and} \quad
    B_k := 2 \tilde{B}_k = 2(k-1) + 4 \sum_{i=1}^{k-1} B_i \,.
  \end{displaymath}
  Since the sequences $A_k$ and $B_k$ are increasing we may set $\Cr{gs-eps} :=
  A_m$ and $\Cr{gs-del} := B_m$. Recall that $v_i := \rho e_i$ and $\hat{v}_i :=
  \rho f_i$, so
  \begin{displaymath}
    |v_i - \hat{v}_i| = \rho |e_i - f_i| \le (\Cr{gs-eps} \varepsilon + \Cr{gs-del} \delta) \rho \,.
  \end{displaymath}
  for each $i = 1,\ldots,m$.
\end{proof}

\begin{prop}
  \label{prop:dist-ang}
  Let $U,V \in G(n,m)$ and let $(e_1,\ldots,e_m)$ be some orthonormal basis of
  $V$. Assume that for each $i = 1,\ldots,m$ we have the estimate $\dist(e_i,U)
  = |Q_U(e_i)| \le \vartheta$ for some $\vartheta \in (0,1)$. Then there exists
  a constant $\Cl{dist-ang} = \Cr{dist-ang}(m)$ such that
  \begin{displaymath}
    \dgras(U,V) \le \Cr{dist-ang} \vartheta \,.
  \end{displaymath}
\end{prop}

\begin{proof}
  Set $u_i := \pi_U(e_i)$. For each $i = 1, \ldots, m$ we have $|Q_U(e_i)| \le
  \vartheta$, so
  \begin{align}
    |u_i - e_i| &= |Q_U(e_i)| \le \vartheta
    \quad \text{hence} \notag \\
    \label{est:ui-len}
    1 - \vartheta^2 \le \sqrt{1 - \vartheta^2} &\le |u_i| \le 1 \le 1 + \vartheta^2
    \quad \text{for } i = 1, \ldots, m \,.
  \end{align}
  For any $i \ne j$ the vectors $e_i$ and $e_j$ are orthogonal, hence
  \begin{align*}
    0 = \langle e_i, e_j \rangle 
    &= \langle \pi_U(e_i) + Q_U(e_i), \pi_U(e_j) + Q_U(e_j) \rangle \\
    &= \langle \pi_U(e_i), \pi_U(e_j) \rangle + \langle Q_U(e_i), Q_U(e_j) \rangle \,.
  \end{align*}
  Therefore
  \begin{equation}
    \label{est:uiuj-ang}
    |\langle u_i, u_j \rangle| 
    = | \langle Q_U(e_i), Q_U(e_j) \rangle | 
    \le |Q_U(e_i)| |Q_U(e_j)| \le \vartheta^2 \,.
  \end{equation}
  
  Estimates \eqref{est:ui-len} and \eqref{est:uiuj-ang} show that
  $(u_1,\ldots,u_m)$ is a $\red$-basis of $U$ with constants $\rho = 1$,
  $\varepsilon = \vartheta^2$ and $\delta = \vartheta^2$. Let $(f_1,\ldots,f_m)$
  be the orthonormal basis of $U$ arising from $(u_1,\ldots,u_m)$ by the
  Gram-Schmidt process. Applying Proposition~\ref{prop:gs-red} we obtain
  \begin{displaymath}
    |f_i - e_i| \le |f_i - u_i| + |u_i - e_i|
    \le (\Cr{gs-eps} + \Cr{gs-del}) \vartheta^2 + \vartheta \,.
  \end{displaymath}
  Using Proposition~\ref{prop:close-bases} and the fact that $\vartheta^2 <
  \vartheta < 1$ we finally get
  \begin{displaymath}
    \dgras(U,V) \le 2 m ((\Cr{gs-eps} + \Cr{gs-del}) \vartheta^2 + \vartheta) 
    \le 2 m (\Cr{gs-eps} + \Cr{gs-del} + 1) \vartheta \,.
  \end{displaymath}
  Now we can set $\Cr{dist-ang} = \Cr{dist-ang}(m) := 2 m (\Cr{gs-eps}(m) +
  \Cr{gs-del}(m) + 1)$.
\end{proof}

\begin{prop}
  \label{prop:red-ang}
  Let $(v_1,\ldots,v_m)$ be a $\red$-basis of $V \in G(n,m)$ with constants
  $\rho > 0$, $\varepsilon \in (0,1)$ and $\delta \in (0,1)$. Let
  $(u_1,\ldots,u_m)$ be some basis of $U \in G(n,m)$, such that $|u_i - v_i| \le
  \vartheta \rho$ for some $\vartheta \in (0,1)$ and for each $i =
  1,\ldots,m$. Furthermore, let us assume that
  \begin{equation}
    \label{cond:eps-del}
    \Cr{dist-ang} (\Cr{gs-eps} \varepsilon + \Cr{gs-del} \delta) < 1 \,.
  \end{equation}
  Then there exists a constant $\Cl{red-ang} =
  \Cr{red-ang}(m,\varepsilon,\delta)$ such that
  \begin{displaymath}
    \dgras(U,V) \le \Cr{red-ang} \vartheta \,.
  \end{displaymath}
\end{prop}

\begin{proof}
  Set $e_i := v_i / \rho$ and let $(\hat{e}_1,\ldots,\hat{e}_m)$ be the
  orthonormal basis of $V$ arising from $(e_1,\ldots,e_m)$ by the Gram-Schmidt
  process. Set $f_i := u_i / \rho$.
  \begin{align*}
    |Q_U (\hat{e}_i)| &\le |Q_U (\hat{e}_i - e_i)| + |Q_U (e_i)|
    \le |\hat{e}_i - e_i| \dgras(U,V) + |e_i - f_i| \\
    &\le |\hat{e}_i - e_i| \dgras(U,V) + \vartheta \,.
  \end{align*}
  From Proposition~\ref{prop:gs-red} we have $|\hat{e}_i - e_i| \le \Cr{gs-eps}
  \varepsilon + \Cr{gs-del} \delta$, so 
  \begin{displaymath}
    |Q_U (\hat{e}_i)| 
    \le (\Cr{gs-eps} \varepsilon + \Cr{gs-del} \delta) \dgras(U,V) + \vartheta \,.
  \end{displaymath}
  Applying Proposition~\ref{prop:dist-ang} we obtain
  \begin{align*}
    \dgras(U,V) &\le
    \Cr{dist-ang} (\Cr{gs-eps} \varepsilon + \Cr{gs-del} \delta) \dgras(U,V) + \Cr{dist-ang} \vartheta
    \quad \text{hence} \\
    (1 - \Cr{dist-ang} (\Cr{gs-eps} \varepsilon + \Cr{gs-del} \delta)) \dgras(U,V) &\le \Cr{dist-ang} \vartheta \,.
  \end{align*}
  Since we assumed \eqref{cond:eps-del} we can divide both sides by $1 -
  \Cr{dist-ang} (\Cr{gs-eps} \varepsilon + \Cr{gs-del} \delta)$ reaching the
  estimate
  \begin{displaymath}
    \dgras(U,V) \le \frac{\Cr{dist-ang}}{1 - \Cr{dist-ang} (\Cr{gs-eps} \varepsilon + \Cr{gs-del} \delta)} \vartheta \,.
  \end{displaymath}
  Finally we set
  \begin{displaymath}
    \Cr{red-ang} = \Cr{red-ang}(m,\varepsilon,\delta) 
    := \frac{\Cr{dist-ang}(m)}{1 - \Cr{dist-ang}(m) (\Cr{gs-eps}(m) \varepsilon + \Cr{gs-del}(m) \delta)} \,.
  \end{displaymath}
\end{proof}

\begin{rem}
  \label{rem:alt-grass-metric}
  Propositions~\ref{prop:close-bases} and~\ref{prop:dist-ang} show that the
  metric on $G(n,m)$ given by
  \begin{displaymath}
    \mathfrak{d}(U,V) := \inf \left\{
      \left( \sum_{i=1}^m |v_i - u_i|^2 \right)^{\frac 12}
      :
      \begin{array}{l}
        (v_1,\ldots,v_m) \text{ an orthonormal basis of } V\,, \\
        (u_1,\ldots,u_m) \text{ an orthonormal basis of } U
      \end{array}
    \right\}
  \end{displaymath}
  is equivalent to the metric $\dgras$.
\end{rem}


\mysubsection{Properties of cones}

\mysubsubsection{Homotopies inside cones}

In this section we prove two facts which will allow us to construct complicated
deformations of spheres in Section~\ref{sec:ahl-reg}. In the proof of
Proposition~\ref{prop:big-proj-fat-simp} we construct a set $F$ by glueing
conical caps together. Then we need to know that we can deform one sphere lying
in $F$ to some other sphere lying in $F$ without leaving $F$. To be able to do
this easily we need Proposition~\ref{prop:sph-trans} and
Corollary~\ref{cor:sph-in-cone} stated below.

\begin{defin}
  \label{def:grass-cone}
  Let $H \in G(n,m)$ be an $m$-dimensional subspace of $\Rn$ and let $\delta \in
  (0,1)$ be some number. We define the set
  \begin{displaymath}
    \G(\delta,H) :=
    \{
    V \in G(n,n-m) : \forall v \in V \,\, |Q_H(v)| \ge \delta |v|
    \} \,.
  \end{displaymath}
\end{defin}

In other words $V \in \G(\delta,H)$ if and only if $V$ is contained in the cone
$C(\delta,H)$ (cf. Definition~\ref{def:cones}). If $n = 3$ and $m = 1$ then $H$
is a line in $\R^3$ and the cone $C(\delta,H)$ contains all the $2$-dimensional
planes $V$ such that $\sin(\opang(H,V)) \ge \delta$.

\begin{prop}
  \label{prop:grass-path}
  For any two spaces $U$ and $V$ in $\G(\delta,H)$ there exists a continuous
  path $\gamma : [0,1] \to \G(\delta,H)$ such that $\gamma(0) = V$ and
  $\gamma(1) = U$.
\end{prop}

\begin{cor}
  \label{cor:path-lift}
  The path $\gamma$ from Proposition~\ref{prop:grass-path} lifts to a continuous
  path $\tilde{\gamma} : [0,1] \to O(n)$ in the orthogonal group.
\end{cor}

In the proof of Proposition~\ref{prop:grass-path} we actually construct pieces
of the path $\gamma$ in the orthogonal group $O(n)$ and then we compose such a
piece with the projection onto the Grassmannian. The problem with lifting such a
path occurs when we want to glue separate pieces together. We bypass this
problem using some abstract topological tools in the proof below. With some
effort one could construct the path $\tilde{\gamma}$ by hand, e.g. using the
fact that $SO(n)$ is path-connected and that any orthonormal base of $\Rn$ can
be easily modified to define an element of $SO(n)$ just by multiplying one
vector by $-1$. To keep the proof of Proposition~\ref{prop:grass-path}
relatively simple, we chose to employ some properties of fiber bundles.

\begin{proof}
  We consider the fiber bundles (see~\cite[Examples 4.53 and 4.54]{MR1867354})
  \begin{align*}
    O(n-m) &\to V(n,n-m) \to G(n,n-m) \\
    \text{and} \quad O(m) &\to O(n) \to V(n,n-m) \,,
  \end{align*}
  where $V(n,n-m) = O(n)/O(m)$ is the Stiefel manifold of orthonormal frames of
  $n-m$ vectors in $\Rn$ considered as a subspace of a product of $n-m$
  spheres. According to \cite[Proposition~4.48]{MR1867354}, these bundles have
  the homotopy lifting property with respect to any CW pair $(X,A)$. Let us take
  $X = A = \{ \star \}$. The homotopy we want to lift is
  \begin{align*}
    F : \{\star\} \times [0,1] &\to G(n,n-m) \\
    (\star,t) &\mapsto \gamma(t) \,.
  \end{align*}
  All we need to do is to choose a starting point $\tilde{F}(\star,0) \in
  V(n,n-m)$, which boils down to choosing an orthonormal basis of $\gamma(0) \in
  G(n,n-m)$. Using the homotopy lifting property we get a map
  \begin{displaymath}
    \tilde{F} : \{\star\} \times [0,1] \to V(n,n-m) \,.
  \end{displaymath}
  Now we use the homotopy lifting property once again for the second fiber
  bundle. For the starting point $\tilde{\tilde{F}}(\star,0)$ we need to
  complete the basis $\tilde{F}(\star,0)$ to some orthonormal basis of $\Rn$ but
  we can always do that. Finally we set $\tilde{\gamma}(t) =
  \tilde{\tilde{F}}(\star,t)$.
\end{proof}

\begin{proof}[Proof of Proposition~\ref{prop:grass-path}]
  Fix some $V \in \G(\delta,H)$. It suffices to show that we can continuously
  deform $V$ to the space $H^{\perp}$ inside $\G(\delta,H)$. Then, for any other
  space $U \in \G(\delta,H)$ we can find a second path joining $U$ with
  $H^{\perp}$ and combine these two path to make a path from $V$ to $U$.

  We will construct a finite sequence of paths $\gamma_1$, \ldots,
  $\gamma_{N-1}$ in the Grassmannian $G(n,m)$ and a finite sequence of
  $m$-planes $V =: V_1$, $V_2$, \ldots, $V_N := H^{\perp}$. For each $i =
  1,\ldots,N-1$ the path $\gamma_i$ will join $V_i$ with $V_{i+1}$ and the
  intersection $V_{i+1} \cap H^{\perp}$ will have strictly bigger dimension then
  $V_i \cap H^{\perp}$. For fixed $i$ we shall first construct a path
  $\tilde{\gamma}_i$ in the orthogonal group $O(n)$ and then we shall set
  $\gamma_i = \tilde{\gamma}_i \circ \text{pr}$, where $\text{pr} : O(n) \to
  G(n,n-m)$ is the standard projection mapping. To construct the path
  $\tilde{\gamma}_i$ we find a continuous family of rotations (i.e. elements of
  $O(n)$) which act on the space
  \begin{displaymath}
    X_i := (V_i \cap H^{\perp})^{\perp} \,,
  \end{displaymath}
  stabilizing the orthogonal complement $X_i^{\perp} = V_i \cap H^{\perp}$.
  This way we know, that along the path $\gamma_i$ we never decrease the
  dimension of the space $\gamma_i(t) \cap H^{\perp}$. In other words, once we
  make $V_i$ intersect $H^{\perp}$ on some subspace, we do all the consecutive
  rotations in the orthogonal complement of that subspace, so along the way, we
  can only increase the dimension of the intersection with $H^{\perp}$.

  Set 
  \begin{align*}
    V_1 &:= V \,,&
    X_1 &:= (V_1 \cap H^{\perp})^{\perp} \,,&
    \bar{V}_1 &:= V_1 \cap X_1 \,,&
    &\text{and}&
    H_1^{\perp} &:= H^{\perp} \cap X_1 \,.    
  \end{align*}
  Note that $\bar{V}_1 \cap H_1^{\perp} = \{ 0 \}$ and that $\dim H_1^{\perp} =
  \dim \bar{V}_1$. Choose a vector $v_1 \in \bar{V}_1 \cap \Sphere$ such that
  \begin{displaymath}
    | Q_H(v_1) | = \max_{v \in \bar{V}_1 \cap \Sphere} | Q_H(v) | \,.
  \end{displaymath}
  This condition says that $v_1 \in \bar{V}_1$ is a unit vector which makes the
  smallest angle with $H_1^{\perp}$. Set $h_1 := Q_H(v_1) \in H_1^{\perp}$ and
  set $P := \opspan\{v_1, h_1\}$. Note that $|h_1| < 1$, because we restricted
  ourselfs to the space $X_1$ in which $\bar{V}_1 \cap H_1^{\perp} = \{ 0 \}$. We
  will make the rotation in the plane $P$.

  Set 
  \begin{displaymath}
    u_1 := \frac{h_1 - \langle h_1, v_1 \rangle v_1}{|h_1 - \langle h_1, v_1 \rangle v_1|} \,,
  \end{displaymath}
  so that $\{v_1, u_1\}$ makes an orthonormal basis of $P$. Choose an orthonormal
  basis of $P^{\perp}$ consisting of vectors $v_2$, \ldots, $v_{n-m}$ and $u_2$,
  \ldots, $u_m$ such that
  \begin{align*}
    V_1 &= \opspan \{ v_1, \ldots, v_{n-m} \}  \,, \\
    V_1^{\perp} &= \opspan\{ u_1, \ldots, u_m \} \,.
  \end{align*}
  For any angle $\alpha$ we define the rotation $R_{\alpha} : \Rn \to \Rn$ with
  the formula
  \begin{displaymath}
    R_{\alpha}(z) :=
    \langle z, v_1 \rangle (v_1 \cos \alpha + u_1 \sin \alpha) +
    \langle z, u_1 \rangle (u_1 \cos \alpha - v_1 \sin \alpha) \,.
  \end{displaymath}
  Set $\alpha := \opang(v_1,h_1)$ and define a path $\tilde{\gamma}_1 : [0,1]
  \to O(n)$ in the orthogonal group
  \begin{displaymath}
    \tilde{\gamma}_1(t) := 
    ( R_{t\alpha}(v_1), v_2, \ldots, v_{n-m},
    R_{t\alpha}(u_1), u_2, \ldots, u_m ) \,.
  \end{displaymath}
  Let $\text{pr} : O(n) \to O(n) / (O(n-m) \times O(m)) = G(n,n-m)$ denote the
  standard projection mapping and set $\gamma_1 := \text{pr} \circ
  \tilde{\gamma}_1$. This defines a continuous path in the Grassmanian. Of
  course $\gamma_1(0) = V_1$ and $\gamma_1(1) = \opspan \{ h_1, v_2, \ldots,
  v_{n-m} \}$ which intersects $H^{\perp}$ along $V_1 \cap H^{\perp}$ but also
  along the direction $h_1 \notin V_1 \cap H^{\perp}$.

  Now we set 
  \begin{align*}
    V_2 &:=  \gamma_1(1) \,,&
    X_2 &:= (V_2 \cap H^{\perp})^{\perp} \,,&
    \bar{V}_2 &:= V_2 \cap X_2 \,,&
    &\text{and}&
    H_2^{\perp} &:= H^{\perp} \cap X_2 \,.
  \end{align*}
  If $V_2 \ne H^{\perp}$, we can repeat the whole procedure finding another path
  $\gamma_2$ which joins $V_2$ with some $(n-m)$-plane $V_3 := \gamma_2(1)$
  which intersects $H^{\perp}$ on a subspace of dimension at least $\dim(V_2
  \cap H^{perp}) + 1$.

  Since the dimension of $V_i \cap H^{\perp}$ increases in each step and $\dim
  H^{\perp} = n-m$, after $N \le n-m$ steps we shall have $V_N = H^{\perp}$.
  Glueing consecutive paths $\gamma_j$ together, we construct a path $\gamma$
  between $V$ and $H^{\perp}$ inside $G(n,n-m)$.

  What is left to show, is that for each $t \in [0,1]$ the space $\gamma(t)$ is
  really a member of $\G(\delta,H)$ (i.e. $\gamma(t)$ is contained in the cone
  $C(\delta,H)$). It suffices to show that for each $j$ and for each $t \in
  [0,1]$ the space $\gamma_j(t)$ belongs to $\G(\delta,H)$. We will focus on the
  case $j = 1$. For all other $j$'s the proof is identical.

  Fix some $t \in [0,1]$ and some vector $z \in V \cap \Sphere$. Note that $z_t
  := R_{t\alpha}(z)$ is a vector in $\gamma_1(t) \cap \Sphere$ and that any
  vector $\bar{w} \in \gamma_1(t) \cap \Sphere$ can be expressed as $\bar{w} =
  R_{t\alpha}(\bar{z})$ for some $\bar{z} \in V \cap \Sphere$. Hence, it
  suffices to show that $|Q_H(R_{t\alpha}(z))| \ge \delta$. Set $z_i := \langle
  z, v_i \rangle$ so that
  \begin{displaymath}
    z = \sum_{i = 1}^{n-m} z_i v_i \,.
  \end{displaymath}
  Note that for $i > 1$ we have $v_i \perp P$ and also $R_{t\alpha}(v_i) = v_i$ so
  \begin{displaymath}
    Q_H(v_i) = Q_H(R_{t\alpha}(v_i)) 
    = \pi_{H^{\perp} \cap P}(v_i) + \pi_{H^{\perp} \cap P^{\perp}}(v_i) 
    = \pi_{H{\perp} \cap P^{\perp}}(v_i) \in P^{\perp} \,.
  \end{displaymath}
  For $i = 1$ we have $v_1 \in P$ and also $R_{t\alpha}(v_1) \in P$ so
  \begin{align*}
    Q_H(v_1) &= \pi_{H^{\perp} \cap P}(v_1) \in P \\
    \text{and} \quad
    Q_H(R_{t\alpha}(v_1)) &= \pi_{H^{\perp} \cap P}(R_{t\alpha}(v_1)) \in P \,.
  \end{align*}
  This gives us
  \begin{align*}
    Q_H(v_1) &\perp Q_H(v_i) \quad \text{for } i > 1 \\
    \text{and} \quad
    Q_H(R_{t\alpha}(v_1)) &\perp Q_H(R_{t\alpha}(v_i)) \quad \text{for } i > 1 \,.
  \end{align*}
  Hence, we have
  \begin{displaymath}
    \delta \le | Q_H(z) |^2 
    = \left| z_1 Q_H(v_1) + \sum_{i = 2}^{n-m} z_i Q_H(v_i) \right|^2 
    =  z_1^2 | Q_H(v_1) |^2 + \left| \sum_{i = 2}^{n-m} z_i Q_H(v_i) \right|^2
  \end{displaymath}
  \begin{align*}
    \text{and} \quad
    | Q_H(R_{t\alpha}(z)) |^2 
    &= \left| z_1 Q_H(R_{t\alpha}(v_1)) + \sum_{i = 2}^{n-m} z_i Q_H(v_i) \right|^2 \\
    &=  z_1^2 | Q_H(R_{t\alpha}(v_1)) |^2 + \left| \sum_{i = 2}^{n-m} z_i Q_H(v_i) \right|^2 \,,
  \end{align*}
  so it suffices to show that $| Q_H(R_{t\alpha}(v_1)) |^2 \ge | Q_H(v_1) |^2$.
  From the definition of $v_1$ and $\alpha$ we have $|Q_H(v_1)|^2 = \cos^2 \alpha$
  and from the definition of $R_{t\alpha}$ we have $|Q_H(R_{t\alpha}(v_1))|^2 =
  \cos^2 (1-t)\alpha$. In our setting $0 \le \alpha \le \frac{\pi}2$ and $t \in
  [0,1]$, so $\cos (1-t)\alpha \ge \cos \alpha$ and this completes the proof.
\end{proof}

\begin{cor}
  \label{cor:sph-in-cone}
  Let $H$ and $\delta$ be as in Proposition~\ref{prop:grass-path}. Let $S_1$ and
  $S_2$ be two round spheres centered at the origin, contained in the conical
  cap $\Cone(\delta,H,\rho_1,\rho_2)$ and of the same dimension
  $(n-m-1)$. Moreover assume that $0 \le \rho_1 < \rho_2 $.  There exists an
  isotopy
  \begin{displaymath}
    F : S_1 \times [0,1] \to \Cone(\delta,H,\rho_1,\rho_2) \,,
  \end{displaymath}
  such that
  \begin{displaymath}
    F(-,0) = \opid
    \qquad \text{and} \qquad
    \opim F|_{S_1 \times \{1\}} = S_2 \,. 
  \end{displaymath}
\end{cor}

\begin{proof}
  Let $r_1$ and $r_2$ be the radii of $S_1$ and $S_2$ respectively. We have
  $\rho_1 < r_1,r_2 < \rho_2$. Let $V_1,V_2 \in G(n,n-m)$ be the two subspaces
  of $\Rn$ such that $S_1 \subseteq V_1$ and $S_2 \subseteq V_2$. In other words
  $S_1 = \Sphere_{r_1} \cap V_1$ and $S_2 = \Sphere_{r_2} \cap V_2$. Because
  $S_1$ and $S_2$ are subsets of $\Cone(\delta,H)$, we know that $V_1$ and $V_2$
  are elements of $\G(\delta,H)$. From Proposition~\ref{prop:grass-path} we get
  a continuous path $\gamma$ joining $V_1$ with $V_2$. By
  Corollary~\ref{cor:path-lift}, this path lifts to a path $\tilde{\gamma}$ in
  the orthogonal group $O(n)$. For $z \in S_1$ and $t \in [0,1]$ we set
  \begin{align*}
    F(z,t) &:= \tilde{\gamma}(t) \tilde{\gamma}(0)^{-1} z \,.
  \end{align*}
  This gives a continuous deformation of $S_1 = \Sphere_{r_1} \cap V_1$ into
  $\Sphere_{r_1} \cap V_2$. Now, we only need to adjust the radius but this can be
  easily done inside $V_2 \cap \Shell(\rho_1,\rho_2)$ so the corollary is proved.
\end{proof}

\begin{prop}
  \label{prop:sph-trans}
  Let $H \in G(n,m)$. Let $S$ be a sphere perpendicular to $H$, meaning that $S
  = \Sphere(x,r) \cap (x + H^{\perp})$ for some $x \in H$ and $r > 0$. Assume
  that $S$ is contained in the ''conical cap'' $\Cone(\delta,H,\rho_1,\rho_2)$,
  where $\rho_2 > 0$. Fix some $\rho \in (\rho_1,\rho_2)$. There exists an
  isotopy
  \begin{displaymath}
    F : S \times [0,1] \to \Cone(\delta,H,\rho_1,\rho_2) \,,
  \end{displaymath}
  such that
  \begin{displaymath}
    F(\cdot,0) = \opid 
    \qquad \text{and} \qquad
    \opim F|_{S \times \{1\}} = \Sphere_{\rho} \cap H^{\perp} \,.
  \end{displaymath}
\end{prop}

\begin{figure}[!htb]
  \centering
  \includegraphics{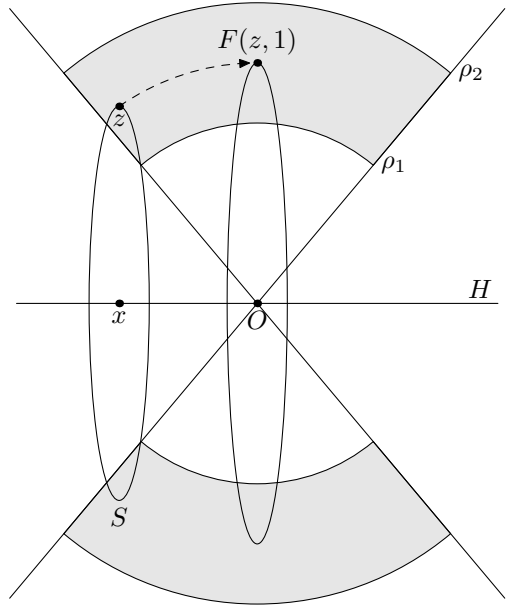}
  \caption{When we move the center of a sphere to the origin, we need to control
    the radius so that the deformation is performed inside the conical cap.}
  \label{F:sphere-move}
\end{figure}

\begin{proof}
  Any point $z \in S$ can be uniquely decomposed into a sum $z = x + r y$, where
  $y \in \Sphere \cap H^{\perp}$ is a point in the unit sphere in
  $H^{\perp}$. We define
  \begin{displaymath}
    F(x + r y, t) := (1-t)x + y \sqrt{r^2 + |x|^2 - |(1-t)x|^2} \,.
  \end{displaymath}
  This gives an isotopy which deforms $S$ to a sphere perpendicular to $H$ and
  centered at the origin (see Figure~\ref{F:sphere-move}). Fix some $z = x + ry
  \in S$. The sphere $S$ is contained in $\Cone(\delta,H)$, so it follows that
  \begin{displaymath}
    \frac{|Q_H(F(z,t))|}{|F(z,t)|} 
    = \frac{\sqrt{r^2 + |x|^2 - |(1-t)x|^2}}{\sqrt{r^2 + |x|^2}}
    \ge \frac{r}{\sqrt{r^2 + |x|^2}}
    = \frac{|Q_H(z)|}{|z|} \ge \delta \,.
  \end{displaymath}
  This shows that the whole deformation is performed inside $\Cone(\delta,H)$. Next,
  we only need to continuously change the radius to the value $\rho$ but this can
  be easily done inside $H^{\perp} \cap (\Ball_{\rho_2} \setminus \Ball_{\rho_1})$.
\end{proof}

\mysubsubsection{Intersecting cones}

In this paragraph we prove a result which allows us to handle the situation of
two intersecting cones. Let $P$ and $H$ be to $m$-planes such that $\dgras(P,H)
< 1$ and such that the cones $\Cone(\sqrt{1 - \alpha^2},P)$ and $\Cone(\sqrt{1 -
  \beta^2},H)$ intersect. The question is: does the intersection
$\Cone(\alpha,P) \cap \Cone(\beta,H)$ contain a cone $\Cone(\gamma,H)$ for some
$\gamma \in (0,1)$? We give a sufficient condition for $\alpha$ and $\beta$
which ensures a positive answer. This will become useful in the proof of
Proposition~\ref{prop:big-proj-fat-simp} where we construct a set $F$ by glueing
some conical caps together and we need to assure that certain spheres contained
in $F$ are linked with $\Sigma$. Knowing that the intersection of two conical
caps contains another one allows us to continuously translate spheres from the
first conical cap to the second.

\begin{prop}
  \label{prop:two-cones}
  Let $\alpha > 0$ and $\beta > 0$ be two real numbers satisfying $\alpha+\beta <
  \sqrt{1-\beta^2}$ and let $H_0,H_1 \in G(n,m)$ be two $m$-planes in $\Rn$.
  Assume that
  \begin{displaymath} 
    \Cone(\sqrt{1-\alpha^2},H_0^{\perp}) \cap \Cone(\sqrt{1-\beta^2},H_1^{\perp}) \ne \emptyset \,.
  \end{displaymath}
  Then for any $\epsilon > 0$ we have the inclusion
  \begin{equation}
    \label{eq:gen-cone-int}
    \Cone((\alpha+\beta)/\sqrt{1-\beta^2} + \epsilon,H_0)
    \subseteq
    \Cone(\epsilon,H_1)\,.
  \end{equation}
  In particular, if $\alpha+\beta \le (1-\beta)\sqrt{1-\beta^2}$, then
  \begin{displaymath}
    H_0^{\perp} \subseteq \Cone(\alpha,H_0) \cap \Cone(\beta,H_1) \,.
  \end{displaymath}
\end{prop}

\begin{proof}
  First we estimate the ``angle'' between $H_0$ and $H_1$. Since the cones
  $\Cone(\sqrt{1-\alpha^2},H_0^{\perp})$ and $\Cone(\sqrt{1-\beta^2},H_1^{\perp})$ have
  nonempty intersection they both must contain a common line $L \in G(n,1)$.
  \begin{displaymath}
    L \subseteq \Cone(\sqrt{1-\alpha^2},H_0^{\perp}) \cap \Cone(\sqrt{1-\beta^2},H_1^{\perp}) \,.
  \end{displaymath}
  Choose some point $z \in H_1$ and find a point $y \in L$ such that $z =
  \pi_{H_1}(y)$ (see Figure~\ref{F:PLH}). Since $y \in
  \Cone(\sqrt{1-\beta^2},H_1^{\perp})$ it follows that $|Q_{H_1}(y)| < \beta
  |y|$.  Furthermore, by the Pythagorean theorem
  \begin{align*}
    |y|^2 = |\pi_{H_1}(y)|^2 + |Q_{H_1}(y)|^2 
    &\le |z|^2 + \beta^2|y|^2 \\
    \text{hence}
    |y| &\le \frac{|z|}{\sqrt{1 - \beta^2}} \,.
  \end{align*}
  Because $y$ also belongs to the cone $\Cone(\sqrt{1-\alpha^2},H_0^{\perp})$ we
  have $|Q_{H_0}(y)| < \delta |y|$, so we obtain
  \begin{align}
    |Q_{H_0}(z)| &\le |Q_{H_0}(y)| + |Q_{H_0}(z-y)| 
    \le |Q_{H_0}(y)| + |z-y|  \notag \\
    &= |Q_{H_0}(y)| + |Q_{H_1}(y)|
    \le \alpha |y| + \beta |y|  
    \le \frac{\alpha + \beta}{\sqrt{1-\beta^2}} |z|
    \qquad \text{for all } z \in H_1
    \label{est:QH0z} \,.
  \end{align}

  \begin{figure}[!htb]
    \centering
    \includegraphics{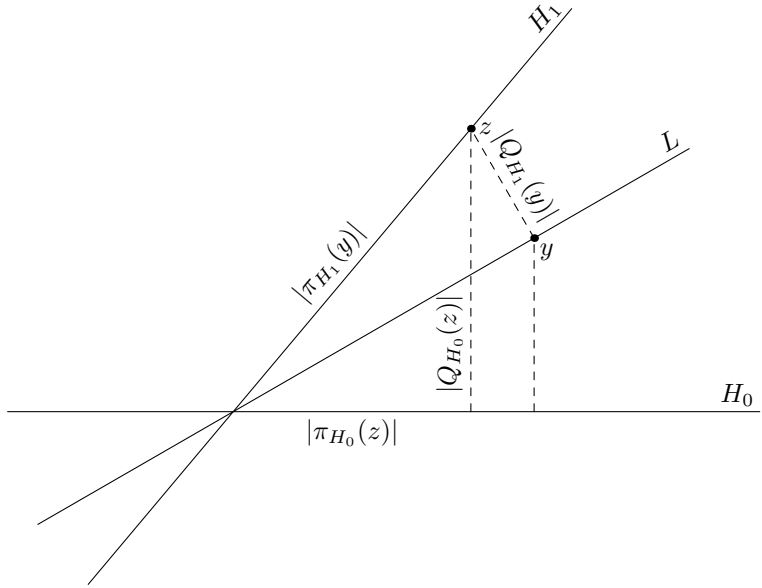}
    \caption{The line $L$ lies in the intersection of two cones:
      $\Cone(\sqrt{1 - \alpha^2},H_0^{\perp})$ and
      $\Cone(\sqrt{1 - \beta^2},H_1^{\perp})$.
      This allows us to estimate the ``angle'' between $H_0$ and $H_1$.}
    \label{F:PLH}
  \end{figure}

  Choose some $\epsilon > 0$ and let
  \begin{displaymath}
    x \in C\left(
      \frac{\alpha+\beta}{\sqrt{1-\beta^2}} + \epsilon,H_0
    \right) \,,
    \quad \text{so} \quad
    |Q_{H_0}(x)| \ge \left(
      \frac{\alpha+\beta}{\sqrt{1-\beta^2}} + \epsilon
    \right)
    |x| \,.
  \end{displaymath}
  If $\epsilon$ is small enough, then such $x$ exists by the assumption that
  $\alpha+\beta < \sqrt{1-\beta^2}$. For bigger $\epsilon$ the inclusion
  $\Cone((\alpha+\beta)/\sqrt{1-\beta^2} + \epsilon,H_0) \subseteq \Cone(\epsilon,H_1)$
  is trivially true. From the triangle inequality
  \begin{align*}
    \frac{\alpha+\beta}{\sqrt{1-\beta^2}} |x|
    &\le |Q_{H_0}(x)| 
    \le |Q_{H_0}(Q_{H_1}(x))| + |Q_{H_0}(\pi_{H_1}(x))| \\
    &\le |Q_{H_1}(x)| + |Q_{H_0}(\pi_{H_1}(x))| \,,
  \end{align*}
  hence
  \begin{displaymath}
    |Q_{H_1}(x)| \ge \frac{\alpha+\beta}{\sqrt{1-\beta^2}} |x| + \epsilon |x| - |Q_{H_0}(\pi_{H_1}(x))| \,.
  \end{displaymath}
  Because $\pi_{H_1}(x) \in H_1$ and because of estimate \eqref{est:QH0z} we have
  \begin{displaymath}
    |Q_{H_1}(x)| 
    \ge \frac{\alpha+\beta}{\sqrt{1-\beta^2}} |x| + \epsilon |x| - \frac{\alpha+\beta}{\sqrt{1-\beta^2}} |\pi_{H_1}(x)|
    \ge \epsilon |x| \,,
  \end{displaymath}
  which ends the proof.
\end{proof}



\mysubsection{Flatness}
Recall the definition of P. Jones' $\beta$-numbers
\begin{defin}
  \label{def:beta}
  Let $\Sigma \subseteq \Rn$ be any set. Let $x \in \Sigma$ and $r > 0$. We
  define the $m$-dimensional \emph{$\beta$ numbers} of $\Sigma$ by the formula
  \begin{align*}
    \nbeta(x,r) &:= \frac 1r \inf \left\{
      \sup_{z \in \Sigma \cap \CBall(x,r)} \dist(z,x+H) : H \in G(n,m)
    \right\} \\
    &= \frac 1r \inf \left\{
      \sup_{z \in \Sigma \cap \CBall(x,r)} |Q_H(z-x)| : H \in G(n,m)
    \right\} \,.
  \end{align*}
\end{defin}

\begin{defin}
  \label{def:hdist}
  For any two sets $E,F \subseteq \Rn$ we define the \emph{Hausdorff distance}
  between these two sets to be
  \begin{displaymath}
    \HD(E,F) := \sup\{ \dist(y,F) : y \in E \} + \sup\{ \dist(y,E) : y \in F \} \,.
  \end{displaymath}
\end{defin}

We will also need the following definition, which originated from Reifenberg's
work~\cite{MR0114145} and his famous topological disc theorem (see~\cite{Sim96}
for a modern proof).
\begin{defin}
  \label{def:theta-fun}
  Let $\Sigma \subseteq \Rn$. For $x \in \Sigma$ and $r > 0$ we define the
  \emph{$\theta$ numbers}
  \begin{align*}
    \ntheta(x,r) &:= \frac 1r \inf\{
    \HD(\Sigma \cap \CBall(x,r), (x + H) \cap \CBall(x,r)) : H \in G(n,m)
    \} \,.
  \end{align*}
\end{defin}

\begin{rem}
  \label{rem:beta-le-theta}
  For each $x \in \Sigma$ and all $r > 0$ we always have $\nbeta(x,r) \le
  \ntheta(x,r)$.
\end{rem}

In~\cite{MR1808649}, David, Kenig and Toro introduced a slightly different
definition of $\beta(x,r)$ and $\theta(x,r)$ using open balls
\begin{align*}
  \beta_m(x,r) &:= \frac 1r
  \inf \left\{
    \sup_{z \in \Sigma \cap \Ball(x,r)} |Q_H(z-x)| : H \in G(n,m)
  \right\} \,, \\
  \theta_m(x,r) &:= \frac 1r \inf\{
  \HD(\Sigma \cap \Ball(x,r), (x + H) \cap \Ball(x,r)) : H \in G(n,m)
  \} \,.
\end{align*}
We use closed balls just for convenience. Unfortunately the $\beta$ and the
$\theta$ numbers are not monotone with respect to $r$, and there is no obvious
relation between $\ntheta$ and $\theta_m$. We shall prove the following
\begin{prop}
  \label{prop:theta-theta}
  For each $x \in \Sigma$ and each $r > 0$ we have
  \begin{align*}
    \beta_m(x,r) &\le \nbeta(x,r) \\
    \text{and} \quad
    \theta_m(x,r) &\le 3 \ntheta(x,r) \,.
  \end{align*}
\end{prop}

\begin{proof}
  The case of $\beta$-numbers is easy. Let us fix some $H \in
  G(n,m)$, then certainly
  \begin{displaymath}
    \sup_{z \in \Sigma \cap \Ball(x,r)} |Q_H(z-x)| 
    \le \sup_{z \in \Sigma \cap \CBall(x,r)} |Q_H(z-x)| \,,
  \end{displaymath}
  hence $\beta_m(x,r) \le \nbeta(x,r)$. For the $\theta$ numbers the situation is
  somewhat more complicated.
  \begin{multline}
    \label{eq:HdistH}
    \HD(\Sigma \cap \Ball(x,r), (x + H) \cap \Ball(x,r))
    = \sup\{ |Q_H(y-x)| : y \in \Sigma \cap \Ball(x,r) \} \\
    \phantom{= }+ \sup\{ \dist(y,\Sigma \cap \Ball(x,r)) : z \in (x + H) \cap \Ball(x,r) \} \,.
  \end{multline}
  Let
  \begin{displaymath}
    \overline{\theta}_H := \tfrac 1r \HD \left(
      \Sigma \cap \CBall(x,r), (x + H) \cap \CBall(x,r)
      \right) \,.
  \end{displaymath}
  Note that the value of \eqref{eq:HdistH} is at most $2r$, so if
  $\overline{\theta}_H \ge \frac 23$, then we obviously have
  \begin{equation}
    \HD(\Sigma \cap \Ball(x,r), (x + H) \cap \Ball(x,r)) \le 2r \le 3 \overline{\theta}_H \,.
  \end{equation}
  We will show that this is also true for $\overline{\theta}_H \le \frac
  23$. The first term of \eqref{eq:HdistH} can be estimated as in the case of
  $\beta$ numbers. Indeed,
  \begin{displaymath}
    \sup\{ |Q_H(y-x)| : y \in \Sigma \cap \Ball(x,r) \} 
    \le \sup\{ |Q_H(y-x)| : y \in \Sigma \cap \CBall(x,r) \}
    \le \overline{\theta}_H r \,.
  \end{displaymath}
  To estimate the second term in \eqref{eq:HdistH} we need to divide the set $(x
  + H) \cap \Ball(x,r)$ into two parts. Set
  \begin{align*}
    A_1 &:= (x + H) \cap \Ball(x,(1-\overline{\theta}_H)r) \\
    \text{and} \quad
    A_2 &:= (x + H) \cap (\Ball(x,r) \setminus \Ball(x,(1-\overline{\theta}_H)r)) \,.
  \end{align*}
  Note that for each $z \in A_1$ there exists a point $y \in \Sigma \cap
  \CBall(x,r)$ such that $|y-z| \le \overline{\theta}_H r$, so $|z-x| \le
  |z-y|+|y-x| < r$. Hence $y \in \Sigma \cap \Ball(x,r)$. On the other hand if
  we take $y \in \partial \Ball(x,r)$, then $|z-y| \ge \overline{\theta}_H
  r$. This shows that
  \begin{displaymath}
    \sup\{ \dist(y,\Sigma \cap \Ball(x,r)) : z \in A_1 \} \le \overline{\theta}_H r \,.
  \end{displaymath}
  For each $z \in A_2$ we can find $z' \in A_1$ such that $|z - z'| \le \overline{\theta}_H
  r$ and repeating the previous argument we obtain
  \begin{displaymath}
    \sup\{ \dist(y,\Sigma \cap \Ball(x,r)) : z \in A_2 \} \le 2 \overline{\theta}_H r \,.
  \end{displaymath}
  Therefore
  \begin{displaymath}
    \HD(\Sigma \cap \Ball(x,r), (x + H) \cap \Ball(x,r)) \le 3 \overline{\theta}_H r \,.
  \end{displaymath}
  Taking the infimum over all $H \in G(n,m)$ on both sides and dividing by $r$
  we reach our conclusion $\theta_m(x,r) \le 3 \ntheta(x,r)$.
\end{proof}

For convenience we also introduce the following
\begin{defin}
  \label{def:beta-plane}
  Let $\Sigma \subseteq \Rn$ be any set. Let $x \in \Sigma$ and $r > 0$. We
  say that $H \in G(n,m)$ is the \emph{best approximating $m$-plane} for
  $\Sigma$ in $\CBall(x,r)$ and write $H \in \BAP(x,r)$ if the
  following condition is satisfied
  \begin{displaymath}
    \HD(\Sigma \cap \CBall(x,r), (x + H) \cap \CBall(x,r))
    \le \ntheta(x,r) \,.
  \end{displaymath}
\end{defin}
Since $G(n,m)$ is compact, such $H$ always exists, but it might not be unique,
e.g. consider the set $\Sigma = \Sphere \cup \{ 0 \}$ and take $x=0$, $r=2$.

\begin{rem}
  For each $x,y \in \Sigma$ and each $H \in \BAP(x,|x-y|)$ we have
  \begin{displaymath}
    \dist(y, x + H) \le \nbeta(x,|x-y|) \,.
  \end{displaymath}
\end{rem}

\begin{defin}[\!\!\cite{MR1808649}, Definition~1.3]
  \label{def:rfvc}
  We say that a closed set $\Sigma \subseteq \Rn$ is \emph{Reifenberg-flat with
    vanishing constant} (of dimension $m$) if for every compact subset $K
  \subseteq \Sigma$
  \begin{displaymath}
    \lim_{r \to 0} \sup_{x \in K} \theta_m(x,r) = 0 \,.
  \end{displaymath}
\end{defin}

The following proposition was proved by David, Kenig and Toro.
\begin{prop}[\!\!\cite{MR1808649}, Proposition~9.1]
  \label{prop:dkt-reg}
  Let $\tau \in (0,1)$ be given. Suppose $\Sigma$ is a Reifenberg-flat set with
  vanishing constant of dimension $m$ in $\Rn$ and that, for each compact subset
  $K \subseteq \Sigma$ there is a constant $C_K$ such that
  \begin{displaymath}
    \beta_m(x,r) \le C_K r^{\tau} \quad \text{for each $x \in K$ and $r \le 1$.}
  \end{displaymath}
  Then $\Sigma$ is a $C^{1,\tau}$-submanifold of $\Rn$.
\end{prop}

In \S\ref{sec:tangent-planes} we show how to use this proposition to prove the
regularity of a certain class (cf. Definition~\ref{def:fine}) of sets with
finite integral curvature - but this is not enough for our purposes. We need to
control the parameters of a local graph representation of $\Sigma$ in terms of
the energy $\E_p(\Sigma)$ (see Definition~\ref{def:p-energy}). We need to prove
that there exists a scale $R$ such that $\Sigma \cap \Ball(x,R)$ is a graph of
some function $F_x$, and the bound for the H{\"o}lder constant of $DF_x$ and the
radius $R$ can be estimated in terms of $\E_p(\Sigma)$. Hence, we formulate
Theorem~\ref{thm:C1tau} and we prove it independently of
Proposition~\ref{prop:dkt-reg}.


\mysubsection{Voluminous simplices}

In Section~\ref{sec:p-energy} we give the definition of the energy functional
$\E_p$. This functional is just the integral over all $(m+1)$-simplices with
vertices on $\Sigma$. The integrand measures the ''regularity'' of each simplex
divided by its diameter. For ''quite regular'' simplices it is proportional to
the inverse of the diameter. Here we formalize what we mean by ''quite regular''
defining tha class of \emph{$(\eta,d)$-voluminous} simplices and prove that
simplices close to a fixed voluminous simplex are again voluminous. We will need
this result in the proof of Proposition~\ref{prop:dSigma-lower} to estimate the
$p$-energy of $\Sigma$. Having one voluminous simplex and knowing that there are
many (in the sense of measure) points of $\Sigma$ close to each vertex of that
simplex, we can use the result of this section to estimate $\E_p(\Sigma)$ from
below. This will show (cf. Proposition~\ref{prop:eta-d-balance}) that whenever
we have a bound $\E_p(\Sigma) < E$, then at some small scale, depending only on
$E$, all the simplices with vertices on $\Sigma$ are almost flat.

Let $T = \simp(x_0, \ldots, x_{k+1}) \subseteq \Rn$ be a $(k+1)$-dimensional
simplex. Recall (see Definition~\ref{def:face-hmin}) that $\face_jT$ and
$\height_jT$ denote the $j^{th}$ face and the $j^{th}$ height of $T$ respectively.
\begin{defin}
  \label{def:regsimp}
  Let $\eta \in (0,1)$ and $d > 0$. Choose some $k \in \{1,\ldots,n-1\}$. We say
  that $T = \simp(x_0, \ldots, x_{k+1}) \subseteq \Rn$ is
  $(\eta,d)$-\emph{voluminous} and write $T \in \Reg_k(\eta,d)$ if the following
  conditions are satisfied
  \begin{itemize}
  \item $T$ is contained in some ball of radius $d$, i.e.
    \begin{equation}
      \exists\ x \in \Rn \quad T \subseteq \CBall(x,d) \,, \label{fat:ball}
    \end{equation}
  \item the measure of the base of $T$ is not less than $(\eta d)^k$, i.e.
    \begin{equation}
      \HM^{k}(\face_{k+1} T) \ge (\eta d)^k \,, \label{fat:base}
    \end{equation}
  \item the height of $T$ is not less than $\eta d$, i.e.
    \begin{equation}
      \height_{k+1}(T) \ge \eta d \,. \label{fat:height}
    \end{equation}
  \end{itemize}
\end{defin}

The following remarks will be used in the proof of
Proposition~\ref{prop:perturb} but they also show that we obtain an equivalent
definition of a voluminous simplex if we replace conditions \eqref{fat:base} and
\eqref{fat:height} by just one condition: $\hmin(T) \ge \eta d$. However, our
definition of $\Reg_k(\eta,d)$ is more convenient for proving theorems stated in
Section \ref{sec:ahl-reg}.

\begin{rem}
  \label{rem:hmin-ito-measures}
  Let $k \in \{1,\ldots,n-1\}$. For any $i = 0, \ldots, k+1$ the
  $(k+1)$-dimensional measure of $T$ is given by the formula
  \begin{displaymath}
    \HM^{k+1}(T) = \frac 1{k+1} \height_i(T) \HM^k(\face_i T) \,.
  \end{displaymath}
  Hence, we can express $\hmin(T)$ only in terms of measures of simplices
  \begin{displaymath}
    \hmin(T) = (k+1) \HM^{k+1}(T) \left( \max_{0 \le i \le k+1} \HM^k(\face_i T) \right)^{-1} \,.
  \end{displaymath}
\end{rem}

\begin{rem}
  \label{rem:hmin-estimates}
  Let $k \in \{1,\ldots,n-1\}$. If $T \in \Reg_k(\eta,d)$ then we can estimate
  its measure from below by
  \begin{equation} \label{est:meas-low}
    \HM^{k+1}(T) \ge \frac 1 {k+1} (\eta d)^{k+1} \,.
  \end{equation}

  Using the Pythagorean theorem, one can easily prove that $\hmin(T)$ is less or
  equal to any height of any simplex in the skeleton of $T$ of any
  dimension. This means in particular, that
  \begin{equation} \label{est:diam-low}
    |x_i - x_j| \ge \hmin(T) \quad \text{for any } i \ne j \,.
  \end{equation}

  Due to condition \eqref{fat:ball} all the $l$-dimensional faces of $T$ have
  measure bounded from above by $\omega_l d^l$, where $\omega_l :=
  \HM^l(\Ball \cap \R^l)$. Hence we get an estimate for the $l$-measure of any
  $l$-simplex in the $l$-skeleton of $T$ for any $l \le k+1$. In particular
  \begin{align}
    \frac 1 {(k+1)!} \hmin(T)^{k+1} \le \HM^{k+1}(T) &\le \omega_{k+1} d^{k+1} \,, \label{est:meas-T} \\
    \frac 1 {k!} \hmin(T)^k \le \HM^k(\face_i T) &\le \omega_k d^k \,. \label{est:meas-face-T}
  \end{align}
  Note that \eqref{fat:ball} lets us also derive a lower bound on $\hmin(T)$
  \begin{displaymath} 
    \hmin(T) = \frac{(k+1) \HM^{k+1}(T)}{\max_{0 \le i \le k+1} \HM^k(\face_i T)} 
    \ge \frac{(\eta d)^{k+1}}{\omega_k d^k} = d \frac{\eta^{k+1}}{\omega_k} \,.
  \end{displaymath}
  Combining this and \eqref{est:meas-face-T} we obtain
  \begin{equation}
    \label{est:hmin}
    d \frac{\eta^{k+1}}{\omega_k} \le \hmin(T) \le d \sqrt[k]{\omega_k k!} \,.
  \end{equation}
\end{rem}

\begin{defin}
  Let $k \in \{1,\ldots,n-1\}$ and let $T = \simp(x_0, \ldots, x_{k+1})$, $T' =
  \simp(x_0', \ldots, x_{k+1}')$ be two $(k+1)$-simplices in $\Rn$. We define
  the \emph{pseudo-distance between $T$ and $T'$} as
  \begin{displaymath}
    \|T - T'\| := \min
    \left\{
      \max_{0 \le i \le k+1} |x_{i} - x_{\sigma_i}'| : \sigma \in \perm(k+2)
    \right\} \,,
  \end{displaymath}
  where $\perm(k+2)$ denotes the set of all permutations of the set $\{0,1,
  \ldots, k+1\}$.
\end{defin}

\begin{rem}
  $\|T - T'\| = 0$ if and only if $T$ and $T'$ represent the same geometrical
  simplex, meaning that they can only differ by a permutation of vertices.
\end{rem}

Now we prove that all simplices close to some fixed voluminous simplex are again
voluminous with slightly changed parameters. We need this result for the proof
of Proposition~\ref{prop:dSigma-lower} relating the $p$-energy to the values of
$\beta$-numbers.
\begin{prop}
  \label{prop:perturb}
  Let $\eta \in (0,1)$ and $T \in \Reg_k(\eta,d)$. There exists a small,
  positive number $\varsigma_k = \varsigma_k(\eta)$ such that for each $T'$
  satisfying $\|T - T'\| \le \varsigma_k d$ we have $T' \in \Reg_k(\tfrac 12
  \eta, \frac 32 d)$.
\end{prop}

\begin{proof}
  First we ensure that $\varsigma_k d$ is less than half of the length of the
  shortest side of $T$. Then $T'$ can be obtained from $T$ by moving each vertex
  inside a ball of radius $\varsigma_k d$. Using \eqref{est:diam-low} and
  \eqref{est:hmin} we get
  \begin{displaymath}
    \tfrac 12 \min_{i \ne j} |x_i - x_j|
    \ge \tfrac 12 \hmin(T) 
    \ge  d \frac{\eta^{k+1}}{2\omega_k} \,.
  \end{displaymath}
  Hence
  \begin{equation}
    \label{est:alpha-side}
    \varsigma_k \le \frac{\eta^{k+1}}{2 \max \{ 1, \omega_k \}}
    \quad \text{is enough to ensure} \quad
    \varsigma_k d \le \tfrac 12 \min_{i \ne j} |x_i - x_j| \,.
  \end{equation}

  The plan is to move the vertices of $T$ one by one controlling the parameters
  $\eta$ and $d$ at each step. Note that all the simplices involved in this
  process are contained in the ball $\CBall(x,(1 + \varsigma_k) d)$, where $x$
  is the point defined in \eqref{fat:ball}. We set the value of the second
  parameter to $(1 + \varsigma_k)d$ and never change it. This means that
  $\varsigma_k$ should be at most $\tfrac 12$ and that is why we put $\max\{ 1,
  \omega_k \}$ in \eqref{est:alpha-side}, which now guarantees that $\varsigma_k
  \le \tfrac 12$ because $\eta \in (0,1)$. After changing $d$, the first
  parameter $\eta$ has to be adjusted, so that $T$ meets the conditions imposed
  on voluminous simplices. One can easily see that $T \in
  \Reg_k(\frac{\eta}{1+\varsigma_k},(1+\varsigma_k)d)$. Now we need to calculate
  how does the first parameter change when we move the first vertex $x_0$ to a
  new position $\tilde{x}_0$, such that $|x_0 - \tilde{x}_0| \le \varsigma_k d$.

  Set $T_1 := \simp(\tilde{x}_0, x_1, \ldots, x_{k+1})$, where $\tilde{x}_0 \in
  \CBall(x_0,\varsigma_k d)$. Note that
  \begin{displaymath}
    \HM^{k}(\face_{k+1} T) = \frac 1m \height_0(\face_{k+1} T) \HM^{k-1}(\face_0 \face_{k+1} T) \,.
  \end{displaymath}
  The only factor of the above product which depends on $x_0$ is
  $\height_0(\face_{k+1} T)$. If we move $x_0$ inside $\CBall(x_0,\varsigma_k d)$ we can
  change the value of $\height_0(\face_{k+1} T)$ by at most $\varsigma_k d$. This means
  that $\HM^{k}(\face_{k+1} T)$ changes by at most $\frac 1m \varsigma_k d
  \HM^{k-1}(\face_0 \face_{k+1} T)$. Our simplex $T$ lies inside the ball
  $\CBall(x,(1+\varsigma_k)d)$, so the measure $\HM^{k-1}(\face_0 \face_{k+1} T)$ cannot
  exceed $\omega_{k-1} ((1+\varsigma_k)d)^{k-1}$. This gives the estimate
  \begin{equation}
    \label{est:hm-diff}
    \left| \HM^k(\face_{k+1} T) - \HM^k(\face_{k+1} T_1) \right|
    \le \frac{\omega_{k-1}}{k} \frac{\varsigma_k}{1+\varsigma_k} ((1+\varsigma_k)d)^k \,.
  \end{equation}
  Using the same method for $(k+1)$-dimensional simplices we obtain
  \begin{equation}
    \label{est:hm1-diff}
    \left| \HM^{k+1}(T) - \HM^{k+1}(T_1) \right|
    \le \frac{\omega_k}{(k+1)} \frac{\varsigma_k}{(1+\varsigma_k)} ((1+\varsigma_k)d)^{k+1} \,.
  \end{equation}
  Let $\Upsilon = \Upsilon(k) > 0$ be some big number. We will fix its value later. To ensure
  that condition \eqref{fat:base} does not change too much for $T_1$ we impose
  another constraint,
  \begin{equation}
    \label{est:alpha-base-diff}
    (1+\varsigma_k)^{k-1} \varsigma_k \le \frac{k \eta^k}{\Upsilon \omega_{k-1}} \,.
  \end{equation}
  For such $\varsigma_k$ we have
  \begin{multline}
    \label{est:base-diff}
    \HM^k(\face_{k+1} T_1) 
    \ge \HM^k(\face_{k+1} T) - \frac 1K \left( \frac{\eta}{1+\varsigma_k} \right)^k ((1+\varsigma_k)d)^k \\
    \ge \frac{\Upsilon-1}{\Upsilon} \left( \frac{\eta}{1+\varsigma_k} \right)^k ((1+\varsigma_k)d)^k 
    \ge \left( \frac{\frac{\Upsilon-1}{\Upsilon+1}\eta}{1+\varsigma_k} \right)^k ((1+\varsigma_k)d)^k \,.
  \end{multline}
  Here, we used the estimate $\eqref{est:meas-face-T}$ for $T \in
  \Reg_k(\frac{\eta}{1+\varsigma_k},(1+\varsigma_k)d)$.

  Finally, we can estimate the height $\height_{k+1}(T_1)$ as follows:
  \begin{displaymath}
    \height_{k+1}(T_1) 
    = \frac{(k+1)\HM^{k+1}(T_1)}{\HM^k(\face_{k+1} T_1)}
    \, \operatorname*{\ge}^{\eqref{est:hm1-diff}}_{\eqref{est:hm-diff}}\,
    \frac{(k+1)\HM^{k+1}(T) - \frac{\varsigma_k}{1+\varsigma_k} \omega_k ((1+\varsigma_k)d)^{k+1}}
    {\HM^k(\face_{k+1} T) + \frac{\varsigma_k}{1+\varsigma_k} \frac{\omega_{k-1}}{k} ((1+\varsigma_k)d)^k} \,.
  \end{displaymath}
  To obtain a handy form of this estimate we impose the following constraints on
  $\varsigma_k$:
  \begin{align*}
    \frac{\varsigma_k}{1+\varsigma_k} \omega_k ((1+\varsigma_k)d)^{k+1} &\le \frac 1K (k+1) \HM^{k+1}(T) \\
    \text{and} \quad
    \frac{\varsigma_k}{1+\varsigma_k} \frac{\omega_{k-1}}{k} ((1+\varsigma_k)d)^k &\le \frac 1K \HM^k(\face_{k+1} T) \,.
  \end{align*}
  Using \eqref{est:meas-T}, \eqref{est:meas-face-T} and \eqref{est:hmin} adjusted
  for the class $\Reg_k(\frac{\eta}{1+\varsigma_k},(1+\varsigma_k)d)$ we can guarantee these
  constraints by choosing $\varsigma_k$ satisfying
  \begin{align}
    (1+\varsigma_k)^{(k+1)^2 - 1}\varsigma_k &\le \frac{\eta^{(k+1)^2}}{\Upsilon \omega_k^{k+2} k!}
    \label{est:alpha-num} \\
    \text{and} \quad
    (1+\varsigma_k)^{k(k+1) - 1}\varsigma_k &\le \frac{\eta^{k(k+1)}}{\Upsilon \omega_k^k \omega_{k-1} (k-1)!} \,.
    \label{est:alpha-denom}
  \end{align}
  This way we get the estimate
  \begin{equation}
    \label{est:height-diff}
    \height_{k+1}(T_1) \ge 
    \frac{(k+1)\HM^{k+1}(T)(1 - \frac 1K)}{\HM^k(\face_{k+1} T)(1 + \frac 1K)}
    = \tfrac{\Upsilon-1}{\Upsilon+1} \height_{k+1}(T)
    \ge \frac{\tfrac{\Upsilon-1}{\Upsilon+1} \eta}{1+\varsigma_k} (1+\varsigma_k)d \,.
  \end{equation}

  Up to now we have a few restrictions on $\varsigma_k$, namely
  \eqref{est:alpha-side}, \eqref{est:alpha-base-diff}, \eqref{est:alpha-num} and
  \eqref{est:alpha-denom}. Recall that $\eta < 1$, so among these inequalities
  the smallest right-hand side is in \eqref{est:alpha-num}. Adding one more
  constraint
  \begin{displaymath}
    \varsigma_k \le 2^{1/(k+1)^2} - 1
  \end{displaymath}
  we can assume that all the left-hand sides of \eqref{est:alpha-side},
  \eqref{est:alpha-base-diff}, \eqref{est:alpha-num} and \eqref{est:alpha-denom}
  are at most $2\varsigma_k$. Now, we can safely set
  \begin{equation}
    \label{def:alpha}
    \varsigma_k := \min\left\{
      2^{1/(k+1)^2} - 1,
      \frac{\eta^{(k+1)^2}}{2 \Upsilon \omega_k^{k+2} k!}
    \right\} \,.
  \end{equation}
  With this value of $\varsigma_k$ we have 
  \begin{displaymath}
    T \in \Reg_k \left( \frac{\eta}{(1+\varsigma_k)}, (1+\varsigma_k)d \right)
    \quad \text{and} \quad
    T_1 \in \Reg_k \left( \frac{\tfrac{\Upsilon-1}{\Upsilon+1}\eta}{(1+\varsigma_k)}, (1+\varsigma_k)d \right) \,.
  \end{displaymath}

  Set $\eta' = \frac{\Upsilon-1}{\Upsilon+1}\eta$ and let $T_2 = \simp(\tilde{x}_0,
  \tilde{x}_1, \ldots, x_{k+1})$ be a simplex obtained from $T_1$ by moving
  $x_1$ to a new position $\tilde{x}_1$, such that $|x_1 - \tilde{x}_1| \le
  \varsigma_k d$ and leaving other vertices fixed. Note that $T_1 \in
  \Reg_k(\frac{\eta'}{1+\varsigma_k},(1+\varsigma_k)d)$. Repeating the previous reasoning we
  get
  \begin{displaymath}
    T_2 \in \Reg_k\left( \frac{\tfrac{\Upsilon-1}{\Upsilon+1}\eta'}{(1+\varsigma_k)}, (1+\varsigma_k)d \right)
    = \Reg_k\left( \left(\tfrac{\Upsilon-1}{\Upsilon+1}\right)^2 \frac{\eta}{(1+\varsigma_k)}, (1+\varsigma_k)d \right)\,.
  \end{displaymath}
  Moving each vertex one by one we obtain by induction
  \begin{displaymath}
    T' \in \Reg_k \left(
      \left( \tfrac{\Upsilon-1}{\Upsilon+1} \right)^{k+2}\tfrac{\eta}{1 + \varsigma_k}, (1+\varsigma_k) d
    \right)
    \subseteq \Reg_k \left(
      \tfrac 23 \left( \tfrac{\Upsilon-1}{\Upsilon+1} \right)^{k+2}\eta, \tfrac 32 d
    \right) \,.
  \end{displaymath}
  Now we can fix the value of $\Upsilon(k)$
  \begin{equation}
    \label{def:upsilon}
    \Upsilon(k) := \frac{1 + \left(\tfrac 34\right)^{1/(k+2)}}{1 - \left(\tfrac 34\right)^{1/(k+2)}}
  \end{equation}
  and we get the desired conclusion $T' \in \Reg_k(\tfrac 12 \eta, \frac 32 d)$.
\end{proof}

In Section \ref{sec:ahl-reg} we will need to know how does $\varsigma_k$ depend on
$\eta$, when $\eta \to 0$.

\begin{rem}
  Recall that
  \begin{displaymath}
    \omega_k := \HM^{k}(\Ball \cap \R^k) = \frac{\pi^{k/2}}{\Gamma(\frac k2 + 1)} \,,
  \end{displaymath}
  so $\omega_k$ converges to zero when $k \to \infty$. Set
  \begin{equation}
    \label{def:Omega}
    \Omega := \sup\{ \omega_k : k \in  \N \} \,.
  \end{equation}
  We can find an absolute constant $\Cl{eta-const} \in (0,1)$ such that for
  every $k \in \N$
  \begin{displaymath}
    2^{1/(k+1)^2} - 1 \ge \frac{\sqrt{\Cr{eta-const}}}{(k+1)^2}
    \quad \text{and} \quad
    \frac 1{(k+1)^2} \ge \frac{\sqrt{\Cr{eta-const}}}{2 \Upsilon(k) \Omega^{k+2} k!} \,.
  \end{displaymath}
  Recall that $\varsigma_k$ was defined by \eqref{def:alpha}. Since $\eta \in (0,1)$
  we have
  \begin{equation}
    \label{eq:alpha-eta}
    \frac{\Cr{eta-const} \eta^{(k+1)^2}}{2 \Upsilon(k) \Omega^{k+2} k!} 
    \le \varsigma_k(\eta) 
    \le \frac{\eta^{(k+1)^2}}{2 \Upsilon(k) \omega_k^{k+2} k!} \,,
  \end{equation}
  so
  \begin{displaymath}
    \varsigma_k(\eta) \approx \eta^{(k+1)^2} \,.
  \end{displaymath}
\end{rem}



\mysubsection{The $p$-energy functional}
\label{sec:p-energy}

First we define a higher dimensional analogue of the Menger curvature defined
for curves.
\begin{defin}
  \label{def:disc-curv}
  Let $T = \simp(x_0, \ldots, x_{m+1})$. The \emph{discrete curvature} of $T$ is
  \begin{displaymath}
    \K(T) := \frac{\HM^{m+1}(T)}{\diam(T)^{m+2}} \,.
  \end{displaymath}
\end{defin}
Note that $\K(\alpha T) = \frac 1{\alpha}\K(T) \to \infty$ when $\alpha \to 0$,
so our curvature behaves under scaling like the original Menger curvature. If
$T$ is a regular simplex (meaning that all the side lengths are equal), then
$\K(T) \simeq \frac{1}{\diam T} \simeq R(T)^{-1}$, where $R(T)$ is the radius of
a circumsphere of the vertices of $T$.

For $m = 1$ using the sine theorem we obtain
\begin{align*}
  \frac 1{R(T)} &= \frac{4 \operatorname{Area}(T)}{|x_0-x_1| |x_1-x_2| |x_2-x_0|} \\
  \text{and} \quad
  \K(T) &= \frac{\operatorname{Area}(T)}{\max\{|x_0-x_1|,|x_1-x_2|,|x_2-x_0|\}^3} \,.
\end{align*}
Hence, for an equilateral triangle this two quantities are the same up to an
absolute constant. For all other triangles we only have $\K(T) \le R(T)^{-1}$.

In the case of surfaces ($m=2$), Strzelecki and von der Mosel~\cite{0911.2095}
suggested the following definition of discrete curvature
\begin{displaymath}
  \K'(T) := \frac{\operatorname{Volume}(T)}{\operatorname{Area}(T) \diam(T)^2} \,.
\end{displaymath}
For a regular tetrahedron $\operatorname{Volume}(T) = \frac{\sqrt 2}{12}d^3$ and
$\operatorname{Area}(T) = \sqrt 3 d^2$, so in this case
\begin{displaymath}
  \K'(T) = \frac{\sqrt 2}{12 \sqrt 3 \diam(T)} = \frac 1{\sqrt 3} \K(T) \,.
\end{displaymath}
Once again we see that these definitions coincide for regular simplices. Note
also that $\operatorname{Area}(T) \le 4 \pi d^2$ so $\K(T) \le 4\pi \K'(T)$.

We emphasis the behavior on regular simplices because small, close to regular
(or \emph{voluminous}) simplices are the reason why $\E_p(\Sigma)$ might get
very big or infinite. For the class of voluminous simplices $T \in
\Reg_m(\eta,d)$ the value $\K(T)$ is comparable with yet another possible
definition of discrete curvature
\begin{displaymath}
  \K''(T) := \frac{\hmin(T)}{\diam(T)^2} 
  = \frac{1}{\diam(T)} \frac{\hmin(T)}{\diam(T)} \,,
\end{displaymath}
which is basically $\frac 1{\diam(T)}$ multiplied by a scale-invariant
''regularity coefficient'' $\frac{\hmin(T)}{\diam(T)}$. This last factor
prevents $\K''$ from blowing up on simplices with vertices on smooth manifolds.

One could ask, if we cannot define $\K(T)$ to be $R(T)^{-1}$. Actually
$R(T)^{-1}$ is not good in the sense that there are examples (see~\cite[Appendix
B]{0911.2095}) of $C^2$ manifolds for which $R(T)^{-1}$ explodes. These examples
use the fact that a circumsphere of a small, very elongated simplex may be quite
different from the tangent sphere and intersect the affine tangent space on a
big set. This is the advantage of our definition of $\K(T)$. It is defined in
such a way that very thin simplices have small discrete curvature.

\begin{fact}
  \label{fact:reg-curv}
  If $T \in \Reg_m(\eta,d)$ then
  \begin{equation}
    \label{est:curv}
    \K(T) \ge \frac{(\eta d)^{m+1}}{(m+1) (2d)^{m+2}} = \frac 1 {(m+1) 2^{m+2}} \frac{\eta^{m+1}}{d} \,.
  \end{equation}
\end{fact}

\begin{defin}
  Let $\Sigma \subseteq \Rn$ be any $\HM^m$-measurable set. We define the
  measure $\mu_{\Sigma}$ to be the $(m+2)$-fold product of the $m$-dimensional
  Hausdorff measures, restricted to $\Sigma$, i.e.
  \begin{displaymath}
    \mu_{\Sigma} := \underbrace{\HM^m|_{\Sigma} \otimes \cdots \otimes \HM^m|_{\Sigma}}_{m+2} \,.
  \end{displaymath}
\end{defin}

In this paper we usually work with only one set $\Sigma$, so if there is no
ambiguity, we will drop the subscript and write just $\mu$ for the measure
$\mu_{\Sigma}$.

\begin{defin}
  \label{def:p-energy}
  For $\Sigma \subseteq \Rn$ a $\HM^m$-measurable set we define the
  \emph{$p$-energy} functional
  \begin{displaymath}
    \E_p(\Sigma) := \int_{\Sigma^{m+2}} \K(T)^p\ d\mu_{\Sigma}(T) \,.
  \end{displaymath}
\end{defin}

\begin{prop}
  \label{prop:beta-curv}
  If $\Sigma \subseteq \Rn$ is $m$-dimensional, compact and such that
  \begin{displaymath}
    \exists R > 0 \ 
    \exists C > 0 \ 
    \forall x \in \Sigma \ 
    \forall r \in (0,R]
    \quad
    \nbeta(x,r) \le C r
  \end{displaymath}
  then the discrete curvature $\K$ is uniformly bounded on
  $\Sigma^{m+2}$. Therefore for such $\Sigma$ the $p$-energy $\E_p(\Sigma)$ is
  finite for any $p > 0$.
\end{prop}

\begin{proof}
  Let us assume that there exists a sequence of simplices $T_k$ such that
  $\K(T_k)$ is unbounded, meaning
  \begin{equation}
    \label{eq:bigvol}
    \forall \tilde{C} > 0 \ 
    \exists k_0\ 
    \forall k \ge k_0
    \quad
    \HM^{m+1}(T_k) \ge \tilde{C} \diam(T_k)^{m+2} \,.
  \end{equation}
  Let us denote the vertices of $T_k$ by $x_0^k$, $x_1^k$, \ldots,
  $x_{m+1}^k$. Set $d_k := \diam(T_k)$. Since $\Sigma$ is compact the diameter
  of $T_k$ is bounded. Hence the measure $\HM^{m+1}(T_k)$ is also bounded, so if
  $\K(T_k)$ explodes, then $d_k$ must converge to $0$.

  Choose $k_0 \in \N$ such that $d_k < \min \{ R, \frac 1C \}$ for each $k \ge
  k_0$. For each $k$ fix some $m$-plane $H_k \in G(n,m)$ such that
  \begin{equation}
    \label{eq:points-dist}
    \forall y \in \Sigma \cap \CBall(x_0^k,d_k)
    \quad
    \dist(y, x_0^k + H_k) \le C d_k^2 \,.
  \end{equation}
  This is possible because $\nbeta(x_0^k,d_k) \le C d_k$. Fix some $k \ge k_0$
  and set $h_k := C d_k^2 \le d_k$. We shall estimate the measure of $T_k$ and
  contradict~\eqref{eq:bigvol}.

  Without loss of generality we can assume $x_0^k$ lies at the origin. Let us
  choose an orthonormal coordinate system $v_1$, \ldots, $v_n$ such that $H_k =
  \opspan\{v_1, \ldots, v_m\}$. Because of~\eqref{eq:points-dist} in our
  coordinate system we have
  \begin{displaymath}
    T_k \subseteq [-d_k, d_k]^m \times [-h_k,h_k]^{n-m}\,.
  \end{displaymath}
  Of course $T_k$ lies in some $(m+1)$-dimensional section of the above
  product. Let 
  \begin{align*}
    V_k &:= \aff \{ x_0^k, \ldots, x_{m+1}^k \} = \opspan \{ x_1^k, \ldots, x_{m+1}^k \} \,,\\
    Q(a,b) &:= [-a,a]^m \times [-b,b]^{n-m} \,,\\
    Q_k &:= Q(d_k,h_k) \\
    \text{and} \quad
    P_k &:= V_k \cap Q_k \,.
  \end{align*}
  Note that all of the sets $V_k$, $Q_k$ and $P_k$ contain $T_k$. Choose another
  orthonormal basis $w_1$, \ldots, $w_n$ of $\Rn$, such that $V_k = \opspan\{w_1,
  \ldots, w_{m+1}\}$. Let $S_k := \{ x \in V_k^{\perp} : |\langle x, w_i \rangle|
  \le h_k \}$, so $S_k$ is just the cube $[-h_k,h_k]^{n-m-1}$ placed in the
  orthogonal complement of $V_k$. Note that $\diam S_k = 2 h_k \sqrt{n-m-1}$. In
  this setting we have
  \begin{equation} \label{eq:PxS}
    P_k \times S_k 
    = P_k + S_k 
    \subseteq Q(d_k + 2h_k \sqrt{n-m-1}, h_k + 2h_k\sqrt{n-m-1}) \,.
  \end{equation}
  Recall that $h_k = C d_k^2 \le d_k$. We obtain the following estimate
  \begin{align}
    \label{eq:PxS-upper-bound}
    \HM^n(T_k \times S_k) 
    &\le \HM^n(P_k \times S_k) \\
    &\le \HM^n(Q(d_k + 2h_k \sqrt{n-m-1}, h_k + 2 h_k \sqrt{n-m-1})) \notag \\
    &\le (2 d_k + 4 h_k \sqrt{n-m-1})^m (2 h_k + 4 h_k \sqrt{n-m-1})^{n-m} \notag \\
    &\le (2 + 4\sqrt{n-m-1})^m (2C + 4C\sqrt{n-m-1})^{n-m} d_k^m h_k^{n-m} \notag \\
    &=: C'(n,m) C^{n-m} d_k^m h_k^{n-m} \notag \,.
  \end{align}
  Choose $\tilde{C} > C'(n,m) C^{n-m+1}$ and use \eqref{eq:bigvol} to find $k$
  such that $\HM^{m+1}(T_k) \ge \tilde{C} d_k^{m+2}$. Then we obtain
  \begin{align}
    \label{eq:PxS-lower-bound}
    \HM^n(T_k \times S_k) &= \HM^{m+1}(T_k) \HM^{n-m-1}(S_k) \\
    &\ge \tilde{C} 2^{n-m-1} h_k^{n-m-1} d_k^{m+2} \notag \\
    &> \frac{2^{n-m-1}}{C} \tilde{C} h_k^{n-m} d_k^m \notag \\
    &\ge 2^{n-m-1} C^{n-m} C'(n,m) h_k^{n-m} d_k^m \notag \,.
  \end{align}
  Now, \eqref{eq:PxS-upper-bound} and \eqref{eq:PxS-lower-bound} give a
  contradiction, so condition \eqref{eq:bigvol} must have been false.
\end{proof}

\begin{cor}
  \label{cor:C2mani}
  If $M \subseteq \Rn$ is a compact, $m$-dimensional, $C^2$ manifold embedded in
  $\Rn$ then the discrete curvature $\K$ is uniformly bounded on
  $M^{m+2}$. Therefore the $p$-energy $\E_p(M)$ is finite for every $p > 0$.
\end{cor}

\begin{proof}
  Since $M$ is a compact $C^2$-manifold, it has a tubular neighborhood
  \begin{displaymath}
    M_{\varepsilon} = M + \overline{B}_{\varepsilon}
    := \{ x + y : x \in M,\, y \in \overline{B}_{\varepsilon} \}
  \end{displaymath}
  of some radius $\varepsilon > 0$ and the nearest point projection $\pi :
  M_{\varepsilon} \to M$ is a well-defined, continuous function (see
  e.g.~\cite{MR0110078} for a discussion of the properties of the nearest point
  projection mapping $\pi$). To find $\varepsilon$ one proceeds as follows. Take
  the principal curvatures $\kappa_1,\ldots,\kappa_m$ of $M$. These are
  continuous functions $M \to \R$, because $M$ is a $C^2$ manifold. Next set
  \begin{displaymath}
    \varepsilon := \sup_{x \in M} \max \{ |\kappa_1|, \ldots, |\kappa_m| \} \,.
  \end{displaymath}
  Such maximal value exists due to continuity of $\kappa_j$ for each $j =
  1,\ldots,m$ and compactness of $M$.
  
  We will show that for all $r \le \varepsilon$ and all $x \in \Sigma$ we have
  \begin{equation}
    \label{eq:C2-beta}
    \nbeta(x,r) \le \frac 1{2 \varepsilon} r \,.
  \end{equation}
  Next, we apply Proposition~\ref{prop:beta-curv} and get the desired result.

  Choose $r \in (0,\varepsilon]$. Fix some point $x \in \Sigma$ and pick a
  point $y \in T_{x}M^{\perp}$ with $|x - y| = \varepsilon$. Note that $y$
  belongs to the tubular neighborhood $M_{\varepsilon}$ and that $\pi(y) =
  x$. Hence, the point $x$ is the only point of $M$ in the ball
  $\CBall(y,\varepsilon)$. In other words $M$ lies in the complement of
  $\CBall(y,\varepsilon)$. This is true for any $y$ satisfying $y \in
  T_{x}M^{\perp}$ and $|x - y| = \varepsilon$, so we have
  \begin{displaymath}
    M \subseteq \Rn \setminus \bigcup \left\{
      \CBall(y,\varepsilon) :
      y \perp T_{x}M,\, |y - x| = \varepsilon
    \right\} \,.
  \end{displaymath}
  Pick another point $\bar{x} \in \Sigma \cap \CBall(x,r)$. We then have
  \begin{equation}
    \label{eq:x-pos}
    \bar{x} \in \CBall(x,r) \setminus \bigcup
    \left\{
      \CBall(y,\varepsilon) :
      y \perp T_{x}M,\, |y - x| = \varepsilon
    \right\} \,.
  \end{equation}
  \begin{figure}[!htb]
    \centering
    \includegraphics{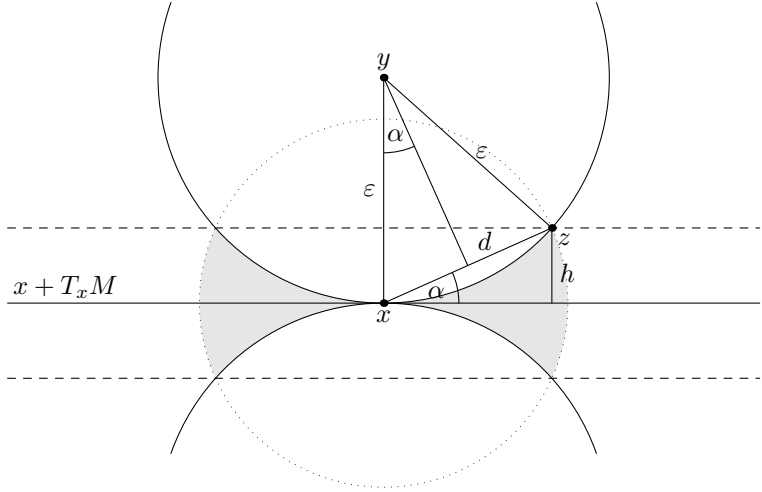}
    \caption{All of $M \cap \CBall(x,r)$ lies in the grey area. The point
      $\bar{x}$ lies in the complement of $\Ball(y,\varepsilon)$ and inside
      $\CBall(x,r)$ so it has to be closer to $T_{x}M$ than $z$.}
    \label{F:height}
  \end{figure}
  Using \eqref{eq:x-pos} and simple trigonometry, it is ease to calculate the
  maximal distance of $\bar{x}$ from the tangent space $T_{x}M$. Let $z$ be any
  point in the intersection $\partial \Ball(x,r) \cap \partial
  \Ball(y,\varepsilon)$. Note that points of $M \cap \CBall(x,\varepsilon)$ must
  be closer to $T_{x}M$ than $z$. In other words
  \begin{equation}
    \label{eq:z-max-dist}
    \forall x \in M \cap \Ball(x,r)
    \quad
    \dist(x,T_{x}M) \le \dist(z,T_{x}M) \,.
  \end{equation}
  This situation is presented on Figure \ref{F:height}. Let $\alpha$ be the
  angle between $T_{x}M$ and $z$ and set $h := \dist(z,T_{x}M)$. We use the fact
  that the distance $|z - x|$ is equal to $r$.
  \begin{equation} \label{eq:h-calc}
    \sin \alpha = \frac{|z - x|}{2 \varepsilon} = \frac{h}{|z - x|}
    \quad \Rightarrow \quad
    h = \frac{|z - x|^2}{2 \varepsilon} = \frac{r^2}{2 \varepsilon} \,.
  \end{equation}
  This proves \eqref{eq:C2-beta} and now we can apply
  Proposition~\ref{prop:beta-curv}.
\end{proof}

\begin{rem}
  Note that the only property of $M$, which allowed us to prove
  Corollary~\ref{cor:C2mani} was the existence of an appropriate tubular
  neighborhood $M_{\varepsilon}$. One can easily see that
  Corollary~\ref{cor:C2mani} still holds if $M$ is just a set of \emph{positive
    reach} as was defined in~\cite{MR0110078}.
\end{rem}

\begin{rem}
  In a forthcoming, joint paper with Marta Szuma{'n}ska~\cite{SKMS}, we prove
  that graphs of a $C^{1,\nu}$ functions have finite integral Menger curvature
  whenever
  \begin{displaymath}
    \nu > \nu_0 := 1 - \frac{m(m+1)}{p}
  \end{displaymath}
  We also construct an example of a $C^{1,\nu_0}$ function such that its graph
  has infinite $p$-energy. This shows that $\nu_0$ is optimal and can not be
  better.
\end{rem}



\mysubsection{Classes of admissible and of fine sets}

In this paragraph we introduce the definitions of two classes of sets. This is
the outcome of the way we worked on this paper. First we proved uniform Ahlfors
regularity (Theorem~\ref{thm:uahlreg}) for the class $\A(\delta,m)$ of
\emph{$(\delta,m)$-admissible} sets. The definition (Definition~\ref{def:adm})
of $\A(\delta,m)$ was based on the definition introduced by Strzelecki and von
der Mosel~\cite[Definition~2.10]{1102.3642} and seemed to be the most
appropriate one for the purpose of the proof of Theorem~\ref{thm:uahlreg}.
However, in the proof of $C^{1,\tau}$ regularity (Theorem~\ref{thm:C1tau}) it is
enough to work with less restrictive conditions, so we introduced the class
$\F(m)$ of \emph{$m$-fine} sets (Definition~\ref{def:fine}). It turns out that
if the $p$-energy of an $m$-dimensional set $\Sigma$ is finite ($\E_p(\Sigma) <
\infty$) for some $p > m(m+2)$ then $\Sigma$ is $(\delta,m)$-admissible if and
only if it is $m$-fine. If we do not assume finiteness of the $p$-energy then
the relation between $\F(m)$ and $\A(\delta,m)$ is not clear. Nevertheless,
starting from a set $\Sigma$ in any of these classes and assuming finiteness of
the $p$-energy we are able to prove $C^{1,\alpha}$ regularity.

\mysubsubsection{Admissible sets}

\begin{defin}
  \label{def:perp}
  Let $H \in G(n,m)$. We say that a sphere $S$ is \emph{perpendicular to $H$}
  if it is of the form $S = \Sphere(x,r) \cap (x + H^{\perp})$ for some $x \in
  H$ and some $r > 0$.
\end{defin}

\begin{defin}
  \label{def:adm}
  Let $\delta \in (0,1)$ and let $I$ be a countable set of indices. Let $\Sigma$
  be a compact subset of $\Rn$. We say that $\Sigma$ is
  $(\delta,m)$-\emph{admissible} and write $\Sigma \in \A(\delta,m)$ if the
  following conditions are satisfied
  \begin{enumerate}[I.]
  \item \textbf{Ahlfors regularity.}
    \label{adm:reg}
    There exist constants $A_{\Sigma} > 0$ and $R_{\Sigma} > 0$ such that for
    any $x \in \Sigma$ and any $r < R_{\Sigma}$ we have
    \begin{equation}
      \label{eq:adm:reg}
      \HM^m(\Sigma \cap \Ball(x,r)) \ge A_{\Sigma} r^m \,.
    \end{equation}
  \item \textbf{Structure.}
    \label{adm:str}
    There exist compact, closed, $m$-dimensional manifolds $M_i$ of class $C^1$
    and continuous maps $f_i : M_i \to \Rn$, $i \in I$, such that
    \begin{equation} \label{eq:adm:str}
      \Sigma = \bigcup_{i \in I} f_i(M_i) \cup Z \,,
    \end{equation}
    where $\HM^m(Z) = 0$.
  \item \textbf{Mock tangent planes and flatness.}
    \label{adm:cone}
    There exists a dense subset $\Sigma^* \subseteq \Sigma$ such that
    \begin{itemize}
    \item $\HM^m(\Sigma \setminus \Sigma^*) = 0$,
    \item for each $x \in \Sigma^*$ there is an $m$-plane $H = H_x \in G(n,m)$
      and a radius $r_0 = r_0(x) > 0$ such that
      \begin{equation} \label{eq:adm:cone}
        |Q_H(y-x)| < \delta |y-x| \quad \text{for each } y \in \Ball(x,r_0) \cap \Sigma \,.
      \end{equation}
    \end{itemize}
  \item \textbf{Linking.}
    \label{adm:link}
    Let $x \in \Sigma^*$ and set $\Slk_x := \Sphere(x,\frac 12 r_0) \cap (x + H_x^{\perp})$.
    Then $\Slk_x$ satisfies
    \begin{equation} \label{eq:adm:link}
      \oplk(\Sigma,\Slk_x) = 1 \,.
    \end{equation}
  \end{enumerate}
\end{defin}

Condition \ref{adm:reg} says that the set $\Sigma$ should be at least
$m$-dimensional. It ensures that $\Sigma$ does not have very long and thin
''fingers''. Intuitively the constant $A_{\Sigma}$ gives a lower bound on the
thickness of any such ''finger''. This means that $\Sigma$ is really
$m$-dimensional and does not behave like a lower dimensional set at any point.

Condition \ref{adm:str} is convenient for the condition \ref{adm:link}. The
degree modulo $2$ was defined for $C^1$-manifolds and continuous mappings so, to
be able to talk about linking number we need to assume \ref{adm:str}. Actually
\ref{adm:str} is a very weak constraint.

Condition \ref{adm:link} says that at each point of $\Sigma$ there is a sphere
$\Slk_x$ which is linked with $\Sigma$. This means intuitively, that we cannot
move $\Slk_x$ far away from $\Sigma$ without tearing one of these sets. Examples
\ref{ex:flat} and \ref{ex:spiral} show that this condition is unavoidable for
the theorems stated in this paper to be true.

Finally, we believe that it is not really necessary to assume a priori that
Condition~\ref{adm:cone} holds. We suspect that if we assume that the $p$-energy
$\E_p(\Sigma)$ (see Definition~\ref{def:p-energy}) is finite for some $p >
m(m+2)$, then condition~\ref{adm:cone} is automatically satisfied. Up to now,
now we were not able to prove this.

\begin{ex}
  \label{ex:mfld}
  Let $\Sigma$ be any closed, compact, $m$-dimensional submanifold of $\Rn$ of
  class $C^1$. Then $\Sigma \in \A(\delta,m)$ for any $\delta \in (0,1)$.

  It is easy to verify that $\Sigma \in \A(\delta,m)$. Take $M_1 = \Sigma$ and
  $f_1 = \opid$. The set $Z$ will be empty, so $\Sigma^* = \Sigma$. At each
  point $x \in \Sigma$ we set $H_x$ to be the tangent space $T_x\Sigma$. Small
  spheres centered at $x \in \Sigma$ and contained in $x + H_x^{\perp}$ are
  linked with $\Sigma$; for the proof see e.g. \cite[pp. 194-195]{MR898008}.
  Note that we do not assume orientability; that is why we used degree modulo
  $2$.
\end{ex}

\begin{ex}
  \label{ex:union-mfld}
  Let $\Sigma = \bigcup_{i=1}^N \Sigma_i$, where $\Sigma_i$ are closed, compact,
  $m$-dimensional submanifolds of $\Rn$ of class $C^1$. Moreover assume that
  these manifolds intersect only on sets of zero $m$-dimensional Hausdorff
  measure, i.e.
  \begin{displaymath}
    \HM^m(\Sigma_i \cap \Sigma_j) = 0 \quad \text{for } i \ne j \,.
  \end{displaymath}
  Then $\Sigma \in \A(\delta,m)$ for any $\delta \in (0,1)$.
\end{ex}

The above examples were taken from \cite{1102.3642}. Now we give some negative
examples showing the role of condition~\ref{adm:link}.

\begin{ex}
  \label{ex:flat}
  Let $H \in G(n,m)$ and let $\Sigma = \pi_H(\Sphere) = \Ball \cap H$. Then
  $\Sigma$ satisfies conditions \ref{adm:reg}, \ref{adm:str} and
  \ref{adm:cone} but it does not satisfy \ref{adm:link}. Hence, it is not
  admissible. Although $\Sigma$ is a compact, $m$-dimensional submanifold of
  $\Rn$ of class $C^1$, it is not closed.
\end{ex}

\begin{ex}
  \label{ex:spiral}
  Let $\gamma : [0,1] \to \R^2$ be defined by
  \begin{displaymath}
    \gamma(t) =
    \left\{
      \begin{array}{lr}
        2^{-2^{1/t}}(\cos \tfrac{\pi}{2t}, \sin \tfrac{\pi}{2t})  & \text{for } t > 0\\
        (0,0) & \text{for } t = 0 \,.
      \end{array}
    \right.
  \end{displaymath}
  We set $\Sigma = \gamma([0,1]) \times [0,1]^{m-1}$. This set satisfies all the
  conditions \ref{adm:reg}, \ref{adm:str} and \ref{adm:cone} but it does not
  satisfy \ref{adm:link}. For the decomposition into a sum $\bigcup f_i(M_i)$ we
  may use a sphere $\Sphere$, then find a continuous mapping $\Sphere
  \to \partial [0,1]^m$, next compose it with the projection $\pi_{\R^m}$ and
  finally compose it with the mapping $(\gamma,\opid) : [0,1]^m \to \R^{m+1}$.
  Set $M_1 = \Sphere$ and set $f_1$ to be the discussed composition.

  This set has the property that for each $r > 0$ there is an $m$-plane $P$ such
  that the distance of any point $x \in \Sigma \cap \Ball(0,r)$ to $P$ is
  approximately $r^2$. Therefore $\Sigma$ gets flatter and flatter when we
  decrease the scale. Using Proposition~\ref{prop:beta-curv} we see that the
  discrete curvature $\K$ is bounded on $\Sigma^{m+2}$ and that $\E_p(\Sigma)$
  is finite for any $p > 0$. This shows that condition \ref{adm:link} is really
  crucial in our considerations.
\end{ex}

\begin{ex}
  \label{ex:decomp}
  Let $\Sigma = \Sphere \cap \R^{m+1}$. Of course $\Sigma$ is admissible as it
  falls into the case presented in Example~\ref{ex:mfld}. We want to emphasize
  that there are good and bad decompositions of $\Sigma$ into the sum $\bigcup
  f_i(M_i)$ from condition \ref{adm:str}.

  The easiest one and the best one is to set $M_1 = \Sigma$ and $f_1 = \opid$.
  But there are other possibilities. Set $M_1 = \Sphere \cap \R^{m+1}$ and $M_2
  = \Sphere \cap \R^{m+1}$ and set
  \begin{align*}
    f_1(x_1,\ldots,x_{m+1}) &:= (x_1,\ldots,x_m,|x_{m+1}|) \,, \\
    f_2(x_1,\ldots,x_{m+1}) &:= (x_1,\ldots,x_m,-|x_{m+1}|) \,,
  \end{align*}
  so that $f_1$ maps $M_1$ to the upper hemisphere and $f_2$ maps $M_2$ to the
  lower hemisphere. This decomposition is bad, because condition \ref{adm:link}
  is not satisfied at any point.
\end{ex}

\mysubsubsection{Fine sets}

Here we introduce the class of $m$-fine sets which captures exactly the
conditions which are needed to prove $C^{1,\tau}$ regularity in
\S\ref{sec:tangent-planes}.
\begin{defin}
  \label{def:fine}
  Let $\Sigma \subseteq \Rn$ be a compact set. We call $\Sigma$ an
  \emph{$m$-fine set} and write $\Sigma \in \F(m)$ if there exist constants
  $A_{\Sigma} > 0$, $R_{\Sigma} > 0$ and $M_{\Sigma} \ge 2$ such that
  \begin{enumerate}[I.]
  \item \textbf{(Ahlfors regularity)}
    \label{fine:reg}
    for all $x \in \Sigma$ and all $r \le R_{\Sigma}$ we have
    \begin{equation}
      \label{eq:fine:reg}
      \HM^m(\Sigma \cap \Ball(x,r)) \ge A_{\Sigma} r^m
    \end{equation}
  \item \textbf{(control of gaps in small scales)}
    \label{fine:gaps}
    and for each $x \in \Sigma$ and each $r \le R_{\Sigma}$ we have
    \begin{displaymath}
      \ntheta(x,r) \le M_{\Sigma} \nbeta(x,r) \,.
    \end{displaymath}
  \end{enumerate}
\end{defin}

\begin{ex}
  Let $M$ be any $m$-dimensional, compact, closed manifold of class $C^1$ and
  let $f : M \to \Rn$ be an immersion. Then the image $\Sigma := \opim(f)$ is an
  $m$-fine set. At each point $x \in M$, there is a radius $R_x$ such that the
  neighborhood $U_x \subseteq f^{-1}(\Ball(f(x),R_x))$ of $x$ in $M$ is mapped
  to the set $V_x := f(U_x) \subseteq \Ball(f(x),R_x)$ and is a graph of some
  Lipschitz function $\Phi_x : Df(x)T_xM \to (Df(x)T_xM)^{\perp}$. If we choose
  $R_x$ small then we can make the Lipschitz constant of $\Phi_x$ smaller than
  some $\varepsilon > 0$. Due to compactness of $M$ and continuity of $Df$ we
  can find a global radius $R_{\Sigma} := \min\{ R_x : x \in M \}$. Then we can
  safely set $A_{\Sigma} = \sqrt{1 - \varepsilon^2}$ and $M_{\Sigma} = 4$.
\end{ex}

Intuitively condition~\ref{fine:gaps} says that $\Sigma$ is ''continuous'' and
has no holes. Consider the case of a unit square in the $2$-plane,
i.e. $\Sigma_0 = \partial [0,1]^2$. Let $\Sigma_1$ be the set obtained from
$\Sigma_0$ by removing some small open interval $J$ from one of the sides of
$\Sigma_0$. Then we have nonempty boundary $\partial \Sigma_1$. For small radii
at the boundary points the $\beta$-numbers will be small and the
$\theta$-numbers will be roughly equal to $\frac 12$. Hence there is no chance
for $\Sigma_1$ to satisfy condition~\ref{fine:gaps}. Note that we can fix that
problem by filling the ''gap'' we made earlier with a complement of some Cantor
set lying inside $\overline{J}$ but then the resulting set $\Sigma_2$ is not
compact. This shows that $m$-fine sets can not be too ''thin'' or too
''sparse''. Nevertheless they can be very ''thick''.
\begin{ex}
  Let $\Sigma$ be the van Koch snowflake in $\R^2$. Then $\Sigma \in \F(1)$ but
  it fails to be $1$-dimensional.
\end{ex}

\begin{ex}
  Let $m = 1$, $n = 2$ and
  \begin{displaymath}
    \Sigma = \bigcup_{k=1}^{\infty} (-Q_k) \cup \left\{
      (t,0) \in \R^2 : t \in [-1,1]
    \right\} \cup \bigcup_{k=1}^{\infty} Q_k \,,
  \end{displaymath}
  where
  \begin{displaymath}
    Q_0 = \partial \big([0,1] \times [0,1]\big)
    \quad \text{and} \quad
    Q_k = \Big( \sum_{j=1}^k 2^{-j}, -\tfrac 12 \Big) + 2^{-(k+1)} Q_0 \,.
  \end{displaymath}
  See Figure~\ref{F:oscillation} for a graphical presentation. Condition
  \ref{fine:gaps} holds at the boundary points $(-1,0)$ and $(1,0)$ of $\Sigma$,
  because the $\beta$-numbers do not converge to zero with $r \to 0$ at these
  points. All the other points of $\Sigma$ are internal points of line segments
  or corner points of squares, so at these points conditions \ref{fine:reg} and
  \ref{fine:gaps} are also satisfied. Hence, $\Sigma$ belongs to the class
  $\F(1)$.

  \begin{figure}[!htb]
    \centering
    \includegraphics{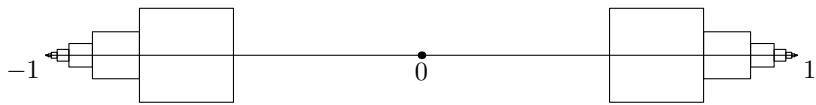}
    \caption{This set is $1$-fine despite the fact that it has boundary points.}
    \label{F:oscillation}
  \end{figure}

  This example shows that condition \ref{fine:gaps} does not exclude boundary
  points but at any such boundary point we have to add some oscillation, to
  prevent $\beta$-numbers from getting too small. The same effect can be
  observed in the following example
  \begin{displaymath}
    \Sigma = \partial \big([1,2] \times [-1,1]\big) \cup
    \overline{ \left\{ (x, x \sin( \tfrac 1x)) : x \in (0,1] \right\} } \,.
  \end{displaymath}
\end{ex}




\mysection{Uniform Ahlfors regularity}
\label{sec:ahl-reg}

In this paragraph, after introducing all the preparatory material we are ready
to prove our first important result:
\begin{thm}
  \label{thm:uahlreg}
  Let $E < \infty$ be some positive constant and let $\Sigma \in \A(\delta,m)$
  be an admissible set, such that $\E_p(\Sigma) \le E$ for some $p >
  m(m+2)$. There exist two constants $\Cl{uahlreg1} = \Cr{uahlreg1}(\delta,m)$
  and $\Cl{uahlreg2} = \Cr{uahlreg2}(\delta,m)$ and a radius
  \begin{displaymath}
    \Cl[R]{uar-rad} = \Cr{uar-rad}(E,p,m,\delta) 
    := \left( \frac{\Cr{uahlreg1} \Cr{uahlreg2}^p}{E} \right)^{\frac 1{p-m(m+2)}} \,,
  \end{displaymath}
  such that for each $\rho \le \Cr{uar-rad}$ and each $x \in \Sigma$ we have
  \begin{displaymath}
    \HM^m(\Sigma \cap \Ball(x,\rho)) \ge (1 - \delta^2)^{\frac m2} \omega_m \rho^m \,.
  \end{displaymath}
\end{thm}

\begin{cor}
  \label{cor:ind-const}
  If $\Sigma \in \A(\delta,m)$ with some constants $A_{\Sigma}$ and $R_{\Sigma}$
  and if $\E_p(\Sigma) \le E < \infty$ for some $p > m(m+2)$, then $\Sigma \in
  \A(\delta,m)$ with constants $R_{\Sigma}' := \Cr{uar-rad}$ and $A_{\Sigma}' :=
  (1-\delta^2)^{m/2} \omega_m$, which depend only on $E$, $m$, $p$ and $\delta$.
\end{cor}

In other words we claim that a bound on the $p$-energy implies uniform Ahlfors
regularity below some fixed scale. This means that whenever $\Sigma$ has
$p$-energy lower than $E$, then it cannot have very long and very thin
''tentacles'' in that scale. The thickness of any such ''tentacle'' is bounded
from below by a constant depending only on $E$. Another way to understand this
result is the intuition that $\Sigma$ has to really be $m$-dimensional when we
look at it in small scales. At large scales one can see some very thin
''antennas'', which look like lower dimensional objects, but looking closer he
or she will see that these ''antennas'' are really thick tubes. The scale at
which we have to look depends only on the $p$-energy.

\mysubsection{Bounded energy and flatness}
\label{sec:bdd-ene-flat}

\begin{prop}
  \label{prop:eta-d-balance}
  Let $\Sigma \subseteq \Rn$ be some $m$-Ahlfors regular, $\HM^m$-measurable
  set, meaning that there exist constants $A_{\Sigma} > 0$ and $R_{\Sigma} > 0$
  such that for all $x \in \Sigma$ and all $r \in (0,R_{\Sigma})$
  \begin{displaymath}
    \HM^m(\Sigma \cap \Ball(x,r)) \ge A_{\Sigma} r^m \,,
  \end{displaymath}
  Assume that $\E_p(\Sigma) \le E < \infty$ for some $p > m(m+2)$. Furthermore,
  assume that there exists a simplex $T_0 = \simp(x_0,\ldots,x_{m+1})$ with
  vertices on $\Sigma$ and such that $T_0 \in \Reg_m(\eta,d)$ for some $d \le
  R_{\Sigma} / \varsigma_m$. Then $\eta$ and $d$ must satisfy
  \begin{equation}
    \label{est:eta-d}
    d \ge \left( \frac{\Cl{eta-d1}\Cl{eta-d2}^p A_{\Sigma}^{m+2}}{E} \right)^{1/\lambda} \eta^{\kappa/\lambda}
    \quad \text{or equivalently} \quad
    \eta \le \left( \frac{E}{\Cr{eta-d1}\Cr{eta-d2}^p A_{\Sigma}^{m+2}} \right)^{1/\kappa} d^{\lambda/\kappa} \,,
  \end{equation}
  where
  \begin{align*}
    \lambda &= \lambda(m,p) := p - m(m+2) \,, & 
    \kappa &= \kappa(m,p) := (m+1)(m(m+1)(m+2)+p) \,, \\
    \Cr{eta-d2} &= \Cr{eta-d2}(m) 
    := \frac{1}{(m+1)2^{m+2}} \,, &
    \Cr{eta-d1} &= \Cr{eta-d1}(m) 
    := \left( \frac{\Cr{eta-const}}{2 \Upsilon(m) \Omega^{m+2} m!} \right)^{m(m+2)} \,,
  \end{align*}
  $\Upsilon(m)$ is a constant defined by~\eqref{def:upsilon} and $\Omega$ is
  defined by \eqref{def:Omega}.
\end{prop}

\begin{proof}
  We shall estimate the $p$-energy of $\Sigma$. Let $\varsigma_m$ be defined by
  \eqref{def:alpha}.
  \begin{multline}
    \label{eq:energy-low}
    \infty > E \ge \E_p(\Sigma) = \int_{\Sigma^{m+2}} \K^p(T)\ d\mu(T) \\
    \ge \int_{\Sigma \cap \Ball(x_0,\varsigma_m d)} \cdots \int_{\Sigma \cap \Ball(x_{m+1},\varsigma_m d)}
    \K^p(\simp(y_0,\ldots,y_{m+1}))\ d\HM^m_{y_0} \ldots\ d\HM^m_{y_{m+1}} \,.
  \end{multline}
  Proposition~\ref{prop:perturb} combined with Fact~\ref{fact:reg-curv} lets us
  estimate the integrand
  \begin{displaymath}
    \K^p(\simp(y_0,\ldots,y_{m+1})) 
    \ge \left( \frac{\eta^{m+1}}{(m+1) 2^{m+2}d} \right)^p \,.
  \end{displaymath}
  From the $m$-Ahlfors regularity of $\Sigma$, we get a lower bound on the measure
  of the sets over which we integrate
  \begin{displaymath}
    \HM^m(\Sigma \cap \Ball(x_i,\varsigma_m d)) \ge A_{\Sigma} (\varsigma_m d)^m \,.
  \end{displaymath}
  Plugging the last two estimates into \eqref{eq:energy-low} we obtain
  \begin{displaymath}
    E \ge ( A_{\Sigma} (\varsigma_m d)^m )^{m+2} \left( \frac{\eta^{m+1}}{(m+1) 2^{m+2}d} \right)^p 
    = \Cr{eta-d2}(m)^p \frac{A_{\Sigma}^{m+2}}{d^{p - m(m+2)}}  \varsigma_m^{m(m+2)} \eta^{p(m+1)} \,.
  \end{displaymath}
  Recalling \eqref{eq:alpha-eta} we get
  \begin{displaymath}
    E \ge  \Cr{eta-d1}(m) \Cr{eta-d2}(m)^p \frac{A_{\Sigma}^{m+2}}{d^{p - m(m+2)}} \eta^{(m+1)(m(m+1)(m+2) + p)} \,,
  \end{displaymath}
  which gives us the balance condition
  \begin{displaymath}
    d^{p - m(m+2)} E \ge \Cr{eta-d1}(m) \Cr{eta-d2}(m)^p A_{\Sigma}^{m+2} \eta^{(m+1)(m(m+1)(m+2) + p)}\,.
  \end{displaymath}
  Inequalities \eqref{est:eta-d} and \eqref{est:eta-d} now follow.
\end{proof}

This lemma is interesting in itself. It says that whenever the energy of
$\Sigma$ is finite, we cannot have very small and voluminous simplices with
vertices on $\Sigma$. It gives a bound on the ''regularity'' (i.e. parameter
$\eta$) of any simplex in terms of its diameter $d$ and we see that $\eta$ goes
to $0$ when we decrease $d$. Now we shall prove that an upper bound on $\eta$
imposes an upper bound on the Jones' $\beta$-numbers.

\begin{cor}
  \label{cor:beta-est}
  Let $\Sigma \subseteq \Rn$ be as in Proposition~\ref{prop:eta-d-balance}. Then
  there exists a constant $\Cl{beta-est} = \Cr{beta-est}(m,p,A_{\Sigma})$ such
  that for any $x \in \Sigma$ and any $r \in (0,R_{\Sigma})$ we have
  \begin{displaymath}
    \nbeta(x,r) \le \Cr{beta-est} E^{\frac 1{\kappa}} r^{\tau} \,,
  \end{displaymath}
  where
  \begin{align}
    \label{def:tau}
    \tau = \frac{\lambda}{\kappa} = \frac{p - m(m+2)}{(m+1)(m(m+1)(m+2)+p)} \in (0,1) \,.
  \end{align}
\end{cor}

\begin{proof}
  Fix some point $x \in \Sigma$ and a radius $r \in (0,R_{\Sigma})$. Let $T
  = \simp(x_0, \ldots, x_{m+1})$ be an $(m+1)$-simplex such that $x_i \in
  \Sigma \cap \CBall(x,r)$ for $i = 0,1,\ldots,m+1$ and such that
  $T$ has maximal $\HM^{m+1}$-measure among all simplices with vertices in
  $\Sigma \cap \CBall(x,r)$.
  \begin{displaymath}
    \HM^{m+1}(T) = \max\{ \HM^{m+1}(\simp(x_0',\ldots,x_{m+1}')) : x_i' \in \Sigma \cap \CBall(x,r) \} \,.
  \end{displaymath}
  The existence of such simplex follows from the fact that the set $\Sigma
  \cap \CBall(x,r)$ is compact and from the fact that the function
  $T \mapsto \HM^{m+1}(T)$ is continuous with respect to $x_0$, \ldots,
  $x_{m+1}$.

  Rearranging the vertices of $T$ we can assume that $\hmin(T) =
  \height_{m+1}(T)$, so the largest $m$-face of $T$ is
  $\simp(x_0,\ldots,x_m)$. Let $H = \opspan\{ x_1-x_0, \ldots, x_m-x_0 \}$, so
  that $x_0 + H$ contains the largest $m$-face of $T$. Note that the distance
  of any point $y \in \Sigma \cap \CBall(x,r)$ from the affine plane
  $x_0 + H$ has to be less then or equal to $\hmin(T) = \dist(x_{m+1},x_0+H)$.
  If we could find a point $y \in \Sigma \cap \CBall(x,r)$ with
  $\dist(y,x_0+H) > \hmin(T)$, than the simplex $\simp(x_0,\ldots,x_m,y)$
  would have larger $\HM^{m+1}$-measure than $T$ but this is impossible due to
  the choice of $T$.

  Since $x \in \Sigma \cap \CBall(x,r)$, we know that
  $\dist(x,x_0+H) \le \hmin(T)$, so we obtain
  \begin{equation}
    \label{est:beta-hmin}
    \forall y \in \Sigma \cap \CBall(x,r) \quad
    \dist(y,x+H) \le 2 \hmin(T) \,.
  \end{equation}
  Now we only need to estimate $\hmin(T) = \height_{m+1}(T)$ from above. We have
  (cf. Remark~\ref{rem:hmin-estimates}) $\HM^m(\face_{m+1}T) \ge \frac 1{m!}
  \hmin(T)^m$, hence
  \begin{displaymath}
    T \in \Reg_m \left( \tfrac{\hmin(T)}{r \sqrt[m]{m!}}, r \right) \,.
  \end{displaymath}
  Let $\eta = \tfrac{\hmin(T)}{r \sqrt[m]{m!}}$. From
  Proposition~\ref{prop:eta-d-balance} we know that $\eta \le \eta_0$, so
  we obtain
  \begin{equation}
    \label{est:hmin2}
    \frac{\hmin(T)}{r \sqrt[m]{m!}} \le \eta_0
    \quad \Rightarrow \quad
    \hmin(T) \le \frac{\eta_0}{\sqrt[m]{m!}} r \,.
  \end{equation}

  Estimates \eqref{est:beta-hmin} and \eqref{est:hmin2} immediately give us an
  upper bound on the $\beta$-numbers
  \begin{displaymath}
    \nbeta(x,r) \le \frac{2\eta_0}{\sqrt[m]{m!}} 
    = \frac 2{\sqrt[m]{m!}} \left(
      \frac{E}{\Cr{eta-d1}\Cr{eta-d2}^p A_{\Sigma}^{m+2}}
    \right)^{1/\kappa} r^{\lambda/\kappa} 
    =: \Cr{beta-est} E^{\frac 1{\kappa}} r^{\lambda/\kappa}\,.
  \end{displaymath}
\end{proof}


\mysubsection{Proof of Theorem~\ref{thm:uahlreg}}

The proof of Theorem~\ref{thm:uahlreg} has several steps. The whole idea was
taken from the paper of Strzelecki and von der Mosel~\cite{0911.2095}. We repeat
the same steps but in greater generality. Paradoxically, when working in a more
abstract setting we were able to simplify things. The crucial part is
Proposition~\ref{prop:big-proj-fat-simp} which allows us to find
$(\eta,d(x_0))$-voluminous simplices with vertices on $\Sigma$ at a scale
$d(x_0)$ which may vary depending on the choice of the first vertex. It is an
analogue of~\cite[Theorem~3.3]{0911.2095} and the proof rests on an algorithm
quite similar to the one described by Strzelecki and von der Mosel but it
considers only two cases and clearly exposes the essential difficulty of the
reasoning.

Earlier we proved Proposition~\ref{prop:eta-d-balance} which gives us a balance
condition between $\eta$ and $d$. The fact that $\eta$ from
Proposition~\ref{prop:big-proj-fat-simp} depends only on $\delta$ and $m$ and
does not depend on $x_0$ lets us prove (Proposition~\ref{prop:dSigma-lower})
that there is a lower bound $\Cr{uar-rad}$ for $d(x_0)$ which depends only on
the $p$-energy. The reasoning used here mimics the proof
of~\cite[Proposition~3.5]{0911.2095}.

Besides the existence of good simplices Proposition~\ref{prop:big-proj-fat-simp}
ensures also that at any scale below $d(x_0)$ our set $\Sigma$ has big
projection onto some affine $m$-plane. This immediately gives us Ahlfors
regularity below the scale $d(x_0)$. Now, since we have a lower bound $d(x_0)
\ge \Cr{uar-rad}$ and $\Cr{uar-rad}$ does not depend on the choice of $x_0$, we
obtain the desired result. All this is proven for $x_0 \in \Sigma^*$, so the
final step (Proposition~\ref{prop:big-proj-all-points}) is to show that it works
for any other point $x_0 \in \Sigma \setminus \Sigma^*$ but this is easily done
by passing to a limit. The proof is basically the same as the proof
of~\cite[Proposition~3.4]{0911.2095}.

Proposition~\ref{prop:big-proj-fat-simp} is proved by defining an algorithmic
procedure. We start by choosing some point $x_0 \in \Sigma^*$. From the
definition of an admissible set we know that we can touch $\Sigma$ by some cone
$x_0 + \Cone(\delta,H_0)$ and that there are no points of $\Sigma \cap
\Ball(x_0,\rho_0)$ inside this cone for small $\rho_0$. We increase the radius
$\rho_0$ until we hit $\Sigma$. Condition~\ref{adm:link} of the
Definition~\ref{def:adm} ensures that we can choose a well spread $m$-tuple of
points in $\Sigma \cap \Ball(x_0,\rho_0)$. We do that just by choosing $m$
points $y_1$, \ldots, $y_m$ on $\partial \Ball(x_0,\sqrt{1-\delta^2}\rho_0)$
such that the vectors $(y_1 - x_0)$, \ldots, $(y_m - x_0)$ form an orthogonal
basis of $H_0$ - this is what we mean be a ,,well spread tuple of points''. Then
we use Lemma~\ref{lem:intpoint} to find appropriate points $x_i \in \Sigma \cap
\Ball(x_0,\rho_0)$ for $i = 1,2, \ldots, m$. The points $x_0$, $x_1$, \ldots,
$x_m$ span some $m$-plane $P$. Now, we stop and analyze the situation. There are
two possibilities. Either we can find a point of $\Sigma$ far from $P$ at scale
comparable to $\rho_0$, or $\Sigma$ is almost flat at scale $\rho_0$ which means
that it is very close to $P$. In the first case we can stop, since we have found
a good simplex. In the second case we need to continue. We set $H_1 := P$ and
repeat the procedure but now we consider not the set $\Cone(\delta,H_1) \cap
\Ball(x_0,\rho_1)$ but only the conical cap $\Cone(\delta,H_1,\tfrac 12
\rho_0,\rho_1)$. From the fact that $\Sigma$ is close to $H_1 = P$ at scale
$\rho_0$ we deduce that $\Cone(\delta,H_1,\tfrac 12 \rho_0,\rho_1)$ does not
intersect $\Sigma$ for $\rho_1 \le 2 \rho_0$. We increase $\rho_1$ until we hit
$\Sigma$ and iterate the whole algorithm.

In the course of the proof we build an increasing sequence of sets $F_i$ made up
from the conical caps $\Cone(\delta,H_i,\tfrac 12 \rho_{i-1}, \rho_i)$. For each
$i$ the set $F_i$ does not intersect $\Sigma$, it contains the conical cap
$\Cone(\delta,H_i,\tfrac 12 \rho_{i-1}, 2\rho_{i-1})$ and appropriate spheres
contained in $F_i$ are linked with $\Sigma$. Using these properties of $F_i$ and
using Lemma~\ref{lem:intpoint} we obtain big projections of $\Sigma \cap
\Ball(x_0,\rho_i)$ onto $H_i$ for each $i$. The idea to use the linking number
and to construct continuous deformations of spheres inside conical caps comes
from~\cite{1102.3642}.

\begin{prop}
  \label{prop:big-proj-fat-simp}
  Let $\delta \in (0,1)$ and $\Sigma \in \A(m,\delta)$ be an admissible
  set in $\Rn$. There exists a real number $\eta = \eta(\delta,m) > 0$ such that
  for every point $x_0 \in \Sigma^*$ there is a stopping distance $d = d(x_0) >
  0$, and a $(m+1)$-tuple of points $(x_1, x_2, \ldots, x_{m+1}) \in
  \Sigma^{m+1}$ such that
  \begin{displaymath}
    T = \simp(x_0, \ldots, x_{m+1}) \in \Reg_m(\eta, d) \,.
  \end{displaymath}
  Moreover, for all $\rho \in (0, d)$ there exists an $m$-dimensional subspace
  $H = H(\rho) \in G(n,m)$ with the property
  \begin{equation}
    \label{eq:big-proj}
    (x_0 + H) \cap \Ball(x_0, \sqrt{1 - \delta^2} \rho)
    \subseteq \pi_{x_0 + H}(\Sigma \cap \Ball(x_0, \rho)) \,.
  \end{equation}
\end{prop}

\begin{cor}
  \label{cor:uahlreg}
  For any $x_0 \in \Sigma^*$ and any $\rho \le d(x_0)$ we have
  \begin{equation}
    \label{eq:uahlreg}
    \HM^m(\Sigma \cap \Ball(x_0,\rho)) \ge (1 - \delta^2)^{\frac m2} \omega_m \rho^m \,.
  \end{equation}
\end{cor}

\begin{proof}
  Orthogonal projections are Lipschitz mappings with constant $1$ so they cannot
  increase the measure. From~\eqref{eq:big-proj} we know that the image of
  $\Sigma \cap \Ball(x_0, \rho)$ under $\pi_{x_0 + H}$ contains the ball $(x_0 +
  H) \cap \Ball(x_0, \sqrt{1 - \delta^2} \rho)$. The measure of that ball is $(1
  - \delta^2)^{\frac m2} \omega_m \rho^m$, hence the
  inequality~\eqref{eq:uahlreg}.
\end{proof}

\begin{proof}[Proof of Proposition~\ref{prop:big-proj-fat-simp}]
  Without loss of generality we can assume that $x_0 = 0$ is the origin. To
  prove the proposition we will construct finite sequences of
  \begin{itemize}
  \item compact, connected, centrally symmetric sets $F_0 \subseteq F_1 \subseteq \ldots \subseteq F_N$,
  \item $m$-dimensional subspaces $H_i \subseteq \Rn$ for $i = 0, 1, \ldots, N$,
  \item and of radii $\rho_0 < \rho_1 < \cdots < \rho_N$.
  \end{itemize}
  For brevity, we define 
  \begin{displaymath}
    r_i := \sqrt{1 - \delta^2} \rho_i \,.
  \end{displaymath}
  The above sequences will satisfy the following conditions
  \begin{itemize}
  \item the interior of $F_i$ is disjoint with $\Sigma$
    \begin{equation}
      \label{cond:disj}
      \Sigma \cap \opint F_i = \emptyset \,,
    \end{equation}
  \item the radii grow geometrically, i.e.
    \begin{equation}
      \rho_{i+1} \ge 2 \rho_i \,, \label{cond:growth}
    \end{equation}
  \item each $F_i$ contains a large conical cap
    \begin{equation}
      \label{cond:cone}
      \Cone(\delta,H_{i+1},\tfrac 12 \rho_i,\rho_{i+1}) \subseteq F_{i+1} \,,
    \end{equation}
  \item all spheres $S$ centered at $H_i \cap \Ball_{r_i}$, perpendicular to $H_i$ and
    contained in $F_i$ are linked with $\Sigma$
    \begin{equation}
      \label{cond:link}
      \forall\, x \in H_i \cap \Ball_{r_i} \;
      \forall\, s > 0 \;
      \left(
        S := \Sphere(x,s) \cap (x + H_i^{\perp}) \subseteq F_i
        \quad \Rightarrow \quad
        \oplk(\Sigma, S) = 1
      \right) .
    \end{equation}
  \end{itemize}

  Let us define the first elements of these sequences. We set $F_0 := \emptyset$,
  $H_0 := H_1 := H_{x_0}$ and $\rho_0 := 0$. Let
  \begin{align*}
    \rho_1 &:= \inf \{ s > 0 : \Cone(\delta,H_0,0,s) \cap \Sigma \ne \emptyset \} \,, \\
    F_1 &:= \Cone(\delta,H_1,0,\rho_1) \,.
  \end{align*}
  Directly from the definition of an admissible set, we know that $\rho_1 > 0$, so
  the condition \eqref{cond:growth} is satisfied for $i = 0$. Conditions
  \eqref{cond:disj} and \eqref{cond:cone} are immediate for $i = 0$. Using
  Proposition~\ref{prop:sph-trans} one can deform any sphere $S$ from condition
  \eqref{cond:link} to the sphere $\Slk_x$ defined in \ref{adm:link} of the
  definition of $\A(\delta,m)$. This shows that \eqref{cond:link} is satisfied for
  $i = 0$.

  We proceed by induction. Assume we have already defined the sets $F_i$,
  subspaces $H_i$ and radii $\rho_i$ for $i = 0,1, \ldots, I$. Now, we will show
  how to continue the construction.

  Let $(e_1,e_2, \ldots, e_m)$ be an orthonormal basis of $H_I$. We choose $m$
  points lying on $\Sigma$ such that
  \begin{displaymath}
    x_i \in \Sigma \cap \Ball(r_I e_i, \delta \rho_I) \cap (H_I^{\perp} + r_I e_i) \,.
  \end{displaymath}
  In particular
  \begin{equation}
    \label{eq:good-ball}
    x_i \in \Ball(x_0, 2\rho_I)
    \quad \text{for} \quad
    i \in \{0,1,\ldots,m\}\,.
  \end{equation}
  Condition \eqref{cond:link} tells us that such points exist. The $m$-simplex $R
  := \simp(x_0, x_1, \ldots, x_m)$ will be the base of our $(m+1)$-simplex $T$. Note,
  that when we project $R$ onto $H_I$ we get the simplex
  \begin{displaymath}
    \pi_{H_I}(R) = \simp(0, r_I e_1, r_I e_2, \ldots, r_I e_m) \,.
  \end{displaymath}
  Since $\pi_{H_I}$ is a Lipschitz mapping with constant $1$, we can estimate the
  measure of $R$ as follows
  \begin{equation}
    \label{eq:good-base}
    \HM^m(R) \ge \HM^m(\pi_{H_I}(R)) 
    = \frac{1}{m!} r_I^m 
    = \frac{(\sqrt{1 - \delta^2})^m}{2^m m!} (2\rho_I)^m \,.
  \end{equation}
  This shows that the conditions \eqref{fat:ball} and \eqref{fat:base} of the
  definition of the class $\Reg_m(\tilde{\eta},2\rho_I)$ are satisfied with 
  \begin{displaymath}
    \tilde{\eta} := \frac{\sqrt{1 - \delta^2}}{2\sqrt[m]{m!}} \,.
  \end{displaymath}

  Recall that $x_0 = 0$. Let $P$ be the subspace spanned by $\{x_i\}_{i=1}^m$,
  i.e.
  \begin{displaymath}
    P := \opspan\{ x_1, x_2, \ldots, x_m \} \,.
  \end{displaymath}
  We need to find one more point $x_{m+1} \in \Sigma$ such that the distance
  $\dist(x_{m+1}, P) \ge \eta \rho_I$ for some positive $\eta = \eta(\delta,m) \le
  \tilde{\eta}$.

  Choose a small positive number $h_0 = h_0(\delta) \le \frac 12$ such that
  \begin{equation}
    \label{def:h0choice}
    \delta + 2h_0\delta \le (1 - 2h_0\delta) \sqrt{1 - (2h_0\delta)^2} \,.
  \end{equation}
  This is always possible because when we decrease $h_0$ to $0$ the left-hand
  side of \eqref{def:h0choice} converges to $\delta < 1$ and the right-hand side
  converges to $1$. We need this condition to be able to apply
  Proposition~\ref{prop:two-cones} later on.
  \begin{rem}
    \label{rem:delta14}
    Note that if $\delta \le \frac 14$, we can set $h_0 := \frac 12$ because
    then
    \begin{align*}
      \delta + 2 h_0 \delta &\le \tfrac 12 \\
      \text{and} \quad 
      (1 - 2 h_0 \delta) \sqrt{1 - (2 h_0 \delta)^2}
      &\ge \tfrac 34 \tfrac{\sqrt{15}}{16} \ge \tfrac{9}{16} \,.
    \end{align*}
  \end{rem}

  There are two possibilities (see Figure~\ref{F:casesAB})
  \begin{enumerate}[(A)]
  \item \label{pos:good} there exists a point $x_{m+1} \in \Sigma \cap
    \Shell(\tfrac 12 \rho_I,2 \rho_I)$ such that
    \begin{displaymath}
      \dist(x_{m+1},P) \ge h_0 \delta \rho_I \,,
    \end{displaymath}
  \item \label{pos:flat} $\Sigma$ is contained in a small neighborhood of $P$, i.e.
    \begin{displaymath} 
      \Sigma \cap \Shell(\tfrac 12 \rho_I,2 \rho_I) \subseteq  P + \Ball_{h_0 \delta \rho_I}\,.
    \end{displaymath}
  \end{enumerate}

  \begin{figure}[!htb]
    \centering
    \includegraphics{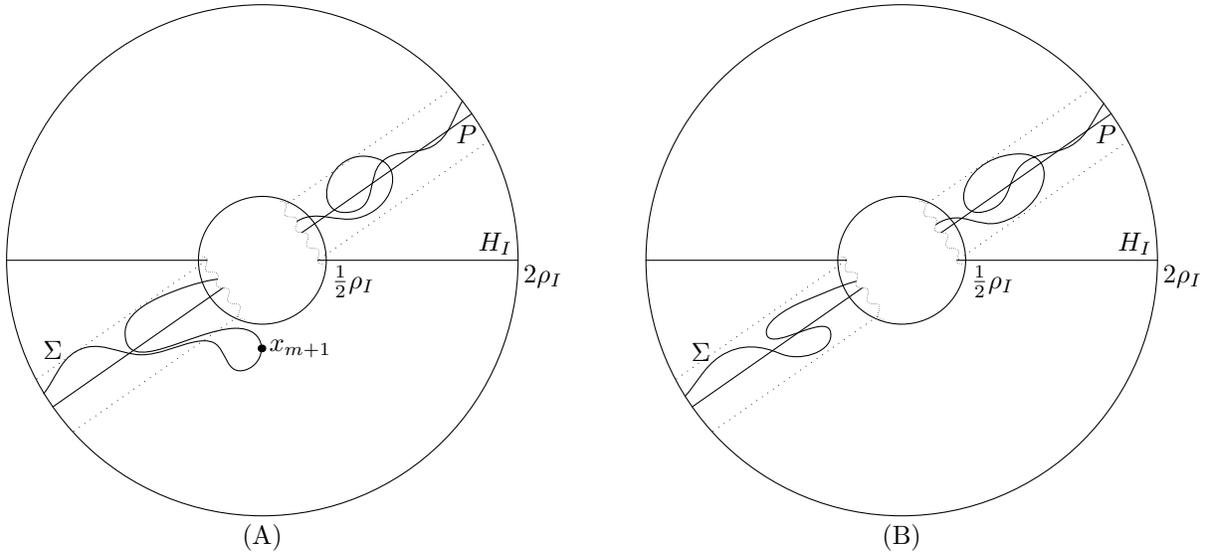}
    \caption{The two possible configurations.}
    \label{F:casesAB}
  \end{figure}

  If case \eqref{pos:good} occurs, then we can end our construction immediately.
  The point $x_{m+1}$ satisfies
  \begin{itemize}
  \item $x_{m+1} \in \Ball(x_0, 2 \rho_I)$,
  \item $\dist(x_{m+1}, P) \ge (\tfrac 12 h_0 \delta) (2 \rho_I)$.
  \end{itemize}
  We may set
  \begin{align}
    N &:= I \,, &
    \eta &:= \min \left\{ \tilde{\eta}, \tfrac 12 h_0 \delta \right\} 
    = \min \left\{ \frac{\sqrt{1 - \delta^2}}{2\sqrt[m]{m!}}, \frac{h_0 \delta}{2} \right\} \,, \label{def:eta} \\
    d &= d(x_0) := 2 \rho_I \quad \text{and}&
    T &:= \simp(x_0, \ldots, x_{m+1}) \,. \notag
  \end{align}
  Using \eqref{eq:good-ball} and \eqref{eq:good-base} we get $T \in
  \Reg_m(\eta,d)$.

  If case \eqref{pos:flat} occurs, then our set $\Sigma$ is almost flat in
  $\Shell(\tfrac 12 \rho_I,2 \rho_I)$ so there is no chance of finding a
  voluminous simplex in this scale and we have to continue our construction. Let
  \begin{itemize}
  \item $H_{I+1} := P$,
  \item $\rho_{I+1} := \inf \{ s > \rho_I : \Cone(\delta,P,\rho_I,s) \cap \Sigma \ne \emptyset \}$ and
  \item $F_{I+1} := F_I \cup \Cone(\delta,P,\tfrac 12 \rho_I,\rho_{I+1})$.
  \end{itemize}
  We assumed \eqref{pos:flat}, so it follows that
  \begin{equation}
    \label{eq:S-in-cone}
    \forall x \in \Sigma \cap \Shell(\tfrac 12 \rho_I,2 \rho_I)
    \quad
    |Q_P(x)| \le h_0 \delta \rho_I \le 2 h_0 \delta |x| < \delta |x| \,.
  \end{equation}
  This means that $\Cone(\delta,P,\tfrac 12 \rho_I,2\rho_I)$ does not intersect
  $\Sigma$ and we can safely set $H_{I+1} := P$. It is immediate that $\rho_{I+1}
  \ge 2 \rho_I$ so conditions \eqref{cond:disj}, \eqref{cond:growth} and
  \eqref{cond:cone} are satisfied. Now, the only thing left is to verify condition
  \eqref{cond:link}.

  We are going to show that all spheres $S$ contained in $F_{I+1}$ of the form 
  \begin{displaymath}
    S = \Sphere(x,r) \cap (x + P^{\perp}) \,,
    \quad \text{for some } x \in P \cap \Ball_{r_{I+1}}
  \end{displaymath}
  are linked with $\Sigma$. By the inductive assumption, we already know that
  spheres centered at $H_I \cap \Ball_{r_I}$, perpendicular to $H_I$ and
  contained in $F_I$ are linked with $\Sigma$. Therefore, all we need to do is
  to continuously deform $S$ to an appropriate sphere centered at $H_I$ and
  contained in $F_I$ in such a way that we never leave the set $F_{I+1}$ (see
  Figure~\ref{F:htp-in-F}).

  \begin{figure}[!htb]
    \centering
    \includegraphics{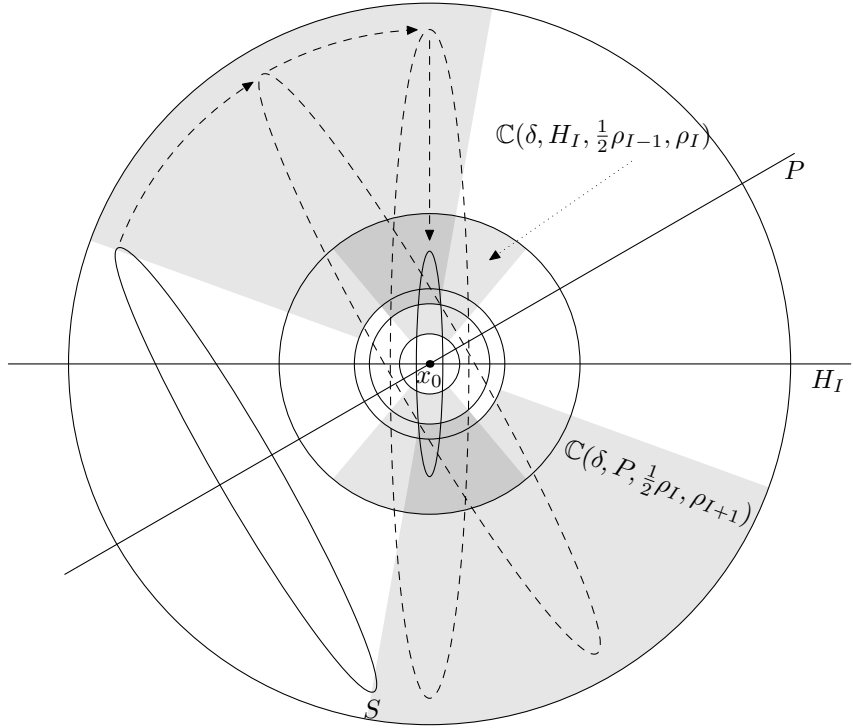}
    \caption{First we move the center of $S$ to $x_0$. Then we rotate $S$ so that
      it is perpendicular to $H_I$. Finally we change the radius so that it is
      between $\frac 12 \rho_{I-1}$ and $\rho_I$.}
    \label{F:htp-in-F}
  \end{figure}

  We know that $F_{I+1}$ contains the conical cap $CC := \Cone(\delta,P,\tfrac 12
  \rho_I,\rho_{I+1})$, so we can use Proposition~\ref{prop:sph-trans} to move
  $S$ inside $CC$, so that it is centered at the origin.

  From \eqref{eq:S-in-cone} we get
  \begin{displaymath}
    \Sigma \cap \Shell(\tfrac 12 \rho_I,2 \rho_I)
    \subseteq
    \Rn \setminus \Cone(2 h_0 \delta, P)
    \subseteq
    \Cone(\sqrt{1 - (2 h_0 \delta)^2}, P^{\perp}) \,.
  \end{displaymath}
  Using this and our inductive assumption we obtain
  \begin{displaymath}
    \Sigma \cap \Shell(\tfrac 12 \rho_I,\rho_I)
    \subseteq
    \Cone(\sqrt{1 - \delta^2},H_I^{\perp}) \cap \Cone(\sqrt{1 - (2 h_0 \delta)^2},P^{\perp}) \,.
  \end{displaymath}

  We have two cones that have nonempty intersection and we chose $h_0$ such
  that~\eqref{def:h0choice} holds, so we can apply
  Proposition~\ref{prop:two-cones} with $\alpha = \delta$ and $\beta = 2 h_0
  \delta$. Hence the intersection $\Cone(\delta,H_I) \cap \Cone(\delta,P)$
  contains the space $H_I^{\perp}$. Therefore
  \begin{displaymath}
    H_I^{\perp} \cap \Shell(\tfrac 12 \rho_I,\rho_{I+1})
    \subseteq
    \Cone(\delta,P,\tfrac 12 \rho_I,\rho_{I+1}) \cap F_I \,.
  \end{displaymath}
  Using Corollary~\ref{cor:sph-in-cone} we can rotate $S$ inside $CC$, so that
  it lies in $H^{\perp}$. Then we decrease the radius of $S$ to the value
  e.g. $\tfrac 34 \rho_I \in (\frac 12 \rho_{I-1},\rho_I)$. Applying the
  inductive assumption we obtain condition \eqref{cond:link} for $i = I+1$.

  The set $\Sigma$ is compact and $\rho_i$ grows geometrically, so our
  construction has to end eventually. Otherwise we would find arbitrary large
  spheres, which are linked with $\Sigma$ but this contradicts compactness.
\end{proof}

\begin{prop}
  \label{prop:dSigma-lower}
  Let $\Sigma \in \A(\delta,m)$ be an admissible set, such that $\E_p(\Sigma)
  \le E < \infty$ for some $p > m(m+2)$. Then the stopping distances $d(x_0)$
  defined in Proposition~\ref{prop:big-proj-fat-simp} have a positive lower bound
  \begin{equation}
    \label{eq:dSigma-lower}
    d(\Sigma) := \inf_{x_0 \in \Sigma^*} d(x_0)
    \ge \left( \frac{\Cr{uahlreg1} \Cr{uahlreg2}^p}{E} \right)^{\frac 1{p-m(m+2)}} \,.
  \end{equation}
  where $\Cr{uahlreg1} = \Cr{uahlreg1}(\delta,m)$ and $\Cr{uahlreg2} =
  \Cr{uahlreg2}(\delta,m)$ are some positive constants which depend only on
  $\delta$ and $m$.
\end{prop}

\begin{proof}
  From Proposition~\ref{prop:eta-d-balance} we know that $d(\Sigma)$ must
  satisfy \eqref{est:eta-d} with the constant $A_{\Sigma}$ and $\eta =
  \eta(\delta,m)$ defined in \eqref{def:eta}. Hence, we already have a positive
  lower bound on $d(\Sigma)$. Now we only need to show that it does not depend
  on $A_{\Sigma}$.

  Fix a point $x_0 \in \Sigma^*$ such that $d(x_0) < (1 + \varepsilon)d(\Sigma)$
  for some small $\varepsilon > 0$. Proposition~\ref{prop:big-proj-fat-simp} gives us a
  simplex $T = \simp(x_0, \ldots, x_{m+1}) \in \Reg_m(\eta,d(x_0))$. From
  Proposition~\ref{prop:perturb} we know that there exists a small number $\varsigma_m
  < \tfrac 12$ such that $T' \in \Reg_m(\tfrac 12 \eta, \tfrac 32 d(x_0))$ for each $T' =
  \simp(x_0', \ldots, x_{m+1}')$ satisfying $|x_i - x_i'| \le \varsigma_m d(x_0)$ for
  $i = 0, \ldots, m+1$. If $\varepsilon < \frac{1}{\varsigma_m} - 1$ then
  \begin{displaymath}
    \varsigma_m d(x_0) \le \varsigma_m (1+\varepsilon) d(\Sigma) \le d(\Sigma) \le d(x_i) \,,
  \end{displaymath}
  so Corollary~\ref{cor:uahlreg} gives us
  \begin{displaymath}
    \HM^m(\Sigma \cap \Ball(x_i, \varsigma_m d(x_0)))
    \ge (1 - \delta^2)^{\frac m2} \omega_m (\varsigma_m d(x_0))^m \,.
  \end{displaymath}
  Now, we can repeat the calculation from the proof of
  Proposition~\ref{prop:eta-d-balance}, replacing $A_{\Sigma}$ by $A_1 =
  A_1(\delta,m) := \sqrt{1 - \delta^2} \omega_m \varsigma_m^m$ to obtain
  \begin{displaymath}
    (1+\varepsilon)d(\Sigma) > d(x_0) \ge
    \left(
      \frac{\Cr{eta-d1} \Cr{eta-d2}^p A_1^{m+2} \eta^{m(m+1)^2(m+2)} (\eta^{m+1})^p}{E}
    \right)^{\frac 1{p-m(m+2)}} \,.
  \end{displaymath}
  The constants $A_1$ and $\eta$ depend only on $\delta$ and $m$ so setting
  \begin{align*}
    \Cr{uahlreg1} &= \Cr{uahlreg1}(\delta,m) := \Cr{eta-d1}(m) A_1(\delta,m) \eta(\delta,m)^{m(m+1)^2(m+2)} \\
    \text{and} \quad
    \Cr{uahlreg2} &= \Cr{uahlreg2}(\delta,m) := \Cr{eta-d2}(m) \eta(\delta,m)^{m+1}
  \end{align*}
  and letting $\varepsilon \to 0$ we reach the estimate \eqref{eq:dSigma-lower}.
\end{proof}

\begin{prop}
  \label{prop:big-proj-all-points}
  Let $\Sigma \in \A(\delta,m)$, $E > 0$ and $p > m(m+2)$. Assume that
  $\E_p(\Sigma) \le E < \infty$. Set
  \begin{equation}
    \label{def:uar-rad}
    \Cr{uar-rad} = \Cr{uar-rad}(E,m,p,\delta)
    := \left( \frac{\Cr{uahlreg1} \Cr{uahlreg2}^p}{E} \right)^{\frac 1{p-m(m+2)}} \,.
  \end{equation}
  Then for each $x \in \Sigma$ and $\rho \le \Cr{uar-rad}$ there exists an $m$-plane $H =
  H(\rho) \in G(n,m)$ such that
  \begin{displaymath}
    (x + H) \cap \Ball(x,\sqrt{1 - \delta^2} \rho) \subseteq \pi_{x+H}(\Sigma \cap \Ball(x,\rho)) \,.
  \end{displaymath}
\end{prop}

\begin{proof}
  Proposition~\ref{prop:big-proj-fat-simp} gives us this result for any $x \in
  \Sigma^*$. We only need to show that it is true also for $x \in \Sigma
  \setminus \Sigma^*$.

  Let $x$ be a point in $\Sigma \setminus \Sigma^*$ and fix a radius $\rho \le
  \Cr{uar-rad}$. Choose a sequence of points $x_i \in \Sigma^*$ converging to $x$. From
  Proposition~\ref{prop:big-proj-fat-simp} we obtain a sequence of $m$-planes $H_i
  \in G(n,m)$ such that
  \begin{displaymath}
    D_i := (x_i + H_i) \cap \Ball(x_i,\sqrt{1 - \delta^2} \rho)
    \subseteq \pi_{x_i+H_i}(\Sigma \cap \Ball(x_i,\rho)) \,.
  \end{displaymath}
  Since the Grassmannian $G(n,m)$ is a compact manifold, passing to a subsequence
  we can assume that $H_i$ converges to some $H$ in $G(n,m)$. Set
  \begin{displaymath}
    D := (x + H) \cap \Ball(x,\sqrt{1 - \delta^2} \rho) \,.
  \end{displaymath}
  Fix a point $w \in D$. We will show that the preimage $\pi_{x+H}^{-1}(w) \cap
  (\Sigma \cap \Ball(x,\rho))$ is nonempty. Chose points $w_i \in D_i$ such that
  $|w_i - x_i| = |w - x|$ and $w_i \to w$. We know that there exist points $y_i
  \in \Sigma \cap \Ball(x_i,\rho)$ such that
  \begin{displaymath}
    \pi_{x_i+H_i}(y_i) = w_i \,,
  \end{displaymath}
  so
  \begin{displaymath}
    y_i = w_i + v_i \quad \text{for some } v_i \in H_i^{\perp} \,.
  \end{displaymath}
  Moreover
  \begin{displaymath}
    \rho^2 \ge |w_i - x_i|^2 + |v_i|^2 \,,
  \end{displaymath}
  hence
  \begin{displaymath}
    |v_i|^2 \le \rho^2 - |w_i - x_i|^2 = \rho^2 - |w-x|^2 \,.
  \end{displaymath}
  We now know that $v_i$ all lie inside a ball of radius $\rho^2 - |w-x|^2$, which
  is compact, so passing to a subsequence, we can assume that $v_i \to v \in H^{\perp}$.
  This gives us
  \begin{align*}
    y_i &= w_i + v_i \to y = w + v \,,\\
    |v|^2 &\le \rho^2 - |w - x|^2  \\
    \text{and} \quad
    |y-x|^2 &= |w-x|^2 + |v|^2 \le \rho
    \quad \Rightarrow \quad y \in \Sigma \cap \Ball(x,\rho) \,.
  \end{align*}
  We have found $y \in \Sigma \cap \Ball(x,\rho)$ such that $\pi_{x+H}(y) = w$
  and this completes the proof.
\end{proof}

\begin{proof}[Proof of Theorem~\ref{thm:uahlreg}]
  We proceed as in the proof of Corollary~\ref{cor:uahlreg}. Orthogonal
  projections are Lipschitz mappings with constant $1$ so they cannot increase
  the measure. From Proposition~\ref{prop:big-proj-all-points} we know that for
  each $x \in \Sigma$ and each $\rho \le \Cr{uar-rad} = \Cr{uar-rad}(E,m,p,\delta)$ there exists
  an $m$-plane $H$ such that the image of $\Sigma \cap \Ball(x, \rho)$ under
  $\pi_{x + H}$ contains the ball $(x + H) \cap \Ball(x, \sqrt{1 - \delta^2}
  \rho)$. The measure of that ball is $(1 - \delta^2)^{\frac m2} \omega_m
  \rho^m$ so the $\HM^m$-measure of $\Sigma \cap \Ball(x,\rho)$ cannot be less
  than this number.
\end{proof}



\mysubsection{Relation between admissible sets and fine sets}

In this paragraph we establish a connection between the class $\A(\delta,m)$ of
admissible sets and the class $\F(m)$ of fine sets. We show
(Theorem~\ref{thm:adm-fine}) that in the class of sets with finite $p$-energy
every admissible set is also fine. Later in \S\ref{sec:tangent-planes} we show
that $m$-fine sets with bounded $p$-energy are $C^{1,\tau}$ manifolds, hence
they are also $(\delta,m)$-admissible for any $\delta \in (0,1)$
(cf. Example~\ref{ex:mfld}).

\begin{prop}
  \label{prop:cone-limit}
  Let $\Sigma \in \A(\delta,m)$ be $(\delta,m)$-admissible set for some $\delta
  \in (0,1)$ such that $\E_p(\Sigma) < \infty$ for some $p > m(m+2)$. Choose any
  number $L$ such that 
  \begin{displaymath}
    \sqrt{\frac{2-\delta}{\delta}} <  L < \frac{1}{\delta} \,.
  \end{displaymath}
  Then for each $x \in \Sigma$ and each $r \le \Cr{uar-rad}$ there exists an
  $m$-plane $H \in G(n,m)$ such that
  \begin{enumerate}
  \item $(x + \Cone(L\delta,H,\frac 58 r,\frac 78 r)) \cap \Sigma = \emptyset$ and
  \item the sphere $S := \Sphere(x,\frac 68 r) \cap (x+H^{\perp})$ is linked
    with $\Sigma$.
  \end{enumerate}
\end{prop}

\begin{proof}
  In the proof of~\ref{prop:big-proj-fat-simp} we have shown that analogous
  conditions hold for $x \in \Sigma^*$. We know that at each $x \in \Sigma^*$
  and for each $r \le \Cr{uar-rad}$ there exists an $m$-plane $H_x \in G(n,m)$
  such that
  \begin{itemize}
  \item $(x + \Cone(\delta,H_x,\frac 12 r,r)) \cap \Sigma = \emptyset$ and
  \item the sphere $S := \Sphere(x,\frac 34 r) \cap (x+H_x^{\perp})$ is linked with $\Sigma$.
  \end{itemize}
  Now we only need to show that we can pass to a limit. Fix a number $K$
  satisfying $\sqrt{\frac{2-\delta}{\delta}} < K < L$ and fix $r \le
  \Cr{uar-rad}$, let $x \in \Sigma \setminus \Sigma^*$ and let $x_k \in
  \Sigma^*$ be a sequence of points converging to $x$. Using compactness of
  $G(n,m)$ and possibly passing to a subsequence we obtain a convergent sequence
  of $m$-planes $H_k$. Let $H_0$ be the limit of $H_k$. For any choice of $\zeta
  > 0$ and $\xi > 0$ we can find $k_0$ such that for $k > k_0$ we have
  \begin{displaymath}
    \dgras(H_k,H_0) \le \zeta
    \quad \text{and} \quad
    |x_k - x_0| \le \xi \,.
  \end{displaymath}

  \begin{lem}[Step 1]
    \label{lem:cone-rotate}
    There exists $\zeta = \zeta(\delta,K)$ such that whenever $\dgras(H_k,H_0)
    \le \zeta$ then
    \begin{displaymath}
      \Cone(K\delta,H_0) \subseteq \Cone(\delta,H_k) \,.
    \end{displaymath}
  \end{lem}
  \begin{proof}
    Let $x \in \Cone(K\delta,H_0)$. First we estimate $|\pi_{H_k}(x)|$.
    \begin{align*}
      |\pi_{H_k}(x)| 
      &\le |\pi_{H_k}(\pi_{H_0}(x))| + |\pi_{H_k}(Q_{H_0}(x))| \\
      &\le |\pi_{H_0}(x)| + \zeta |Q_{H_0}(x)|
      \le |x| (\sqrt{1 - (K\delta)^2} + \zeta) \,.
    \end{align*}
    Now we can wite
    \begin{displaymath}
      |Q_{H_k}(x)| 
      \ge |x| - |\pi_{H_k}(x)|
      \ge |x| (1 - \sqrt{1 - (K\delta)^2} - \zeta) \,.
    \end{displaymath}
    Therefore, we need to find $\zeta > 0$ such that $1 - \sqrt{1 -
      (K\delta)^2} - \zeta \ge \delta$. Let us calculate
    \begin{displaymath}
      1 - \sqrt{1 - (K\delta)^2} - \zeta \ge \delta
      \quad \iff \quad 
      \zeta \le 1 - \delta - \sqrt{1 - (K\delta)^2} \,.
    \end{displaymath}
    The question remains whether $1 - \delta - \sqrt{1 - (K\delta)^2}$ is
    positive. Another calculation shows
    \begin{displaymath}
      1 - \delta - \sqrt{1 - (K\delta)^2} > 0 
      \quad \iff \quad
      \frac{2 - \delta}{ \delta} <  K^2 \,,
    \end{displaymath}
    but this is exactly what we assumed about $K$. We can safely set
    \begin{displaymath}
      \zeta = \zeta(\delta,K) := 1 - \delta - \sqrt{1 - (K\delta)^2} \,.
    \end{displaymath}
  \end{proof}

  \begin{lem}[Step 2]
    \label{lem:cone-translate}
    There exists $\xi = \xi(K,L,\delta,r)$ such that whenever $|x_k - x_0| \le
    \xi$ then for each $x \in \Rn$ such that $|x - x_0| \ge \tfrac 12 r$
    \begin{displaymath}
      |Q_{H_0}(x - x_0)| \ge L\delta|x - x_0|
      \quad \Rightarrow \quad
      |Q_{H_0}(x - x_k)| \ge K\delta|x - x_k| \,.
    \end{displaymath}
    In other words
    \begin{displaymath}
      (x_0 + \Cone(L\delta,H_0)) \setminus \Ball(x_0,\tfrac 12 R)
      \subseteq
      (x_0 + \Cone(\delta,H_0)) \cap (x_k + \Cone(K\delta,H_0)) \,.
    \end{displaymath}
  \end{lem}
  \begin{proof}
    Let $x \in (x_0 + \Cone(L\delta,H_0))$ be such that $|x - x_0| \ge
    \tfrac 12 r$. We then have
    \begin{align*}
      |Q_{H_0}(x - x_k)|
      \ge |Q_{H_0}(x - x_0)| - |x_k - x_0|
      \ge L\delta |x - x_0| - \xi \,.
    \end{align*}
    We need to find $\xi > 0$ such that $L\delta |x - x_0| - \xi \ge
    K\delta |x - x_k|$. Set 
    \begin{displaymath}
      \xi = \xi(K,L,\delta,r) := \tfrac 14 \delta (L-K) r \,.
    \end{displaymath}
    We obtain
    \begin{align*}
      (1 +  K \delta) \xi &\le 2 \xi 
      \le \delta (L-K) \tfrac 12 r 
      \le \delta (L-K) |x - x_0| \\
      &\, \Rightarrow \quad
      |Q_{H_0}(x - x_k)| 
      \ge L \delta |x - x_0| - \xi 
      \ge K \delta (|x - x_0| + \xi) 
      \ge K \delta |x - x_k|
    \end{align*}
  \end{proof}

  Lemmas~\ref{lem:cone-rotate} and~\ref{lem:cone-translate} give us a good
  choice of $\zeta$ and $\xi$. Shrinking $\xi$ if needed, we can
  assume that $\xi < \frac 18 r$. Then we have
  \begin{align*}
    \CBall(x_0, \tfrac 12 r) \cup \CBall(x_k, \tfrac 12 r) 
    &\subseteq \CBall(x_0, \tfrac 58 r)  \\
    \text{and} \quad
    \Ball(x_0, r) \cap \Ball(x_k, r) 
    &\supseteq \Ball(x_0, \tfrac 78 r) \,.
  \end{align*}
  Hence, for each $k$ big enough
  \begin{equation}
    \label{eq:cone-incl}
    x_0 + \Cone(L\delta,H_0,\tfrac 58 r, \tfrac 78 r) 
    \subseteq
    x_k + \Cone(\delta,H_k,\tfrac 12 r, r) \,,
  \end{equation}
  and we obtain the first required condition
  \begin{displaymath}
    x_0 + \Cone(L\delta,H_0,\tfrac 58 r, \tfrac 78 r) \cap \Sigma = \emptyset \,.
  \end{displaymath}

  \begin{figure}[!htb]
    \centering
    \includegraphics{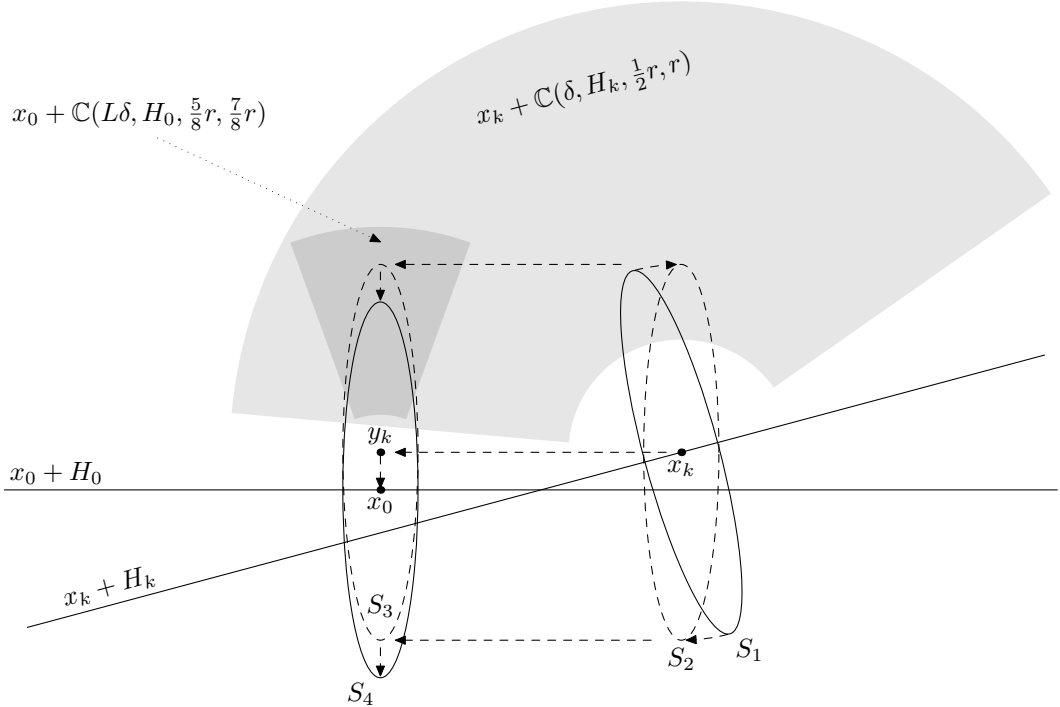}
    \caption{If $x_k$ is sufficiently close to $x_0$, then the cone over $x_k +
      H_k$ contains a small conical cap over $x_0 + H_0$. This allows us to
      continuously transform $S_1$ into $S_4$ without leaving the grey area.}
    \label{F:cone-limit}
  \end{figure}
  To prove the second condition, involving the linked spheres, let us set $S_1
  := \Sphere(x_k,\tfrac 68 r) \cap (x_k + H_k^{\perp})$. From the definition of
  admissible sets we know that $S_1$ is linked with $\Sigma$. We use
  Corollary~\ref{cor:sph-in-cone} to find an isotopy (see
  Figure~\ref{F:cone-limit})
  \begin{displaymath}
    F_1 : S_1 \times [0,1] \to \big( x_k + \Cone(\delta,H_k,\tfrac 12 r, r) \big) \,,
  \end{displaymath}
  which continuously rotates $S_1$ into $S_2 := \Sphere(x_k,\tfrac 68 r) \cap
  (x_k + H_0^{\perp})$. All we need to know is that $S_2$ is contained in $x_k +
  \Cone(\delta,H_k,\tfrac 12 r, r)$ but this follows from Lemma
  \ref{lem:cone-rotate}. Next, we continuously translate $S_2$ into $S_3 :=
  \Sphere(y_k,\tfrac 68 r) \cap (y_k + H_0^{\perp})$, where $y_k := x_k +
  \pi_{H_0}(x_0 - x_k)$, using the isotopy
  \begin{align*}
    F_2 : S_2 \times [0,1] &\to \big( x_k + \Cone(\delta,H_k,\tfrac 12 r, r) \big) \,, \\
    F_2(z,t) &:=  z + t \pi_{H_0}(x_0 - x_k) \,.
  \end{align*}
  To see that this transformation is performed inside $x_k + \Cone(\delta,H_k,\tfrac
  12 r, r)$ let us choose a point $z \in S_2$ and $t \in [0,1]$. Since
  $|\pi_{H_0}(x_0 - x_k)| \le |x_0 - x_k| \le \xi$, we have $\frac 68 r -
  \xi \le |F_2(z,t) - x_k| \le \frac 68 r + \xi$ and
  \begin{displaymath}
    \frac{|Q_{H_0}(F_2(z,t) - x_k)|}{|F_2(z,t) - x_k|}
    \ge \frac{\tfrac 68 r}{\tfrac 68 r + \xi} \ge \delta
    \iff
    \xi \le \frac{6 (1 - \delta)}{8\delta} r \,.
  \end{displaymath}
  To make everything work, we may shrink $\xi$, so that it satisfies the
  above condition. Finally we translate $S_3$ along the vector $Q_{H_0}(x_0 -
  x_k)$ into $S_4 := \Sphere(x_0,\tfrac 68 r) \cap (x_0 + H_0^{\perp})$ with the
  isotopy
  \begin{align*}
    F_3 : S_3 \times [0,1] &\to H_0^{\perp} \cap \Shell(\tfrac 58r, \tfrac 78 r) \,, \\
    F_3(z,t) &:=  z + t Q_{H_0}(x_0 - x_k) \,.
  \end{align*}
  We have $|Q_{H_0}(x_0 - x_k)| \le \xi < \frac 18 r$ and the last
  translation is performed inside $x_0 + H_0^{\perp}$, so it stays in $x_0 +
  \Cone(L\delta,H_0,\tfrac 58 r, \tfrac 78 r)$. This gives the second condition of
  Proposition~\ref{prop:cone-limit}.
\end{proof}

\begin{thm}
  \label{thm:adm-fine}
  If $\Sigma \subseteq \Rn$ is $(\delta,m)$-admissible and additionally
  $\E_p(\Sigma) \le E < \infty$ for some $p > m(m+2)$, then $\Sigma$ is also
  $m$-fine with constants
  \begin{displaymath}
    A_{\Sigma} = (1-\delta^2)^{m/2} \omega_m \,, \quad
    R_{\Sigma} = \min \{ \Cr{uar-rad}, \Cl[R]{adm-fine-rad}(E,m,p,\delta) \}
    \quad and \quad
    M_{\Sigma} = 5 \,.
  \end{displaymath}
\end{thm}

\begin{proof}
  To prevent confusion let us make the following distinction. In the proof we
  refer to constants from the definition of $(\delta,m)$-admissible sets by
  $A_{\Sigma}'$ and $R_{\Sigma}'$. The constants from the definition of $m$-fine
  sets we shall denote by $A_{\Sigma}$, $R_{\Sigma}$ and $M_{\Sigma}$.

  Corollary~\ref{cor:ind-const} states that $A_{\Sigma}' = (1-\delta^2)^{m/2}
  \omega_m$ and $R_{\Sigma}' = \Cr{uar-rad}$, so these constants depend only on
  $E$, $m$, $p$ and $\delta$. Therefore we may set $A_{\Sigma} = A_{\Sigma}'$
  and then all we need to show is that there exist numbers $R_{\Sigma} \le
  R_{\Sigma}'$ and $M_{\Sigma}$ such that for $r \le R_{\Sigma}$ and for all $x
  \in \Sigma$
  \begin{displaymath}
    \ntheta(x,r) \le M_{\Sigma} \nbeta(x,r) \,.
  \end{displaymath}
  From Corollary~\ref{cor:beta-est} we know that $\nbeta(x,r) \le
  \Cr{beta-est}E^{1/\kappa} r^{\tau}$, so it converges to $0$ when $r \to 0$ uniformly
  with respect to $x \in \Sigma$. Fix a point $x_0 \in \Sigma$ and a radius $r
  \le \Cr{uar-rad}$. Choose some $m$-plane $P \in G(n,m)$ such that
  \begin{displaymath}
    \forall y \in \Sigma \cap \Ball(x_0,r)\ \ 
    |Q_P(y - x_0)| \le \nbeta(x,r) \,.
  \end{displaymath}
  Fix a number $L$ such that $\sqrt{\frac{2-\delta}{\delta}} < L <
  \frac{1}{\delta}$ and set
  \begin{displaymath}
    \beta := 2\nbeta(x_0,r)
    \qquad \text{and} \qquad
    \gamma := \sqrt{1 - (L \delta)^2} \in (0,1) \,.
  \end{displaymath}
  Let $H$ be the $m$-plane for the point $x_0$ given by
  Proposition~\ref{prop:cone-limit}, so that
  \begin{displaymath}
    \Cone(L\delta, H, \tfrac 58 r,\tfrac 78 r) \cap \Sigma = \emptyset \,.
  \end{displaymath}
  Let $z \in \Sigma \cap \Ball(x_0,r)$ be any point in the intersection $\Sigma
  \cap \Ball(y, L\delta \frac 78 r) \cap (y + H^{\perp})$, where $y$ is any
  point such that $(y - x_0) \in H$ and $|y - x_0| = \frac 78 r \gamma$. Such
  point $z$ exists since the sphere $\Sphere(y,L \delta \frac 78 r) \cap (y +
  H^{\perp})$ is linked with $\Sigma$ (cf. Lemma~\ref{lem:intpoint}).

  Note that $\frac 78 r \gamma \le |z - x_0| \le \frac 78 r$, so
  \begin{displaymath}
    \frac{|Q_P(z - x_0)|}{|z - x_0|} 
    \le \frac{\beta r}{\tfrac 78 r \gamma}
    = \frac{8 \beta}{7 \gamma} \,,
  \end{displaymath}
  hence
  \begin{displaymath}
    (z - x_0) \in \Cone \left(
      \big( 1 - \tfrac{(8\beta)^2}{(7\gamma)^2} \big)^{\frac 12}, P^{\perp}
    \right) \cap \Cone(\gamma,H^{\perp}) \,.
  \end{displaymath}
  To apply Proposition~\ref{prop:two-cones} we need to ensure the condition
  \begin{align}
    \label{cond:R-Sigma}
    \sqrt{1 - \gamma^2} + \tfrac{8\beta}{7\gamma} 
    &\le (1 - \tfrac{8\beta}{7\gamma})\sqrt{1 - \left(\tfrac{8\beta}{7\gamma}\right)^2} 
    \iff \\
    \iff
    \beta 
    &\le \tfrac 78 \gamma \left(
      (1 - \tfrac{8\beta}{7\gamma})\sqrt{1 - \left(\tfrac{8\beta}{7\gamma}\right)^2} - \sqrt{1 - \gamma^2}
    \right) \,. \notag
  \end{align}
  Substituting $\Psi := \frac{8\beta}{7\gamma}$ in \eqref{cond:R-Sigma} and
  recalling that $\gamma = \sqrt{1 - (L\delta)^2}$ we obtain the following
  inequality
  \begin{equation}
    \label{eq:Psi}
    \Psi \le (1 - \Psi)\sqrt{1 - \Psi^2} - L\delta \,.
  \end{equation}
  Note that if $\Psi \to 0$ then the right-hand side converges to $1 - L\delta >
  0$. Let $\Psi_0$ be the smallest, positive root of the equation $\Psi = (1 -
  \Psi)\sqrt{1 - \Psi^2} - L\delta$. Then any $\Psi \in (0,\Psi_0)$ satisfies
  \eqref{eq:Psi}. Recall that $\tfrac 12 \beta = \nbeta(x,r) \le
  \Cr{beta-est} E^{1/\kappa} r^{\tau}$, so to ensure condition
  \eqref{cond:R-Sigma} it suffices to impose the following constraint
  \begin{equation}
    \label{def:adm-fine-rad}
    r \le \Cr{adm-fine-rad}(E,m,p,\delta) := \Bigg( \frac{7 \gamma \Psi_0}{16 \Cr{beta-est} } \Bigg)^{1/\tau} E^{-1/\lambda} \,.
  \end{equation}
  Now, for such $r$ we can use Proposition~\ref{prop:two-cones} to obtain
  \begin{displaymath}
    H^{\perp} \subseteq \Cone(L\delta, H) \cap \Cone(\tfrac{8\beta}{7\gamma}, P) \,.
  \end{displaymath}

  \begin{figure}[!htb]
    \centering
    \includegraphics{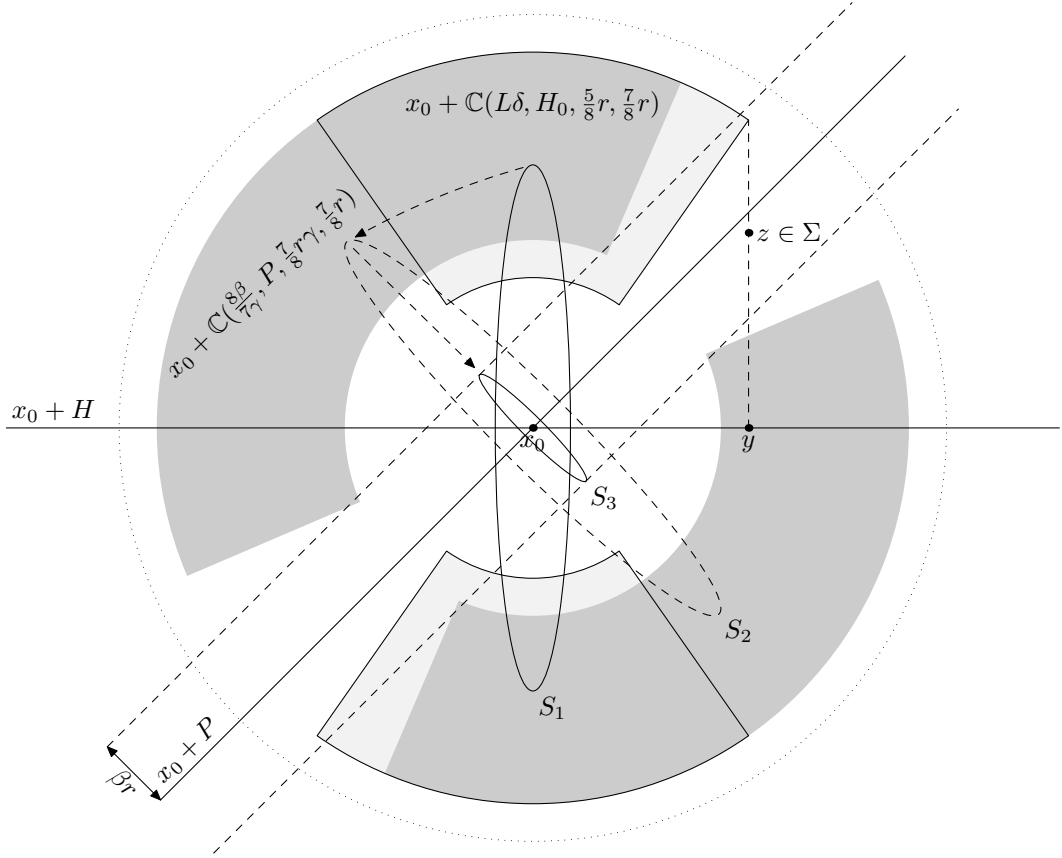}
    \caption{If $\beta$ is small enough, then the cone
      $\Cone(\tfrac{8\beta}{7\gamma}, P)$ contains $H^{\perp}$ and we can
      continuously transform $S_1$ into $S_3$ inside the conical cap
      $\Cone(\tfrac{8\beta}{7\gamma}, P, \frac 78r\gamma, \frac 78r)$.}
    \label{F:adm-fine}
  \end{figure}
  Set $S_1 := \Sphere(x_0,\frac{7}{16}r(\gamma+1)) \cap (x_0 + H^{\perp})$. This
  sphere is contained in the conical cap $\Cone(\tfrac{8\beta}{7\gamma}, P,
  \frac 78r\gamma, \frac 78r)$ (see Figure~\ref{F:adm-fine}). Using
  Corollary~\ref{cor:sph-in-cone} we rotate $S_1$ into $S_2 :=
  \Sphere(x_0,\frac{7}{16}r(\gamma+1)) \cap (x_0 + P^{\perp})$ inside
  $\Cone(\tfrac{8\beta}{7\gamma}, P, \frac 78r\gamma, \frac 78r)$. Note that for
  $x \in \Sigma$ such that $|x - x_0| > \frac 78r\gamma$ we have
  \begin{displaymath}
    \frac{Q_P(x - x_0)}{|x - x_0|} 
    < \frac{\beta r}{\frac 78r\gamma} = \frac{8\beta}{7\gamma}\,,
  \end{displaymath}
  hence the conical cap $\Cone(\tfrac{8\beta}{7\gamma}, P, \frac 78r\gamma, \frac
  78r)$ does not intersect $\Sigma$ and the resulting sphere $S_2$ is still
  linked with $\Sigma$. Next we decrease the radius of $S_2$ to the value $\beta
  r$ obtaining another sphere $S_3 := \Sphere(x_0,\beta r) \cap (x_0 +
  P^{\perp})$ which is also linked with $\Sigma$.

  We can translate $S_3$ along any vector $v \in P$ with $|v| \le \sqrt{1 -
    \beta^2}r$ without changing the linking number. This way we see that for any
  point $w \in (x_0 + P) \cap \CBall(x_0,\sqrt{1 - \beta^2}r)$ there
  exists a point $z \in \Sigma$ such that $|z-w| \le \beta r$.

  For any other point $w \in (x_0 + P)$ with $\sqrt{1 - \beta^2}r \le |w - x_0|
  \le r$ we set 
  \begin{displaymath}
    \tilde{w} := w - (w - x_0)|w - x_0|^{-1} (1 - \sqrt{1 - \beta^2})r \,,
  \end{displaymath}
  so that $|\tilde{w} - x_0| \le \sqrt{1 - \beta^2}r$. Then we find $z \in
  \Sigma$ such that $|\tilde{w} - z| \le \beta r$ and we obtain the estimate
  \begin{align*}
    |z - w| &\le |z - \tilde{w}| + |\tilde{w} - w|
    \le \beta r + (1 - \sqrt{1 - \beta^2}) r \\
    &= r \left( \beta + \frac{\beta^2}{1 + \sqrt{1 - \beta^2}} \right)
    \le 2 \beta r = 4 \nbeta(x,r) r \,.
  \end{align*}
  This implies that $\HD(\Sigma \cap \CBall(x_0,r), (x_0 + P) \cap
  \CBall(x_0,r)) \le 5 \nbeta(x_0,r)$. Therefore the infimum over all $H \in
  G(n,m)$ must be even smaller, so $\ntheta(x_0,r) \le 5 \nbeta(x_0,r)$ for any
  $r \le R_{\Sigma}$ and we can safely set $M_{\Sigma} := 5$.
\end{proof}



\mysection{Existence and oscillation of tangent planes}
\label{sec:tangent-planes}

In this paragraph we prove that boundedness of the $p$-energy $\E_p(\Sigma) \le
E$ implies $C^{1,\tau}$ regularity for some $\tau \in (0,1)$. First we show how
to use the result (Proposition~\ref{prop:dkt-reg}) obtained by David, Kenig and
Toro~\cite{MR1808649} which immediately gives $C^{1,\tau}$ regularity. Then,
independently of~\cite{MR1808649} we prove a bit stronger result
(Theorem~\ref{thm:C1tau}). We adjust the technique presented in~\cite{MR1808649}
to our needs. We also carefully keep track of all the emerging constants and
their dependences to be able to bound the H{\"o}lder norm and the size of the
maps in terms of $E$ and independently of $\Sigma$.

\begin{prop}
  \label{prop:C1theta}
  Let $\Sigma \in \F(m)$ be such that $\E_p(\Sigma) \le E < \infty$. Then
  $\Sigma$ is a closed $C^{1,\tau}$-submanifold of $\Rn$.
\end{prop}

\begin{proof}
  From Corollary~\ref{cor:beta-est} we already have good estimates on the
  $\nbeta$-numbers of $\Sigma$. Namely, for any $r < R_{\Sigma}$ and all $x
  \in \Sigma$ we have
  \begin{displaymath}
    \nbeta(x,r) \le \Cr{beta-est} E^{\frac 1{\kappa}} r^{\tau} \,,
  \end{displaymath}
  where $\Cr{beta-est}$ depends only on $m$, $p$ and $A_{\Sigma}$ and $\tau >
  0$. Since $\Sigma \in \F(m)$ it satisfies the condition \ref{fine:gaps}, so
  for $r < R_{\Sigma}$ we have
  \begin{equation}
    \label{est:theta}
    \ntheta(x,r) \le \Cr{beta-est} M_{\Sigma} r^{\tau} \,,
  \end{equation}
  which converges to $0$ when $r \to 0$ uniformly for all $x \in \Sigma$.
  Proposition~\ref{prop:theta-theta} implies that $\theta_m(x,r)$ also converges
  uniformly to $0$ when $r \to 0$ and that $\beta_m(x,r) \lesssim r^{\tau}$ for
  each $x$ and $r < R_{\Sigma}$. Hence, $\Sigma$ is Reifenberg flat with
  vanishing constant and satisfies the assumptions of
  Proposition~\ref{prop:dkt-reg}. Therefore $\Sigma$ is a $C^{1,\tau}$ manifold.

  Assume that $\Sigma$ is not closed, so $\partial \Sigma \ne \emptyset$. Let
  $x \in \partial \Sigma$ be a boundary point. For $r$ small enough the set
  $\Sigma \cap \Ball(x,r)$ is close to some half-$m$-plane $H_+ \simeq
  \R^{m-1} \times \R_+$. Then one sees easily that $\ntheta(x,r) \ge 1$, but
  this contradicts estimate \eqref{est:theta}.
\end{proof}

The rest of this section is devoted to showing that $\Sigma \in \F(m)$ with
$p$-energy bounded by $E < \infty$ has an atlas of maps of a given size, which
depends only on $E$, $m$ and $p$ but not on $\Sigma$ itself. Moreover we show
that $\Sigma$ is locally a graph of a $C^{1,\tau}$ function with the H\"older
constant also depending only on the energy $E$, the dimension $m$ and the
exponent $p$. In a forthcoming project, we plan use these results to address the
following problem:
\begin{quote}
  In the class of sets $\Sigma \in \F(m)$, normalized so that $0 \in \Sigma$ and
  $\HM^m(\Sigma) \le 1$, with uniformly bounded $p$-energy $\E_p(\Sigma) \le E$
  for some $p > m(m+2)$ there can be only finite number of non-homeomorphic sets
  and the number of homeomorphism classes can be bounded in terms of $E$.
\end{quote}

For the sake of brevity we introduce the following notation
\begin{displaymath}
  \pi_x := \pi_{T_x\Sigma}
  \qquad \text{and} \qquad
  Q_x := Q_{T_x\Sigma} \,,
\end{displaymath}
where $x \in \Sigma$. The main result of this section is
\begin{thm}
  \label{thm:C1tau}
  Let $\Sigma \in \F(m)$ be an $m$-fine set such that $\E_p(\Sigma) \le E <
  \infty$ for some $p > m(m+2)$. Then $\Sigma$ is a smooth manifold of class
  $C^{1,\tau}$, where $\tau$ was defined in \S\ref{sec:bdd-ene-flat} by the
  formula
  \begin{displaymath}
    \tau = \tfrac{\lambda}{\kappa} = \tfrac{p - m(m+2)}{(m+1)(m(m+1)(m+2)+p)} \,.
  \end{displaymath}

  Moreover there exists a constant $\Cl{smooth-rad-const} =
  \Cr{smooth-rad-const}(m,p)$ such that if we set $\Cl[R]{smooth-rad} :=
  \Cr{smooth-rad-const} E^{-1/\lambda}$ then for each point $x \in \Sigma$ there
  exists a $C^{1,\tau}$ function
  \begin{displaymath}
    F_x : T_x\Sigma \cap \CBall_{\frac 12 \Cr{smooth-rad}}
    \to
    T_x\Sigma^{\perp} \cap \CBall_{\Cr{smooth-rad}} \,,
  \end{displaymath}
  such that
  \begin{displaymath}
    (\Sigma - x) \cap \{ y \in \CBall_{\Cr{smooth-rad}} : |\pi_x(y)| \le \tfrac 12 \Cr{smooth-rad}\}
    = F_x(T_x\Sigma \cap \CBall_{\frac 12 \Cr{smooth-rad}}) \,, 
  \end{displaymath}
  \begin{displaymath}
    F_x(0) = 0 \qquad \text{and} \qquad DF_x(0) = 0 \,.
  \end{displaymath}

  Furthermore there exists a constant $\Cl{holder-norm} = \Cr{holder-norm}(m,p)$
  such that for any two points $w_0,w_1 \in T_x\Sigma \cap \CBall_{\frac 12
    \Cr{smooth-rad}}$ we have
  \begin{displaymath}
    \| DF_x(w_1) - DF_x(w_0) \| \le \Cr{holder-norm} E^{1/\kappa} |w_1 - w_0|^{\tau} \,.
  \end{displaymath}
\end{thm}

To prove this theorem we fix a point $x \in \Sigma$ and for each radii $r > 0$
we choose an $m$-plane $P(x,r)$. Then we use the fact that $\ntheta(x,r) \le
M_{\Sigma} \nbeta(x,r) \le M_{\Sigma} \Cr{beta-est} E^{\frac 1{\kappa}}
r^{\tau}$ to show that $P(x,r)$ converge to the tangent plane $T_x\Sigma$, when
$r \to 0$. This also gives a bound on the oscillation of $T_x\Sigma$. Then we
derive Lemma~\ref{lem:1point-rad}, which says that at some small scale we cannot
have two distinct points $y$ and $z$ of $\Sigma$ such that the vector $v = (y -
z)$ is orthogonal to $T_x\Sigma$. Any such vector $v$ would be close to the
tangent plane $T_{z}\Sigma$ and this would violate the bound on the oscillation
of tangent planes proved earlier. From here, it follows that there exists a
small radius $\Cr{graph-rad}$ such that $\Sigma \cap \Ball(x,\Cr{graph-rad})$ is
a graph of some function $F_x$.

Next we define the differential $DF_x$ at a point $w \in T_x\Sigma \cap
\CBall(x,\Cr{graph-rad})$ using the inverse of the projection from $T_y\Sigma$
onto $T_x\Sigma$, where $y = F_x(w) + w$. This can be done since $y$ lies in
$\Sigma \cap \CBall(x,\Cr{graph-rad})$, so the ''angle''
$\dgras(T_x\Sigma,T_y\Sigma)$ is small and due to Remark~\ref{rem:orth-angle}
the projection $\pi_x$ gives a linear isomorphism between $T_x\Sigma$ and
$T_y\Sigma$. After that it is easy to see that the oscillation of $DF_x$ is
roughly the same as the oscillation of $T_x\Sigma$, so $DF_x$ is actually
H{\"o}lder continuous.

\mysubsection{The tangent planes}
Set
\begin{align}
  \label{def:beta-rad}
  \Cl[R]{beta-rad} = \Cr{beta-rad}(E,m,p,M_{\Sigma},A_{\Sigma},R_{\Sigma})
  &:=  \min \left\{
    (4 \Cr{beta-est} E^{1/\kappa} M_{\Sigma})^{-1/\tau},
    R_{\Sigma}
  \right\} \\
  &=  \min \left\{
    (4 \Cr{beta-est} M_{\Sigma})^{-1/\tau} E^{- 1/\lambda},
    R_{\Sigma}
  \right\} \notag
\end{align}
so that $\Cr{beta-est} E^{1/\kappa} \Cr{beta-rad}^{\tau} \le (4
M_{\Sigma})^{-1}$. Then for any $r \le \Cr{beta-rad}$ we have
\begin{displaymath}
  \ntheta(x,r)
  \le M_{\Sigma} \nbeta(x,r)
  \le M_{\Sigma} \Cr{beta-est} E^{1/\kappa} r^{\tau} 
  \le M_{\Sigma} \Cr{beta-est} E^{1/\kappa} \Cr{beta-rad}^{\tau} 
  \le \tfrac 14 \,.
\end{displaymath}

\begin{lem}
  \label{lem:bap-osc}
  Choose a point $x \in \Sigma$ and fix some $r_0 \le \Cr{beta-rad}$. Choose
  another point $y \in \Sigma \cap \CBall(x,\frac 12 r_0)$ and some $r_1 \in
  \left[ \frac 12 r_0, r_0 - |x-y| \right]$. Let $H_0 \in \BAP(x,r_0)$ and $H_1
  \in \BAP(y,r_1)$. Then
  \begin{displaymath}
    \dgras(H_0,H_1) \le \Cl{bap-osc} E^{1/\kappa} r_0^{\tau} \,,
  \end{displaymath}
  where $\Cr{bap-osc} = \Cr{bap-osc}(m,p,M_{\Sigma},A_{\Sigma})$.
\end{lem}

\begin{figure}[!htb]
  \centering
  \includegraphics{ahlreg11.mps}
  \caption{The existence of $z \in \Sigma$ is guaranteed by the condition
    $\ntheta(x,r) \le M_{\Sigma} \nbeta(x,r)$. This allows us to estimate
    $\dgras(H_0,H_1)$.}
  \label{F:bap-osc}
\end{figure}

\begin{proof}
  Set $\beta_0 := \nbeta(x,r_0)$ and $\beta_1 := \nbeta(y,r_1)$. Let $v \in H_1$
  be any vector of length $|v| = r_1 (1-M_{\Sigma}\beta_1)$. Since
  $\ntheta(y,r_1) \le M_{\Sigma}\beta_1$, there exists a point $z \in \Sigma
  \cap \CBall(y+v, M_{\Sigma}\beta_1 r_1)$. Hence $|(y+v) - z| \le
  M_{\Sigma}\beta_1 r_1$ (see Figure~\ref{F:bap-osc}). Note that $\CBall(y+v,
  M_{\Sigma}\beta_1 r_1) \subseteq \CBall(y,r_1) \subseteq
  \CBall(x,r_0)$. Therefore $\dist(z,x + H_0) = |Q_{H_0}(z - x)| \le \beta_0
  r_0$ and we obtain the estimate
  \begin{align*}
    |Q_{H_0}(v)| &\le |Q_{H_0}((y-x)+v)| + |Q_{H_0}(y-x)| \\
    &\le |((y-x)+v) - (z-x)| + |Q_{H_0}(z-x)| + |Q_{H_0}(y-x)| \\
    &\le M_{\Sigma} \beta_1 r_1 + \beta_0 r_0 + \beta_0 r_0 \le (M_{\Sigma} +
    2) \Cr{beta-est} E^{1/\kappa} r_0^{1+\tau} \,.
  \end{align*}
  Since $v$ was chosen arbitrarily we get the following estimate for any
  unit vector $e \in H_1 \cap \Sphere$
  \begin{displaymath}
    |Q_{H_0}(e)| \le (M_{\Sigma} + 2) \Cr{beta-est} E^{1/\kappa} \frac{r_0^{1+\tau}}{r_1 (1-M_{\Sigma}\beta_1)} 
    \le (M_{\Sigma} + 2) \Cr{beta-est} E^{1/\kappa} \frac{4r_0^{1+\tau}}{3r_1} \,.
  \end{displaymath}
  Recall that $r_1 \ge \frac 12 r_0$, so we have
  \begin{displaymath}
    |Q_{H_0}(e)| \le \tfrac 83 (M_{\Sigma} + 2) \Cr{beta-est} E^{1/\kappa} r_0^{\tau} \,.
  \end{displaymath}
  Applying Proposition~\ref{prop:dist-ang} we get
  \begin{displaymath}
    \dgras(H_0,H_1) \le \tfrac 83 (M_{\Sigma} + 2) \Cr{dist-ang} \Cr{beta-est} E^{1/\kappa} r_0^{\tau} \,.
  \end{displaymath}
  Finally we set $\Cr{bap-osc} := \tfrac 83 (M_{\Sigma} + 2) \Cr{dist-ang} \Cr{beta-est}$.
\end{proof}

\begin{lem}
  \label{lem:tan-conv}
  Choose a point $x \in \Sigma$. For each $r \le \Cr{beta-rad}$ fix an $m$-plane
  $P(r) \in \BAP(x,r)$. There exists a limit
  \begin{displaymath}
    \lim_{r \to 0} P(r) =: T_x\Sigma \in G(n,m)
  \end{displaymath}
  and it does not depend on the choice of $P(r) \in \BAP(x,r)$.
\end{lem}

\begin{proof}
  Set $\rho_k := 2^{-k}\Cr{beta-rad}$ and for each $k$ choose $P_k \in
  \BAP(x,\rho_k)$.  Set $\beta_k := \nbeta(x,\rho_k)$. We will show that
  $\{ P(r) \}_{r < \Cr{beta-rad}}$ satisfies the Cauchy condition. Fix
  some $0 < s < t < \rho_0$ and find two natural numbers $k < l$ such that
  $\rho_{l+1} < s \le \rho_l$ and $\rho_{k+1} < t \le \rho_k$.

  Applying Lemma~\ref{lem:bap-osc} with $x = y$, $r_0 = \rho_j$ and $r_1 :=
  \frac 12 r_0 = \rho_{j+1}$ we obtain
  \begin{displaymath}
    \dgras(P_j,P_{j+1}) \le \Cr{bap-osc}E^{1/\kappa} \rho_j^{\tau} \,.
  \end{displaymath}
  Setting $r_0 := \rho_l$ and $r_1 := s$ or $r_0 := \rho_k$ and $r_1 := t$
  we also get
  \begin{align*}
    \dgras(P(s),P_l) &\le \Cr{bap-osc}E^{1/\kappa} \rho_l^{\tau} \,, \\
    \dgras(P(t),P_k) &\le \Cr{bap-osc}E^{1/\kappa} \rho_k^{\tau} \,.
  \end{align*}
  Using these estimates we can write
  \begin{align*}
    \dgras(P(r),P(s)) 
    &\le \dgras(P(r),P_k) + \sum_{j=k}^{l-1} \dgras(P_j,P_{j+1}) + \dgras(P_l,P(s)) \\
    &\le \Cr{bap-osc}E^{1/\kappa} \left( \rho_k^{\tau} + \sum_{j=k}^l \rho_j^{\tau} \right) 
    = \Cr{bap-osc}E^{1/\kappa} \rho_k^{\tau} \left(1 + \sum_{j=0}^{l-k} 2^{-j \tau} \right) \\
    &\le \Cr{bap-osc} E^{1/\kappa} \frac{2^{1+\tau}}{2^{\tau}-1} \rho_k^{\tau}
    =: \Cl{tan-dist} E^{1/\kappa} \rho_k^{\tau} \,,
  \end{align*}
  which shows that the Cauchy condition is satisfied, so $P(r)$ converges in
  $G(n,m)$ to some $m$-plane, which we refer to as the tangent plane
  $T_x\Sigma$. The above estimates are valid for any choice of $P(r) \in
  \BAP(x,r)$, so we have actually shown that $T_x\Sigma$ not only exists but is
  also uniquely determined.
\end{proof}

\begin{rem}
  Note that
  \begin{displaymath}
    \Cr{tan-dist} = \Cr{tan-dist}(m,p,M_{\Sigma},A_{\Sigma}) 
    = \Cr{bap-osc} \frac{2^{1+\tau}}{2^{\tau}-1} \,.
  \end{displaymath}
\end{rem}

\begin{cor}
  \label{cor:tan-dist}
  Choose a point $x \in \Sigma$. For any $r \le \Cr{beta-rad}$ and any $H \in
  \BAP(x,r)$ we have
  \begin{displaymath}
    \dgras(T_x\Sigma, H) \le \Cr{tan-dist} E^{1/\kappa} r^{\tau}
  \end{displaymath}
\end{cor}

\begin{cor}
  \label{cor:tan-point}
  Choose a point $x \in \Sigma$. For any $y \in \Sigma \cap
  \CBall(x,\Cr{beta-rad})$ we have
  \begin{displaymath}
    \dist(y, x+T_x\Sigma) = |Q_x(y-x)| \le \Cl{tan-point} E^{1/\kappa} |y-x|^{1+\tau} \,,
  \end{displaymath}
  where $\Cr{tan-point} = \Cr{tan-point}(m,p,M_{\Sigma},A_{\Sigma})$. In particular
  \begin{displaymath}
    |Q_x(y-x)| 
    \le \Cr{tan-point} E^{1/\kappa} \Cr{beta-rad}^{\tau} |y-x| 
    \le \frac{\Cr{tan-point}}{4 \Cr{beta-est} M_{\Sigma}} |y-x| 
    =: \Cl{lip-const} |y-x| \,.
  \end{displaymath}
\end{cor}
\begin{proof}
  Choose an $m$-plane $H \in \BAP(x,|y-x|)$. Then we have
  \begin{align*}
    |Q_x(y-x)| 
    &\le |Q_H(y-x)| + |Q_x (\pi_H(y-x))| \\
    &\le |y-x| \nbeta(x,|y-x|) + |y-x| \Cr{tan-dist}E^{1/\kappa} |y-x|^{\tau} \\
    &\le \Cr{tan-point} E^{1/\kappa} |y-x|^{1+\tau} \,, 
  \end{align*}
  where $\Cr{tan-point} := \Cr{tan-dist} + \Cr{beta-est}$. This also gives
  \begin{displaymath}
    \Cr{lip-const} = \Cr{lip-const}(m,p,M_{\Sigma}) 
    = \frac{\Cr{tan-dist} + \Cr{beta-est}}{4 \Cr{beta-est} M_{\Sigma}} 
    = \frac{\tfrac 83 (M_{\Sigma} + 2) \Cr{dist-ang} \frac{2^{1+\tau}}{2^{\tau}-1} + 1}{4 M_{\Sigma}} \,.
  \end{displaymath}
\end{proof}

\begin{lem}
  \label{lem:tan-osc}
  Choose any point $x \in \Sigma$. There exists a constant $\Cl{tan-osc} =
  \Cr{tan-osc}(m,p,M_{\Sigma},A_{\Sigma})$ such that for each $y \in \Sigma \cap
  \CBall(x,\tfrac 12 \Cr{beta-rad})$ we have
  \begin{displaymath}
    \dgras(T_x\Sigma, T_y\Sigma) \le \Cr{tan-osc} E^{1/\kappa} |x-y|^{\tau} \,.
  \end{displaymath}
\end{lem}

\begin{proof}
  Let $y \in \Sigma \cap \CBall(x,\frac 12 \Cr{beta-rad})$. Set
  $r_0 := 2 |x-y|$ and $r_1 = |x-y|$. Choose any $H_0 \in \BAP(x,r_0)$ and
  any $H_1 \in \BAP(y,r_1)$. From Lemma~\ref{lem:bap-osc} we have
  \begin{displaymath}
    \dgras(H_0,H_1) \le \Cr{bap-osc}E^{1/\kappa} r_0^{\tau} \,.
  \end{displaymath}
  On the other hand Corollary~\ref{cor:tan-dist} says that
  \begin{displaymath}
    \dgras(T_x\Sigma, H_0) \le \Cr{tan-dist} E^{1/\kappa} r_0^{\tau}
    \qquad \text{and} \qquad
    \dgras(T_y\Sigma, H_1) \le \Cr{tan-dist} E^{1/\kappa} r_0^{\tau} \,.
  \end{displaymath}
  Putting these estimates together we obtain
  \begin{align*}
    \dgras(T_x\Sigma, T_y\Sigma) 
    &\le \dgras(T_x\Sigma,H_0) + \dgras(H_0,H_1) + \dgras(H_1,T_y\Sigma) \\
    &\le ( \Cr{bap-osc} + 2 \Cr{tan-dist} ) E^{1/\kappa} r_0^{\tau}
    = \Cr{tan-osc} E^{1/\kappa} |x-y|^{\tau} \,,
  \end{align*}
  where $\Cr{tan-osc} := \Cr{bap-osc} + 2 \Cr{tan-dist}$.
\end{proof}

\mysubsection{The parameterizing function $F_x$}

Combining Corollary~\ref{cor:tan-point} and Lemma~\ref{lem:tan-osc} one can see
that if we have two distinct points $y,z \in \Sigma$ such that $y-z \perp
T_x\Sigma$ and $|y-z| \lesssim |x-y|$ then the tangent plane $T_y\Sigma$ must
form a large angle with the plane $T_x\Sigma$. Such situation can only happen
far away from $x$ because of the bound on the oscillation of tangent
planes. Hence we have the following
\begin{lem}
  \label{lem:1point-rad}
  Choose any point $x \in \Sigma$. There exists a radius $\Cl[R]{graph-rad}
  > 0$ such that if $y,z \in \Sigma \cap \CBall(x,\frac 12
  \Cr{beta-rad})$ and $(y-z) \perp T_x\Sigma$, then necessarily
  $\max\{|x-y|,|x-z|\} > \Cr{graph-rad}$.
\end{lem}

\begin{proof}
  Choose two points $y,z \in \Sigma \cap \CBall(x,\frac 12
  \Cr{beta-rad})$ such that $(z-y) \perp T_x\Sigma$. Without loss of
  generality we can assume that $|x-y| \ge |x-z|$. First we estimate the
  distance $|y-z|$ using Corollary~\ref{cor:tan-point}. We have
  \begin{align}
    \label{est:yz}
    |y-z| = |Q_x(y-z)|
    &\le |Q_x(y-x)| + |Q_x(x-z)| \\
    &\le \Cr{lip-const} |y-x| + \Cr{lip-const} |x-z| \le 2 \Cr{lip-const} |x-y| \notag \,.
  \end{align}
  Set $\tilde{\Cr{graph-rad}} := \frac{\Cr{beta-rad}}{4\Cr{lip-const}}$. If
  $|x-y| \le \tilde{\Cr{graph-rad}}$, then $\Cr{lip-const} |x-y| \le \tfrac 12
  \Cr{beta-rad}$. Hence $|y-z| \le \tfrac 12 \Cr{beta-rad}$ and we can use
  Corollary~\ref{cor:tan-point} once again to estimate the distance between
  $T_y\Sigma$ and~$z$.

  Using the definition of $\dgras$ we may write
  \begin{align}
    \label{est:TxTy-lower}
    \dgras(T_x\Sigma, T_y\Sigma)
    &\ge |z-y|^{-1} |\pi_x(z-y) - \pi_y (z-y)|
    = |z-y|^{-1} |\pi_y (z-y)| \\
    &\ge |z-y|^{-1} \left( |z-y| - |Q_y(z-y)| \right) \notag \\ 
    &\ge |z-y|^{-1} \left( |z-y| - \Cr{tan-point}E^{1/\kappa}|z-y|^{1+\tau} \right) \notag \\
    &= 1 - \Cr{tan-point}E^{1/\kappa}|z-y|^{\tau} \notag \,.
  \end{align}
  On the other hand Lemma~\ref{lem:tan-osc} gives us
  \begin{equation}
    \label{est:TxTy-upper}
    \dgras(T_x\Sigma, T_y\Sigma) \le \Cr{tan-osc} E^{1/\kappa} |x-y|^{\tau} \,.
  \end{equation}
  Putting these two estimates together we have
  \begin{displaymath}
    1 - \Cr{tan-point} E^{1/\kappa} |z-y|^{\tau} 
    \le \dgras(T_x\Sigma, T_y\Sigma) 
    \le \Cr{tan-osc} E^{1/\kappa} |x-y|^{\tau} \,.
  \end{displaymath}
  By \eqref{est:yz},
  \begin{displaymath}
    1 - \Cr{tan-point} E^{1/\kappa} (2\Cr{lip-const})^{\tau} |x-y|^{\tau} 
    \le \Cr{tan-osc}  E^{1/\kappa} |x-y|^{\tau} \,.
  \end{displaymath}
  Hence
  \begin{displaymath}
    |x-y| \ge E^{-1/\lambda} (\Cr{tan-osc} + \Cr{tan-point} (2\Cr{lip-const})^{\tau})^{-1/\tau} \,.
  \end{displaymath}
  We may set
  \begin{align}
    \label{def:graph-rad}
    \Cr{graph-rad} = \Cr{graph-rad}(E,m,p,M_{\Sigma},A_{\Sigma},R_{\Sigma})
    &:= \min \left\{
      \frac 12 E^{-1/\lambda} (\Cr{tan-osc} + \Cr{tan-point}(2\Cr{lip-const})^{\tau})^{-1/\tau},
      \tilde{\Cr{graph-rad}}
    \right\} \\
    &= \min \left\{
      \frac 12 E^{-1/\lambda} (\Cr{tan-osc} + \Cr{tan-point}(2\Cr{lip-const})^{\tau})^{-1/\tau},
      \frac{\Cr{beta-rad}}{4\Cr{lip-const}}
    \right\} \notag \,.
  \end{align}
\end{proof}

Let us define
\begin{equation}
  \label{def:smooth-rad}
  \Cr{smooth-rad} = \Cr{smooth-rad}(E,m,p,M_{\Sigma},A_{\Sigma},R_{\Sigma})
  := \frac 12 \min\{
  E^{-1/\lambda} (2 \Cr{tan-osc})^{-1/\tau},
  \Cr{graph-rad},
  \tfrac 12 \Cr{beta-rad}
  \} \,.
\end{equation}
This definition assures that for any $y,z \in \Sigma \cap
\CBall(x,\Cr{smooth-rad})$ we have
\begin{displaymath}
  \dgras(T_y\Sigma, T_z\Sigma) \le \tfrac 12 \,.
\end{displaymath}
Here, the radius $\Cr{smooth-rad}$ depends on $A_{\Sigma}$, $M_{\Sigma}$ and
$R_{\Sigma}$ but at the end of this section we shall prove that one can drop
these dependencies just by showing that $A_{\Sigma}$, $M_{\Sigma}$ and
$R_{\Sigma}$ can be expressed solely in terms if $E$, $m$ and $p$.

\begin{cor}
  \label{cor:graph}
  For each $x \in \Sigma$ and each $y \in \Sigma \cap
  \CBall(x,\Cr{smooth-rad})$ the point $y$ is the only point in the
  intersection $\Sigma \cap (y + T_x\Sigma^{\perp}) \cap
  \CBall(x,\Cr{smooth-rad})$. Therefore $(\Sigma - x) \cap
  \CBall_{\Cr{smooth-rad}}$ is a graph of the function
  \begin{align}
    \label{def:Fx}
    F_x : \tilde{\Dom}_x &\to T_x\Sigma^{\perp} \cap \CBall_{\Cr{smooth-rad}}
    \quad \text{defined by}\\
    F_x(w) + w &= (\Sigma - x) \cap (w+T_x\Sigma^{\perp}) \cap \CBall_{\Cr{smooth-rad}} \notag \,,
  \end{align}
  where $\tilde{\Dom}_x \subseteq T_x\Sigma$ is defined as
  \begin{displaymath}
    \tilde{\Dom}_x := \pi_x((\Sigma - x) \cap \CBall_{\Cr{smooth-rad}}) \,.
  \end{displaymath}
\end{cor}

\begin{lem}
  \label{lem:cont}
  For each $x \in \Sigma$ the function $F_x : \tilde{\Dom}_x \to T_x\Sigma^{\perp}$ is
  continuous.
\end{lem}

\begin{proof}
  Set $\tilde{\Sigma} := (\Sigma - x) \cap \CBall_{\Cr{smooth-rad}}$. Since
  $\tilde{\Sigma}$ is an intersection of two compact sets it is compact. By
  definition of $\tilde{\Sigma}$ and $\tilde{\Dom}_x$ we know that
  $\pi_x|_{\tilde{\Sigma}} : \tilde{\Sigma} \to \tilde{\Dom}_x$ is a
  bijection. It is also continuous because it is a restriction of a continuous
  function $\pi_x$. Therefore $\pi_x|_{\tilde{\Sigma}}$ is a homeomorphism and
  the inverse $f_x := (\pi_x|_{\tilde{\Sigma}})^{-1} : \tilde{\Dom}_x \to
  \tilde{\Sigma}$ is also continuous. Note that $F_x(w) = f_x(w) - w =
  Q_x(f_x(w))$ is a composition of continuous functions, hence it is
  continuous.
\end{proof}

Up to now we do not know much about the set $\tilde{\Dom}_x$. We know that $0
\in \tilde{\Dom}_x$, so it is not empty but it might happen that there are
only a few other points in $\tilde{\Dom}_x$. Now we will prove that
$\tilde{\Dom}_x$ contains the whole disc
$\CDisc_{\frac 12 \Cr{smooth-rad}} := \CBall_{\frac 12 \Cr{smooth-rad}} \cap T_x\Sigma$.

\begin{lem}
  \label{lem:domain}
  The set $\Dom_x := \tilde{\Dom}_x \cap \CBall_{\frac 12\Cr{smooth-rad}}$
  coincides with the closed disc $\CDisc_{\frac 12 \Cr{smooth-rad}} :=
  \CBall_{\frac 12\Cr{smooth-rad}} \cap T_x\Sigma$.
\end{lem}

\begin{proof}
  We will show that $\Dom_x$ is both closed and open in $\CDisc_{\frac 12
    \Cr{smooth-rad}}$. First note that $\tilde{\Dom}_x$ is the image of a
  compact set $(\Sigma - x) \cap \CBall_{\Cr{smooth-rad}}$ under a continuous
  mapping $\pi_x$, so it is compact, hence closed in $T_x\Sigma$. Therefore
  $\tilde{\Dom}_x \cap \CDisc_{\frac 12 \Cr{smooth-rad}}$ is closed in
  $\CDisc_{\frac 12 \Cr{smooth-rad}}$ but $\tilde{\Dom}_x \cap \CDisc_{\frac
    12 \Cr{smooth-rad}} = \Dom_x$.

  Now we need to prove that $\Dom_x$ is also open in $\CDisc_{\frac 12
    \Cr{smooth-rad}}$. We do that by contradiction. Assume that $\Dom_x$ is
  not open in $\CDisc_{\frac 12 \Cr{smooth-rad}}$. Then there exists a point
  $w \in \Dom_x$ such that for all $r > 0$ we have $\Ball(w,r) \cap \Dom_x \ne
  \Ball(w,r) \cap \CDisc_{\frac 12 \Cr{smooth-rad}}$. Hence for all $r > 0$
  there exists a point $u \in \Ball(w,r) \cap \CDisc_{\frac 12
    \Cr{smooth-rad}} \setminus \Dom_x$. Fix $r > 0$ so small that $\Ball(w,4r)
  \subseteq \Ball_{\Cr{smooth-rad}}$. We can always do that because $|w| \le
  \tfrac 12 \Cr{smooth-rad}$. Fix some $u \in \Ball(w,r) \cap \CDisc_{\frac 12
    \Cr{smooth-rad}} \setminus \Dom_x$.  There exists $\rho > 0$ such that
  $\Ball(u,\rho) \subseteq \Ball(w,2r) \subseteq \Ball_{\Cr{smooth-rad}}$ and
  $\Ball(u,\rho) \cap \Dom_x = \emptyset$ and $\CBall(u,\rho) \cap \Dom_x \ne
  \emptyset$. In other words we take $\rho$ to be the distance of $u$ from
  $\Dom_x$ (see Figure~\ref{F:domain-disc}).
  \begin{displaymath}
    \rho := \sup \{ s > 0 : \Ball(u,s) \cap \Dom_x = \emptyset \} \le r \,.
  \end{displaymath}
  Set $z := F_x(w) + w \in (\Sigma - x) \cap \Ball_{\Cr{smooth-rad}}$ and
  choose any $v \in \CBall(u,\rho) \cap \Dom_x$. Set $y := F_x(v) + v
  \in (\Sigma - x) \cap \Ball_{\Cr{smooth-rad}}$. Directly from the definition
  of $\tilde{\Dom}_x$ we obtain
  \begin{equation}
    \label{eq:empty-region}
    \forall \tilde{x} \in T_x\Sigma \cap \Ball(u,\rho)\ \ 
    (\Sigma - x) \cap (\tilde{x}+T_x\Sigma^{\perp}) \cap \Ball_{\Cr{smooth-rad}} = \emptyset \,.
  \end{equation}

  \begin{figure}[!htb]
    \centering
    \includegraphics{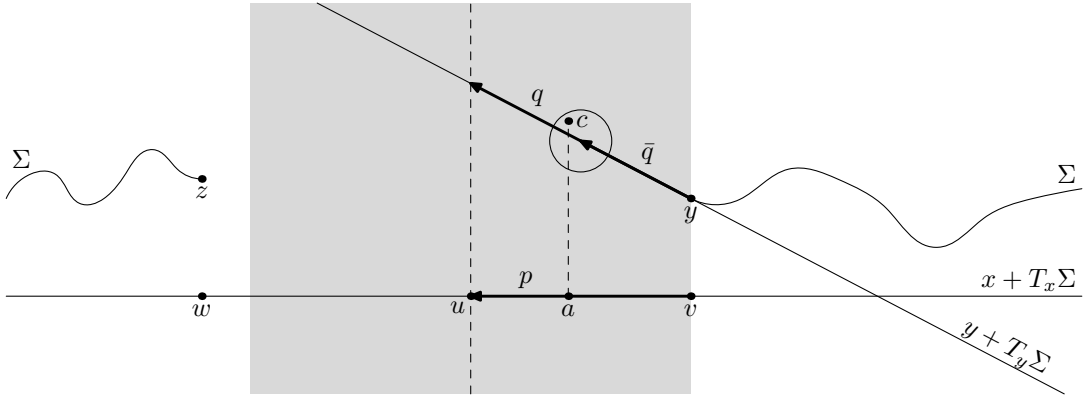}
    \caption{There can not be any points of $\Sigma$ in the grey area.}
    \label{F:domain-disc}
  \end{figure}
  Recalling the definition of $\Cr{smooth-rad}$ we see that
  \begin{equation}
    \label{est:domain-TxTy}
    \dgras(T_x\Sigma,T_y\Sigma) \le \tfrac 12\,,
  \end{equation}
  hence $\pi_x$ gives an isomorphism
  (cf. Remark~\ref{rem:orth-angle}) between $T_x\Sigma$ and $T_y\Sigma$. Set
  $p := u-v \in T_x\Sigma$. Note that $|p| = |u-v| = \rho$. Let $q \in
  T_y\Sigma$ be such that $\pi_x(q) = p$. Because of the angle
  estimate \eqref{est:domain-TxTy} we know that
  \begin{displaymath}
    \forall \bar{x} \in T_y\Sigma
    \quad
    \tfrac 12 |\bar{x}| \le |\pi_x(\bar{x})| \le |\bar{x}| \,.
  \end{displaymath}
  In particular $|p| \le |q| \le 2|p| = 2\rho$. Set $\bar{q} := \frac12 q$,
  so that $|\bar{q}| \le \rho$. Because $\rho \le \Cr{smooth-rad} \le
  \Cr{beta-rad}$ we know that $\ntheta(y,\rho) \le \frac 14$. Hence there
  exists a point $c \in (\Sigma - x) \cap \CBall(y+\bar{q},\frac 14
  \rho)$. Set $a := \pi_x(c)$. We estimate the distance between
  $a \in T_x\Sigma$ and $u \in T_x\Sigma$.
  \begin{align*}
    |a - u| = |\pi_x(c - u)|
    &\le |\pi_x(c - (y + \bar{q}))| + |\pi_x((y + \bar{q}) - u)| \\
    &\le |c - (y + \bar{q})| + |v + \pi_x(\bar{q}) - u| \\
    &\le \tfrac 14 \rho + |(v-u) + \tfrac 12 (u-v)| \le \tfrac 34 \rho < \rho \,.
  \end{align*}
  We have found a point $c \in (\Sigma - x) \cap (a + T_x\Sigma^{\perp}) \cap
  \Ball_{\Cr{smooth-rad}}$ with $|a-u| < \rho$ which contradicts condition
  \eqref{eq:empty-region}, so $\Dom_x$ must be open.
\end{proof}

\begin{cor}
  If $\Sigma$ is a manifold, it must be closed, i.e. $\partial \Sigma =
  \emptyset$.
\end{cor}
It follows from the way we defined $F_x$, that
\begin{cor}
  \label{cor:Fx-image}
  For each $w_1,w_2 \in \Dom_x$ the points $y := F_x(w_1) + w_1$ and $z :=
  F_x(w_2) + w_2$ lie on $\Sigma - x$ and satisfy $|y-z| \in
  \CBall_{\Cr{smooth-rad}}$, hence
  \begin{displaymath}
    \dgras(T_y\Sigma,T_z\Sigma) \le \tfrac 12 \,.
  \end{displaymath}
\end{cor}

\mysubsection{The derivative $DF_x$}
In the following lemma we will need estimates on the norms of projections
between $T_x\Sigma$ and $T_y\Sigma$. For $y \in (\Sigma - x) \cap
\Ball_{\Cr{smooth-rad}}$ we have $\dgras(T_x\Sigma, T_y\Sigma) \le \frac
12$, so from Remark~\ref{rem:orth-angle}, we know that
\begin{align*}
  \pi_x|_{T_y\Sigma} &: T_y\Sigma \to T_x\Sigma \\
  \text{and} \quad
  Q_x|_{T_y\Sigma^{\perp}} &: T_y\Sigma^{\perp} \to T_x\Sigma^{\perp}
\end{align*}
are isomorphisms. Set
\begin{align*}
  L_y &:= (\pi_x|_{T_y\Sigma})^{-1} : T_x\Sigma \to T_y\Sigma \\
  \text{and} \quad
  K_y &:= (Q_x|_{T_y\Sigma^{\perp}})^{-1} : T_x\Sigma^{\perp} \to T_y\Sigma^{\perp} \,.
\end{align*}
In other words $L_y$ is on oblique projection onto $T_y\Sigma$ along
$T_x\Sigma^{\perp}$ and $K_y$ is an oblique projection onto
$T_y\Sigma^{\perp}$ along $T_x\Sigma$. Using the fact that $\dgras(T_x\Sigma,
T_y\Sigma) \le \frac 12$ we obtain
\begin{align*}
  \forall y \in (\Sigma - x) \cap \Ball_{\Cr{smooth-rad}}\ \ 
  \forall v \in T_y\Sigma \ \  
  \tfrac 12 |v| &\le |\pi_x(v)| \le |v| \\
  \text{and} \quad
  \forall y \in (\Sigma - x) \cap \Ball_{\Cr{smooth-rad}}\ \ 
  \forall w \in T_y\Sigma^{\perp} \ \  
  \tfrac 12 |w| &\le |Q_x(w)| \le |w| \,.
\end{align*}
Hence (cf. Remark~\ref{rem:orth-angle})
\begin{align}
  \label{est:Q-norm}
  \forall y \in (\Sigma - x) \cap \Ball_{\Cr{smooth-rad}}\ \ \|K_y\| &\le 2 \\
  \label{est:pi-norm}
  \forall y \in (\Sigma - x) \cap \Ball_{\Cr{smooth-rad}}\ \ \|L_y\| &\le 2 \,.
\end{align}

Note that $L_y$ and $K_y$ are oblique projections and should be understood as
restrictions of mappings $\Rn \to \Rn$ to planes $T_x\Sigma$ and
$T_x\Sigma^{\perp}$ respectively. When we write $\|L_y\|$ and $\|K_y\|$ we
always mean the operator norms taken on $T_x\Sigma$ and $T_x\Sigma^{\perp}$
respectively, so $\|L_y\| = \sup \{ |L_y(u)| : u \in \Sphere \cap T_x\Sigma \}$
and $\|K_y\| = \sup \{ |K_y(u)| : u \in \Sphere \cap T_x\Sigma^{\perp} \}$. For
$z \in \Sigma$ we denote the inclusion mapping by
\begin{displaymath}
  J_z : T_z\Sigma \hookrightarrow \Rn \,.
\end{displaymath}

\begin{lem}
  \label{lem:smooth}
  For each $x \in \Sigma$ the function $F_x : \Dom_x \to T_x\Sigma^{\perp}$ is
  differentiable. Let $w \in \Dom_x \subseteq T_x\Sigma$ and set $y = F_x(w) +
  w$. The differential $DF_x$ at $w$ is then given by (see Figure~\ref{F:DF-def})
  \begin{equation}
    \label{def:DFx}
    DF_x(w) := Q_x \circ J_y \circ L_y = J_y \circ L_y - J_x \,,
  \end{equation}
  In particular this gives $DF_x(0) = 0$.
\end{lem}
By an abuse of notation we shall identify $J_y \circ L_y$ with $L_y$, so that we
can write
\begin{displaymath}
  DF_x(w) = L_y - J_x \,.
\end{displaymath}

\begin{figure}[!htb]
  \centering
  \includegraphics{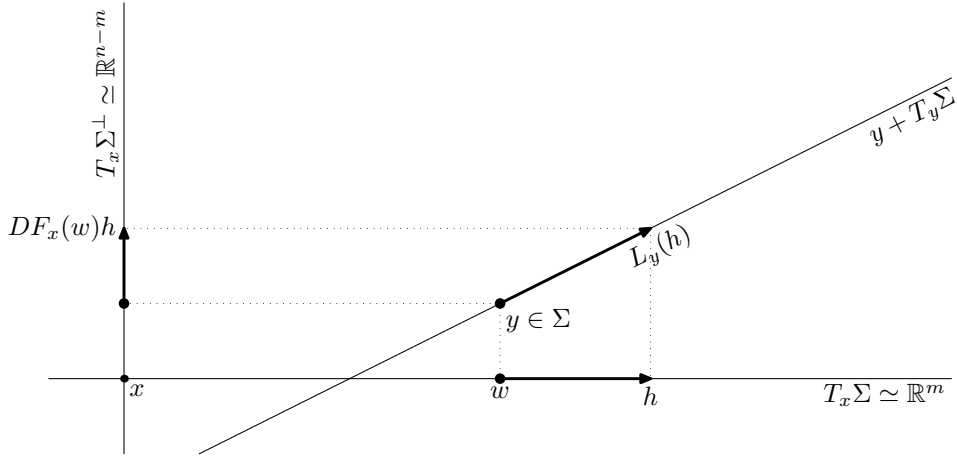}
  \caption{We define $DF_x(w)$ to be the composition of the oblique projection
    onto $T_y\Sigma$, where $y = F_x(w) + w$, with the orthogonal projection
    onto $T_x\Sigma^{\perp}$.}
  \label{F:DF-def}
\end{figure}

\begin{proof}
  Fix some $h \in \tilde{\Dom}_x \subseteq T_x\Sigma$ with $|h|$ small. We define
  \begin{align*}
    y &:= F_x(w) + w \in \Sigma - x \,, 
    \quad  z := F_x(w+h) + (w+h) \in \Sigma - x \\
    \text{and} \quad 
    u &:= F_x(w+h) - F_x(w) - DF_x(w)h = (z - y) - L_yh \in T_x\Sigma^{\perp} \,.
  \end{align*}
  We need to show that $|u|/|h| \to 0$ when $|h| \to 0$. Because $L_y h \in
  T_y\Sigma$, we have $Q_y(u) = Q_y(z-y)$, but $z$ lies
  on $\Sigma - x$, so we can estimate its distance from $T_y\Sigma$ using
  Corollary~\ref{cor:tan-point}.
  \begin{displaymath}
    \dist(z, y+T_y\Sigma) = |Q_y(z-y)| \le \Cr{tan-point} E^{1/\kappa} |z-y|^{1+\tau} \,.
  \end{displaymath}
  We know that $Q_y|_{T_x\Sigma}$ is an isomorphism and $K_y :
  T_y\Sigma^{\perp} \to T_x\Sigma^{\perp}$ is its inverse with $\|K_y\| \le
  2$, so we have the estimate
  \begin{displaymath}
    |u| = |K_y(Q_y(u))|
    = |K_y(Q_y(z-y))| \le \|K_y\| |Q_y(z-y)|
    \le 2 \Cr{tan-point} E^{1/\kappa} |z-y|^{1+\tau} \,.
  \end{displaymath}
  Now we only need to estimate $|z-y|$. Since $\|L_y\| \le 2$ we have
  \begin{displaymath}
    |z-y| = |h + L_y h + u| 
    \le (1 + \|L_y\|)|h| + |u| 
    \le 3|h| + 2 \Cr{tan-point} E^{1/\kappa} |z-y|^{1+\tau} \,,
  \end{displaymath}
  hence
  \begin{equation}
    \label{eq:zy-dist}
    |z-y| \le \frac{3}{1 - 2 \Cr{tan-point} E^{1/\kappa} |z-y|^{\tau}}|h| \,.
  \end{equation}
  Lemma~\ref{lem:cont} says that $F_x$ is continuous, so we can choose $\rho >
  0$ so small, that for each $h$ with $|h| \le \rho$ we have $|z-y|^{\tau} \le
  \frac 14 (2 \Cr{tan-point} E^{1/\kappa})^{-1}$. Then from~\eqref{eq:zy-dist} we obtain
  $|z-y| \le 4|h|$. With that estimate we can write
  \begin{displaymath}
    |h|^{-1} |F_x(w+h) - F_x(w) - DF_x(w)h| = \tfrac{|u|}{|h|}
    \le 2 \Cr{tan-point} E^{1/\kappa} (4|h|)^{\tau} \xrightarrow{h \to 0} 0 \,,
  \end{displaymath}
  so our definition of $DF_x(w)$ is correct.
\end{proof}

\begin{lem}
  \label{lem:DF-holder}
  For each $x \in \Sigma$ the differential $DF_x$ is H\"older continuous with
  H\"older exponent $\tau$ and H\"older norm bounded by some constant
  $\Cr{holder-norm} = \Cr{holder-norm}(m,p,A_{\Sigma},R_{\Sigma},M_{\Sigma})$,
  i.e.
  \begin{equation}
    \label{est:DF-holder}
    \forall x \in \Sigma \,
    \forall w_0,w_1 \in \Dom_x \quad
    \|DF_x(w_0) - DF_x(w_1)\| \le \Cr{holder-norm} E^{1/\kappa} |w_0 - w_1|^{\tau} \,.
  \end{equation}
\end{lem}

\begin{proof}
  Choose two points $w_0,w_1 \in \Dom_x$. As in the previous proof we define
  \begin{align*}
    y &:= F_x(w_0) + w_0 \in (\Sigma - x) \cap \CBall_{\Cr{smooth-rad}} \\
    \text{and} \quad 
    z &:= F_x(w_1) + w_1 \in (\Sigma - x) \cap \CBall_{\Cr{smooth-rad}} \,.
  \end{align*}
  Note that
  \begin{displaymath}
    \| DF_x(w_1) - DF_x(w_0) \| = \| L_z - L_y \| \,.
  \end{displaymath}
  Choose some unit vector $h \in T_x\Sigma \cap \Sphere$. Let $u := L_y(h)$ and
  $v := L_z(h)$. Note that $(u-v) \in T_x\Sigma^{\perp}$. Since the points $y$
  and $z$ lie in $\CBall(x,\Cr{smooth-rad})$ we have
  $\dgras(T_x\Sigma,T_y\Sigma) \le \frac 12$ and $\dgras(T_x\Sigma,T_z\Sigma)
  \le \frac 12$ and $\dgras(T_y\Sigma,T_z\Sigma) \le \frac
  12$. Estimates~\eqref{est:pi-norm} and~\eqref{est:Q-norm} give us the
  following
  \begin{displaymath}
    \|L_y\| \le 2 \,, \quad \|K_y\| \le 2 \,,
    \quad
    \|L_z\| \le 2 \quad \text{and } \|K_z\| \le 2 \,.
  \end{displaymath}
  Hence $|u| \le 2 |h|$ and $|v| \le 2 |h|$ and we obtain
  \begin{align*}
    |u-v| &= |K_z(Q_z(u-v))| \\
    &\le 2 |Q_z(u-v)| 
    = 2 |Q_z(u)| \\
    &\le 2 |u| \dgras(T_z\Sigma,T_y\Sigma) 
    \le 4 |h| \dgras(T_z\Sigma,T_y\Sigma) \\
    &\le 4 \Cr{tan-osc} E^{1/\kappa} |z-y|^{\tau} \,.
  \end{align*}
  This gives
  \begin{displaymath}
    \| DF_x(w_1) - DF_x(w_0) \| 
    \le 4 \Cr{tan-osc} E^{1/\kappa} |z-y|^{\tau}
  \end{displaymath}
  We only need to express the distance $|z-y|$ in terms of $|w_1-w_0|$. Note
  that the point $z$ is close to the tangent plane $y + T_y\Sigma$. More
  precisely from Corollary~\ref{cor:tan-point}
  \begin{align}
    |Q_y(z-y)| &\le \Cr{tan-point} E^{1/\kappa} |z-y|^{1+\tau} \quad \text{which implies} \notag \\
    \label{est:zy-proj}
    |\pi_y(z-y)| &\ge |z-y| (1 - \Cr{tan-point} E^{1/\kappa} |z-y|^{\tau}) \,.
  \end{align}
  \begin{figure}[!htb]
    \centering
    \includegraphics{ahlreg10.mps}
    \caption{The length $|y-z|$ is comparable with $|w_0 - w_1|$ because $z$
      lies close to $T_y\Sigma$ and the angle $\dgras(T_x\Sigma,T_y\Sigma)$ is
      bounded by $\frac 12$.}
    \label{F:DF-holder}
  \end{figure}
  Let 
  \begin{align*}
    b &:= y + L_y(w_1 - w_0) \in (y + T_y\Sigma) \,, \\
    c &:= y + \pi_y(z-y) \in (y + T_y\Sigma) \\
    \text{and} \quad
    w_2 &:= w_1 + \pi_x(c-z) = w_0 + \pi_x(c-y) \in T_x\Sigma \,.
  \end{align*}
  The configuration of points $b$, $c$, $w_1$ and $w_2$ is presented on
  Figure~\ref{F:DF-holder}. Now we have
  \begin{align}
    w_2 - w_0 &= \pi_x(\pi_y(z-y)) \qquad \text{which implies} \notag\\
    \label{est:w2w0}
    2 |w_2 - w_0| &\ge |\pi_y(z-y)| \ge |z-y| (1 - \Cr{tan-point} E^{1/\kappa} |z-y|^{\tau}) \,.
  \end{align}
  Of course $|w_1 - w_0| \ge |w_2 - w_0| - |w_2 - w_1|$, so we only need to
  estimate $|w_2 - w_1| = |\pi_x(c-z)|$. Note that (see
  Figure~\ref{F:DF-holder})
  \begin{align}
    \label{eq:zb-decomp-Q}
    z-c &= (z-y)-(c-y) = Q_y(z-y) \\
    &= Q_y(z - b + b - y) = Q_y(z-b) \notag \\
    \label{eq:zb-decomp-pi}
    \text{and} \quad
    c-b &= (z-b)-(z-c) = \pi_y(z-b) \,.
  \end{align}
  Since $\pi_x(z-y) = \pi_x(b-y) = (w_1 - w_0)$, we have
  $\pi_x(z-b) = 0$, so $(z-b) \in T_x\Sigma^{\perp}$ and we can
  use~\eqref{est:pi-norm} and~\eqref{est:Q-norm} obtaining
  \begin{align*}
    |z - b| = |K_y(z - c)| \le 2 |z - c| \,.
  \end{align*}
  From~\eqref{eq:zb-decomp-Q} and~\eqref{eq:zb-decomp-pi} we know that $(z-b) =
  (z-c) + (c-b)$ and that $(z-c) \perp (c-b)$. Hence
  \begin{align}
    \label{est:w2w1}
    |w_2 - w_1|
    &\le |L_y(w_2 - w_1)| = |c-b| \\
    &= \sqrt{|z-b|^2 - |z-c|^2}
    \le \sqrt 3 |z-c| = \sqrt 3 |Q_y(z-y)| \notag \,.
  \end{align}
  Using~\eqref{est:zy-proj} and~\eqref{est:w2w0} and~\eqref{est:w2w1} we
  obtain
  \begin{align*}
    |w_1 - w_0| &\ge |w_2 - w_0| - |w_2 - w_1| \\
    &\ge \tfrac 12 |z-y| (1 - \Cr{tan-point} E^{1/\kappa} |z-y|^{\tau}) - \sqrt 3 |Q_y(z-y)| \\
    &\ge \tfrac 12 |z-y| (1 - \Cr{tan-point} E^{1/\kappa} |z-y|^{\tau} - \sqrt 3 \Cr{tan-point} E^{1/\kappa} |z-y|^{\tau}) \\
    &\ge \tfrac 12 |z-y| (1 - 3\Cr{tan-point} E^{1/\kappa} |z-y|^{\tau}) \,.
  \end{align*}
  Therefore
  \begin{align*}
    |z - y| &\le |w_1 - w_0| \tfrac{2}{1 - 3 \Cr{tan-point} E^{1/\kappa} |z-y|^{\tau}}
    \quad \text{and finally}\\
    \| DF_x(w_1) - DF_x(w_0) \| 
    &\le 4 \Cr{tan-osc} E^{1/\kappa} |z-y|^{\tau} \\
    &\le 4 \Cr{tan-osc} E^{1/\kappa} \left(
      \frac{2}{1 - 3\Cr{tan-point} E^{1/\kappa} |z-y|^{\tau}}
    \right)^{\tau} |w_1 - w_0|^{\tau} \,.
  \end{align*}
  Since $z,y \in \CBall_{\Cr{smooth-rad}}$ we have $|z-y| \le
  2\Cr{smooth-rad}$, so $|z-y| \le (2 \Cr{tan-osc} E^{1/\kappa})^{-1}$ and we can write
  \begin{displaymath}
    \| DF_x(w_1) - DF_x(w_0) \| 
    \le 4 \Cr{tan-osc} E^{1/\kappa} \left(
      \frac{4\Cr{tan-osc}}{2\Cr{tan-osc} - 3\Cr{tan-point}}
    \right)^{\tau} |w_1 - w_0|^{\tau}
    := \Cr{holder-norm} E^{1/\kappa} |w_1 - w_0|^{\tau} \,.
  \end{displaymath}
  We should still check whether $\Cr{holder-norm}$ is positive and this happens
  only if $2\Cr{tan-osc} - 3\Cr{tan-point} > 0$. Let us recall the definitions
  of all needed constants and calculate
  \begin{align*}
    2\Cr{tan-osc} - 3\Cr{tan-point} 
    &= 2 (\Cr{bap-osc} + 2 \Cr{tan-dist}) - 3 (\Cr{tan-dist} + \Cr{beta-est}) \\
    &= 2 (\Cr{bap-osc} + \Cr{tan-dist} - 3 \Cr{beta-est}) \\
    &= \tfrac{16}{3} (M_{\Sigma} + 2) \Cr{dist-ang} \Cr{beta-est}
      + \Cr{bap-osc} \frac{2^{1+\tau}}{2^{\tau}-1} 
      - 3 \Cr{beta-est} \\
    &= \tfrac{16}{3} (M_{\Sigma} + 2) \Cr{dist-ang} \Cr{beta-est}
      + \frac{2^{1+\tau}}{2^{\tau}-1} \tfrac 83 (M_{\Sigma} + 2) \Cr{dist-ang} \Cr{beta-est}
      - 3 \Cr{beta-est} \\
    &= \tfrac 13 \Cr{beta-est} \left(
      16 (M_{\Sigma} + 2) \Cr{dist-ang}
      + 8 \frac{2^{1+\tau}}{2^{\tau}-1} (M_{\Sigma} + 2) \Cr{dist-ang}
      - 9
    \right) \,.
  \end{align*}
  The constants $M_{\Sigma}$ and $\Cr{dist-ang}$ are positive and greater than
  $1$, so we certainly have $\Cr{holder-norm} > 0$. At this point
  $\Cr{holder-norm}$ depends on $A_{\Sigma}$ and $M_{\Sigma}$ but we shall see
  shortly that $A_{\Sigma}$ and $M_{\Sigma}$ can be expressed solely in terms of
  $E$, $m$ and $p$.
\end{proof}

\begin{proof}[Proof of Theorem~\ref{thm:C1tau}]
  We already proved that $\Sigma$ is a closed manifold of class $C^{1,\tau}$,
  where the size of maps ($\tfrac 12 \Cr{smooth-rad} E^{1/\kappa}$) and the
  bound for the H\"older norm of the differentials of the parameterizations
  ($\Cr{holder-norm} E^{1/\kappa}$) depend on $A_{\Sigma}$, $R_{\Sigma}$ and
  $M_{\Sigma}$. What is left to show is that we can drop the dependence on
  $A_{\Sigma}$, $R_{\Sigma}$ and $M_{\Sigma}$. We shall show that $\Sigma$ is
  actually an $m$-fine set with constants $R'_{\Sigma}$, $M'_{\Sigma}$ and
  $A'_{\Sigma}$ independent of $\Sigma$.
  
  Since $\Sigma$ is a compact, closed and smooth manifold it is
  $(\delta,m)$-admissible for any $\delta \in (0,1)$
  (cf. Example~\ref{ex:mfld}). Let us set $\delta = 1/4$. From
  Theorem~\ref{thm:uahlreg} and Corollary~\ref{cor:ind-const} we know that
  $\Sigma$ is $(\frac 14,m)$-admissible with constants $A_{\Sigma} =
  A_{\Sigma}(m) = \big(\frac{15}{16}\big)^{m/2} \omega_m$ and $R_{\Sigma} =
  \Cr{uar-rad}(E,m,p,\frac 14)$. Moreover, Theorem~\ref{thm:adm-fine} shows that
  for each $x \in \Sigma$ and each $\rho < \Cr{adm-fine-rad}(E,m,p,\frac 14)$ we
  have the estimate
  \begin{displaymath}
    \ntheta(x,\rho) \le 5 \nbeta(x,\rho) \,.
  \end{displaymath}
  Therefore we can safely set
  \begin{displaymath}
    M'_{\Sigma} = 5 \,,
    \quad
    A'_{\Sigma} = (\tfrac{\sqrt{15}}{4})^m \omega_m
    \quad and \quad
    R'_{\Sigma} = \min \big\{ \Cr{uar-rad}(E,m,p,\tfrac 14), \Cr{adm-fine-rad}(E,m,p,\tfrac 14) \big\} \,.
  \end{displaymath}
  Now the constant $A'_{\Sigma}$ depends only on $m$ and the constant
  $M'_{\Sigma}$ is absolute, so $\Cr{holder-norm}$ depends only on $m$ and
  $p$. Furthermore, recalling \eqref{def:uar-rad}, \eqref{def:adm-fine-rad},
  \eqref{def:beta-rad}, \eqref{def:graph-rad} and \eqref{def:smooth-rad} we have
  \begin{displaymath}
    \Cr{smooth-rad} = \Cr{smooth-rad}(E,m,p) = \Cr{smooth-rad-const} E^{-1/\lambda} \,,
  \end{displaymath}
  where
  \begin{multline*}
    \Cr{smooth-rad-const} = \Cr{smooth-rad-const}(m,p) := \frac 12
    \min \Bigg\{
      (2 \Cr{tan-osc})^{-1/\tau}, 
      \frac 12 (\Cr{tan-osc} + \Cr{tan-point}(2\Cr{lip-const})^{\tau})^{-1/\tau}, \\
      \frac{1}{4\Cr{lip-const}} \min \Big\{
        (4 \Cr{beta-est} M_{\Sigma})^{-1/\tau},
        \big(\Cr{uahlreg1} (\tfrac 14,m) \Cr{uahlreg2}^p(\tfrac 14,m)\big)^{1/\lambda},
        \Big( \frac{7 \sqrt 7 \Psi_0}{64 \Cr{beta-est} } \Big)^{1/\tau}
      \Big\}
    \Bigg\} \,.
  \end{multline*}
  Here $\delta = \frac 14$ so we can safely set $L = 3 \in (\sqrt 7, 4)$ and
  then in \eqref{def:adm-fine-rad} we may substitute $\gamma := \sqrt{1 -
    (L\delta)^2} = \frac{\sqrt 7}4$.
\end{proof}

\begin{rem}
  Note that the scale at which we can view $\Sigma$ as a graph of some
  $C^{1,\tau}$ function depends on the energy $\E_p(\Sigma)$. If the energy is
  big, then the radius $\Cr{smooth-rad}$ goes to zero. This behavior is exactly
  what we could expect. If the integral curvature is big, then our set $\Sigma$
  can bend really fast and it is a graph of some function only in very small
  scales.

  Similarly, if the exponent $p$ is close to $m(m+2)$, then $\lambda$ is close
  to zero and if additionally $\E_p(\Sigma) > 1$, then the scale
  $\Cr{smooth-rad}$ becomes very small. The exponent $p_0 = m(m+2)$ is critical
  just as in the Sobolev embedding theorem - for an open set $U \subseteq
  \R^{m(m+2)}$ we have $W^{2,p}(U) \subseteq C^{1,\alpha}(U)$ only for $p >
  m(m+2)$.
\end{rem}

If we follow the proof of Theorem~\ref{thm:C1tau}, we shall see that all we used
was the bound on the $\beta$-numbers of $\Sigma$. After establishing
Corollary~\ref{cor:beta-est} we did not use any properties of the $p$-energy
$\E_p(\Sigma)$. Tracing back the definitions of all the constants
$\Cr{holder-norm}$, $\Cr{bap-osc}$, $\Cr{tan-dist}$, $\Cr{tan-point}$,
$\Cr{lip-const}$ and $\Cr{tan-osc}$ we will see that they were defined only in
terms of $\Cr{beta-est}$ and some other constants which depend solely on
$M_{\Sigma}$, $A_{\Sigma}$, $m$ and $p$. Also, if we analyze
\eqref{def:beta-rad}, \eqref{def:graph-rad} and \eqref{def:smooth-rad} we shall
see, that all the radii $\Cr{smooth-rad}$ (as was defined in
\eqref{def:smooth-rad}), $\Cr{beta-rad}$ and $\Cr{graph-rad}$ were defined only
in terms of $\Cr{beta-est}$, $A_{\Sigma}$, $M_{\Sigma}$, $R_{\Sigma}$ and some
other constants depending only on $m$ and $p$. Hence, we obtain the following
\begin{cor}
  Let $\Sigma \in \F(m)$ be such that for each $x \in \Sigma$ and every $r \in
  (0,R_{\Sigma}]$ we have
  \begin{displaymath}
    \nbeta(x,r) \le L r^{\nu} \,,
  \end{displaymath}
  where $\nu \in (0,1)$ and $L > 0$ is some constant. Then $\Sigma$ is a
  closed manifold of class $C^{1,\nu}$. Moreover we can find a radius $R =
  R(L,m,p,A_{\Sigma},M_{\Sigma},R_{\Sigma},\nu)$ and a constant $K$ which
  depends only on $L$, $m$, $p$, $A_{\Sigma}$, $M_{\Sigma}$ and $\nu$ such
  that
  \begin{itemize}
  \item for each $x \in \Sigma$ the set $\Sigma \cap \CBall(x,R)$ is a graph of
    some $C^{1,\nu}$ function $F_x$
  \item and the H{\"o}lder norm of $DF_x$ is bounded above by $K$.
  \end{itemize}
\end{cor}



\mysection{Improved H\"older regularity}
\label{sec:improved-holder}

In the previous paragraph we showed that $\Sigma$ is a closed manifold of class
$C^{1,\tau}$ but $\tau$ was not an optimal exponent. Now we shall prove that for
any $o \in \Sigma$ the map $F_o$ is of class $C^{1,\alpha}$ (see
Theorem~\ref{thm:improved}), where
\begin{displaymath}
  \alpha := 1 - \tfrac{m(m+2)}{p} \,.
\end{displaymath}
For this purpose we employ a technique developed by Strzelecki, Szuma{\'n}ska
and von der Mosel in~\cite{MR2668877}.

First we show that the oscillation of $DF_o$ is roughly the same as the
oscillation of tangent planes $T_o\Sigma$. Then we choose two points $x$ and $y$
with $|x-y| \simeq r$. After that we examine the set of tuples $(x_0, \ldots,
x_m, z)$ for which the curvature $\K$ is very big. Using finiteness of
$\E_p(\Sigma)$ we prove that this set of \emph{bad parameters} $(x_0, \ldots,
x_m, z)$ has to be small in the sense of measure. Using this knowledge we are
able to find ''good'' tuples, such that for each $i,j = 1,\ldots,m$ and $i \ne
j$
\begin{displaymath}
  \varangle(x_i - x_0, x_j - x_0) \simeq \frac{\pi}2
  \quad \text{and} \quad
  |x_i - x_0| \simeq \frac rN \,.
\end{displaymath}
Moreover $(x_0,\ldots,x_m)$ is such that there are many points $z$ for which
$\K(x_0,\ldots,x_m,z)$ is not too big. If $N$ is a large number and the points
$x_i$ are chosen near $x$, then the affine plane spanned by $(x_0,\ldots,x_m)$
is close to the tangent plane $T_x\Sigma$. Therefore it suffices to estimate the
angle between the planes $X := \aff\{x_0,\ldots,x_m\}$ and $Y :=
\aff\{y_0,\ldots,y_m\}$ where the points $x_i$ and $y_i$ form ''good'' tuples
and are chosen close to $x$ and $y$ respectively. Employing the fact that there
are many points $z$ such that $\K(x_0,\ldots,x_m,z)$ and $\K(y_0,\ldots,y_m,z)$
are simultaneously small, we can derive the estimate $\dgras(X,Y) \lesssim
|x-y|^{\alpha}$.

Fix a point $o \in \Sigma$ and let $\iota \in (0,\frac 14)$ be some small
number, which we shall fix later on. For brevity of the notation let us define
\begin{displaymath}
  \Disc_r := T_o\Sigma \cap \Ball_r \,.
\end{displaymath}

Set
\begin{equation}
  \label{def:eps-rad}
  \Cl[R]{eps-rad} = \Cr{eps-rad}(E,m,p,\iota)
  := E^{-1/\lambda} \min \Big\{ \tfrac 12 \Big( \frac{\iota}{\Cr{holder-norm}} \Big)^{1/\tau}, \tfrac 14 \Cr{smooth-rad-const} \Big\}\,,
\end{equation}
then for all $x,y \in \CDisc_{3\Cr{eps-rad}}$ we have
\begin{displaymath}
  \| DF_o(x) \| \le \iota
  \quad \text{and} \quad
  |F_o(x) - F_o(y)| \le \iota |x - y| \,.
\end{displaymath}
We specify the parameterization
\begin{align*}
  \varphi : \CDisc_{3\Cr{eps-rad}} &\to \Sigma \cap \CBall(o,4\Cr{eps-rad}) \\
  \varphi(x) &:= o + F_o(x) + x \,.
\end{align*}
The oscillation of $D\varphi$ on $S \subseteq \CDisc_{3\Cr{eps-rad}}$ is defined
as
\begin{displaymath}
  \Phi(r,S) := \sup \big\{
  \| D\varphi(x) - D\varphi(y) \| :
  x,y \in S,\, |x-y| \le r
  \big\} \,.
\end{displaymath}
For $x,y \in \CDisc_{3\Cr{eps-rad}}$ we also define 
\begin{displaymath}
  \Dxy(x,y) := \Disc_{|x-y|} + \tfrac{x+y}{2} \subseteq T_o\Sigma \,.
\end{displaymath}

Now we prove that the oscillation of $D\varphi$ is, up to a constant, the same
as oscillation of $T_{\varphi(x)}\Sigma$.
\begin{lem}
  \label{lem:tp-deriv}
  There exists a constant $\Cl{ang-deriv} = \Cr{ang-deriv}(m)$ such that for any
  $x,y \in \CDisc_{3\Cr{eps-rad}}$ we have
  \begin{align}
    \label{est:Dphi-ang}
    \| D\varphi(x) - D\varphi(y) \| 
    &\le 4 \dgras(T_{\varphi(x)}\Sigma, T_{\varphi(y)}\Sigma) \\
    \label{est:ang-Dphi}
    \text{and} \quad
    \dgras(T_{\varphi(x)}\Sigma, T_{\varphi(y)}\Sigma)
    &\le \Cr{ang-deriv} \| D\varphi(x) - D\varphi(y) \| \,.
  \end{align}
\end{lem}

\begin{proof}
  To prove \eqref{est:Dphi-ang} we repeat the same argument as in the proof of
  Lemma~\ref{lem:DF-holder}. We set 
  \begin{align*}
    L_x &:= \left( \pi_o|_{T_{\varphi(x)}\Sigma} \right)^{-1} : T_o\Sigma \to T_{\varphi(x)}\Sigma &
    L_y &:= \left( \pi_o|_{T_{\varphi(y)}\Sigma} \right)^{-1} : T_o\Sigma \to T_{\varphi(y)}\Sigma \\
    K_x &:= \left( Q_o|_{T_{\varphi(x)}\Sigma^{\perp}} \right)^{-1} : T_o\Sigma^{\perp} \to T_{\varphi(x)}\Sigma^{\perp} &
    K_y &:= \left( Q_o|_{T_{\varphi(y)}\Sigma^{\perp}} \right)^{-1} : T_o\Sigma^{\perp} \to T_{\varphi(y)}\Sigma^{\perp} \,.
  \end{align*}
  For $z \in \Sigma$ we also write
  \begin{displaymath}
    J_z : T_z\Sigma \hookrightarrow \Rn \,.
  \end{displaymath}
  for the standard inclusion mapping.

  Since $\Cr{eps-rad} \le \Cr{smooth-rad}$, we know that the norms $\|L_x\|$,
  $\|L_y\|$, $\|K_x\|$ and $\|K_y\|$ are all less or equal to~$2$. We want to
  estimate (cf. \eqref{def:DFx})
  \begin{displaymath}
    \| D\varphi(x) - D\varphi(y) \| = \| DF_o(x) - DF_o(y) \| = \| J_x \circ L_x - J_y \circ L_y \| \,.
  \end{displaymath}
  By an abuse of notation we shall identify $J_z \circ L_z$ with $L_z$, so that
  we can write
  \begin{displaymath}
    \| D\varphi(x) - D\varphi(y) \| = \| L_x - L_y \| \,.
  \end{displaymath}

  Let $h \in \Sphere$ and set $u := J_x(L_x(h))$ and $v := J_y(L_y(h))$. Note that $u-v
  \in T_o\Sigma^{\perp}$ so we can write
  \begin{align*}
    |L_x(h) - L_y(h)| 
    &= |u-v| 
    = |K_x(Q_x(u-v))|
    \le 2 |Q_x(u-v)|
    = 2|Q_x(v)| \\
    &\le 2 |v| \dgras(T_{\varphi(x)}\Sigma,T_{\varphi(y)}\Sigma) 
    \le 4 \dgras(T_{\varphi(x)}\Sigma,T_{\varphi(y)}\Sigma) \,.
  \end{align*}

  The proof of \eqref{est:ang-Dphi} is based on
  Proposition~\ref{prop:red-ang}. Let $(e_1,\ldots,e_m)$ be some orthonormal
  basis of $T_o\Sigma$. For each $i := 1,\ldots,m$ set $u_i := D\varphi(x)(e_i)$
  and $v_i := D\varphi(y)(e_i)$. Then $(u_1,\ldots,u_m)$ is a basis of
  $T_{\varphi(x)}\Sigma$ and $(v_1,\ldots,v_m)$ is a basis of
  $T_{\varphi(y)}\Sigma$. Note that
  \begin{equation}
    \label{est:uv-len}
    1 - \iota \le |u_i| \le 1 + \iota\,.
  \end{equation}
  Recall that $D\varphi(x) = DF_o(x) + I$, so for $i \ne j$ we have
  \begin{align}
    \label{est:uv-product}
    |\langle u_i, u_j \rangle|
    &= |\langle DF_o(x)(e_i) + e_i, DF_o(x)(e_j) + e_j \rangle| \\
    &\le |\langle e_i, DF_o(x)(e_j) \rangle|
    + |\langle DF_o(x)(e_i), e_j \rangle|
    + |\langle DF_o(x)(e_i), DF_o(x)(e_j) \rangle| \notag \\
    &\le 2 \iota + \iota^2 < 3 \iota \notag \,.
  \end{align}
  Estimates \eqref{est:uv-len} and \eqref{est:uv-product} show that
  $(u_1,\ldots,u_m)$ is a $\red$-basis of $T_{\varphi(x)}\Sigma$ with constants
  \begin{displaymath}
    \rho = 1 \,,
    \quad
    \varepsilon = \iota
    \quad \text{and} \quad
    \delta =  3 \iota \,.
  \end{displaymath}
  Moreover
  \begin{displaymath}
    |u_i - v_i| = |D\varphi(x)(e_i) - D\varphi(y)(e_i)| 
    \le \| D\varphi(x) - D\varphi(y) \| \,,
  \end{displaymath}
  To apply Proposition~\ref{prop:red-ang} we still need to check that
  $|D\varphi(x)(e_i) - D\varphi(y)(e_i)| < 1$, which is true because $\iota \in
  (0,\frac 14)$, and we need to impose the following
  \begin{equation}
    \label{cond:iota1}
    \Cr{dist-ang} (\Cr{gs-eps} \iota + \Cr{gs-del} 3 \iota) < 1
    \quad \iff \quad
    \iota < \frac{1}{\Cr{dist-ang} (\Cr{gs-eps} + 3 \Cr{gs-del})} \,.
  \end{equation}
  Set $\iota_0 = \iota_0(m) := (2\Cr{dist-ang} (\Cr{gs-eps} + 3
  \Cr{gs-del}))^{-1}$. Choosing any $\iota \le \iota_0$ and applying
  Proposition~\ref{prop:red-ang} we obtain
  \begin{displaymath}
    \dgras(T_{\varphi(x)}\Sigma, T_{\varphi(y)}\Sigma)
    \le \Cr{ang-deriv} \| D\varphi(x) - D\varphi(y) \| \,,
  \end{displaymath}
  where $\Cr{ang-deriv} = \Cr{ang-deriv}(m) := \Cr{red-ang}(m,\iota_0(m),3\iota_0(m))$.
\end{proof}

\begin{cor}
  \label{cor:tp-deriv-osc}
  For any $x,y \in \CDisc_{3\Cr{eps-rad}}$
  \begin{displaymath}
    \dgras(T_{\varphi(x)}\Sigma, T_{\varphi(y)}\Sigma)
    \le
    \Cr{ang-deriv} \Phi(r,S) \,.
  \end{displaymath}
\end{cor}

\mysubsection{The main theorem and the strategy of the proof}
Now we can prove the main result of this section
\begin{thm}
  \label{thm:improved}
  Let $\Sigma \in \F(m)$ be such that $\E_p(\Sigma) \le E < \infty$ for some $p
  > m(m+2)$. Then $\Sigma$ is a smooth, closed manifold of class $C^{1,\alpha}$,
  where $\alpha = 1 - \frac{m(m+2)}{p}$.
  
  Moreover there exists a radius $\Cl[R]{eps-rad2}$ and a constant
  $\Cl{alpha-hol-norm}$ which depend only on $E$, $m$ and $p$ such that for each
  $o \in \Sigma$
  \begin{itemize}
  \item $\Sigma \cap \Ball(o,\Cr{eps-rad2})$ is a graph of a $C^{1,\alpha}$
    function $F_o$ defined in \S\ref{sec:tangent-planes} by
    formula~\eqref{def:Fx}
  \item and the H\"older norm of $DF_o$ is bounded above by $\Cr{alpha-hol-norm}$.
  \end{itemize}
\end{thm}

We already know that $\Sigma$ is a smooth, closed manifold of class
$C^{1,\tau}$. Now we need to improve the exponent $\tau$ to the optimal value
$\alpha$. The strategy of the proof is as follows. We want to derive an
estimate of the form
\begin{equation}
  \label{est:Phi-intu}
  \Phi(r,\Disc_R) \le \tilde{C} \Phi(\tfrac rN,\Disc_{R+r}) + \hat{C} r^{\alpha} \,.
\end{equation}
Then upon iteration we shall obtain
\begin{displaymath}
  \Phi(r,\Disc_R) \le \tilde{C}^j \Phi(\tfrac r{N^j},\Disc_{3R}) 
  + \hat{C} \sum_{i=1}^j \tilde{C}^{i-1} \left( \frac{r}{N^{i-1}} \right)^{\alpha} \,,
\end{displaymath}
for each $j \in \N$. We know a priori that $\Phi(r,\Disc_{3R}) \le \bar{C}
r^{\tau}$, hence
\begin{displaymath}
  \Phi(r,\Disc_R) \le \tilde{C}^j \bar{C} \left(\frac r{N^j} \right)^{\tau} 
  + \hat{C} \sum_{i=1}^j \tilde{C}^{i-1} \left( \frac{r}{N^{i-1}} \right)^{\alpha} \,.
\end{displaymath}
We choose $N$ big enough to ensure $\tilde{C}/N^{\alpha} \le
\tilde{C}/N^{\tau} < 1$ and we pass to the limit $j \to \infty$ obtaining
\begin{displaymath}
  \Phi(r,\Disc_R) \le \hat{C} r^{\alpha}
  \sum_{i=1}^{\infty} \Big( \frac{\tilde{C}}{N^{\alpha}} \Big)^{i-1} 
  =: \breve{C} r^{\alpha} \,.
\end{displaymath}

To prove \eqref{est:Phi-intu} we define the sets of \emph{bad parameters}
$\Sigma_0 \subseteq \CDisc_{3\Cr{eps-rad}}^{m+1}$ and show that its measure
$\HM^{m(m+1)}(\Sigma_0)$ is small. Then we find points $x_0$, \ldots, $x_m$ and
points $y_0$, \ldots, $y_m$ outside of the set of bad parameters $\Sigma_0$,
such that
\begin{displaymath}
  |x_0 - y_0| \simeq r \,,
  \quad
  |x_i - x_0| \simeq \frac rN
  \quad \text{and} \quad
  |y_i - y_0| \simeq \frac rN \,.
\end{displaymath}
Moreover $(x_1 - x_0, \ldots, x_m - x_0)$ and $(y_1 - y_0, \ldots, y_m - y_0)$
shall form almost orthogonal bases of $T_o\Sigma$. Then we define the planes 
\begin{align*}
  X &:= \opspan \{ \varphi(x_1) - \varphi(x_0), \ldots, \varphi(x_m) - \varphi(x_0) \} \\
  \text{and} \quad
  Y &:= \opspan\{ \varphi(y_1) - \varphi(y_0), \ldots, \varphi(y_m) - \varphi(y_0) \}
\end{align*}
and prove that the ''angles'' $\dgras(X,T_{\varphi(x_0)}\Sigma)$ and
$\dgras(Y,T_{\varphi(y_0)}\Sigma)$ can be bounded above by the oscillation
$\Phi(\frac rN,\CDxy(x_0,y_0))$.

Then we estimate the ''angle'' $\dgras(X,Y)$. This is the most important
ingredient of the proof, which is responsible for the appearance of $r^{\alpha}$
in our estimates. It is the point where we need to use some properties of our
discrete curvature $\K$ and the bound on the $p$-energy resulting from the fact
that $x_i$ and $y_i$ do not belong to $\Sigma_0$. We employ the fact that there
are many points
\begin{displaymath}
  z \in \CDxy(x,y) \setminus (\Sigma_1(x_1,\ldots,x_m) \cup \Sigma_1(y_1,\ldots,y_m))
\end{displaymath}
satisfying
\begin{align}
  \label{est:K-intu}
  \K(\varphi(x_0),\ldots,\varphi(x_m),\varphi(z)) &\le C |x-y|^{\frac{-m(m+2)}{p}}
  \quad \text{and simultaneously} \\
  \K(\varphi(y_0),\ldots,\varphi(y_m),\varphi(z)) &\le C |x-y|^{\frac{-m(m+2)}{p}} \,. \notag
\end{align}
We choose another $(m+1)$ points $z_0,\ldots,z_m \in \CDxy(x,y) \setminus
(\Sigma_1(x_1,\ldots,x_m) \cup \Sigma_1(y_1,\ldots,y_m))$ forming an almost
orthogonal system and we set $Z := \opspan \{ \varphi(z_1) - \varphi(z_0),
\ldots, \varphi(z_m) - \varphi(z_0) \}$. From \eqref{est:K-intu} we get
estimates on the distances 
\begin{displaymath}
  \dist(\varphi(z_i),X) \lesssim |x-y|^{1 + \alpha}
  \quad \text{and} \quad
  \dist(\varphi(z_i),Y) \lesssim |x-y|^{1+\alpha} \,.
\end{displaymath}
Next we use Proposition~\ref{prop:red-ang} to obtain the bounds $\dgras(X,Z)
\lesssim |x-y|^{\alpha}$ and $\dgras(Y,Z) \lesssim |x-y|^{\alpha}$, which
finally gives \eqref{est:Phi-intu}.

\mysubsection{Proof of Theorem~\ref{thm:improved}}
  Choose two points $x,y \in \CDisc_{\Cr{eps-rad}}$ and two big natural numbers
  $k,N \ge 4$. Set
  \begin{displaymath}
    \Kphi(x_0,\ldots,x_{m+1}) := \K(\varphi(x_0),\ldots,\varphi(x_{m+1}))
  \end{displaymath}
  and let
  \begin{align*}
    E(x,y)
    &:= \int_{\varphi(\CDxy(x,y))^{m+2}} \K^p(p_0,\ldots,p_{m+1})\ d\HM^m_{p_0} \cdots d\HM^m_{p_{m+1}} \\
    &= \int_{\CDxy(x,y)^{m+2}} \Kphi^p(x_0,\ldots,x_{m+1})
    |J\varphi(x_0)| \cdots |J\varphi(x_{m+1})| \ dx_0 \cdots dx_{m+1} \,,
  \end{align*}
  where $|J\varphi(x)| = \sqrt{\det(D\varphi(x)^* D\varphi(x))}$.
  We define the sets of \emph{bad parameters}
  \begin{displaymath}
    \Sigma_0 := \left\{
      (x_0,\ldots,x_m) \in \CDxy(x,y)^{m+1} :
      \HM^m(\Sigma_1(x_0,\ldots,x_m)) 
      > \Omega_1  \left( \tfrac{|x-y|}{kN} \right)^m
    \right\}
  \end{displaymath}
  \begin{displaymath}
    \text{and} \quad 
    \Sigma_1(x_0,\ldots,x_m) := \left\{ 
      z \in \CDxy(x,y) :
      \Kphi^p(x_0,\ldots,x_m,z) 
      >  \Omega_2 E(x,y) \left( \tfrac{kN}{|x-y|} \right)^{m(m+2)}
    \right\} \,,
  \end{displaymath}
  where $\Omega_1 := \frac 12 \omega_m$ and $\Omega_2 := \frac{2}{\omega_{m}
    \omega_{m(m+1)}}$. Since $D\varphi(x) = I + DF_o(x)$ we have $|J\varphi(x)|
  \ge 1$. Hence
  \begin{align*}
    E(x,y) 
    &\ge \int_{\CDxy(x,y)^{m+2}} \Kphi^p(x_0,\ldots,x_m,z)\ dx_0 \cdots dx_m\ dz \\
    &\ge \int_{\Sigma_0} \int_{\Sigma_1(x_0,\ldots,x_m)} \Kphi^p(x_0,\ldots,x_m,z)\ dx_0 \cdots dx_m\ dz \\
    &\ge \HM^{m(m+1)}(\Sigma_0) 
    \tfrac 12 \omega_m \left( \tfrac{|x-y|}{kN} \right)^m
    \tfrac{2}{\omega_{m} \omega_{m(m+1)}} E(x,y) \left( \tfrac{kN}{|x-y|} \right)^{m(m+2)} \\
    &= \HM^{m(m+1)}(\Sigma_0) E(x,y) \omega_{m(m+1)}^{-1} \left( \tfrac{kN}{|x-y|} \right)^{m(m+1)} \,.
  \end{align*}
  From here we obtain the estimate
  \begin{displaymath}
    \HM^{m(m+1)}(\Sigma_0) \le \omega_{m(m+1)} \left( \tfrac{|x-y|}{kN} \right)^{m(m+1)}
  \end{displaymath}

  \begin{rem}\ 
    \label{rem:close-good-params}
    \begin{itemize}
    \item For any tuple $(\tilde{x}_0,\ldots,\tilde{x}_m) \in \CDxy(x,y)^{m+1}$
      such that for each $j = 0,\ldots,m$
      \begin{displaymath}
        \left| \tilde{x}_j - \tfrac 12 (x+y) \right| \le \big(1 - \tfrac 1{kN}\big) |x-y|
      \end{displaymath}
      there exists another tuple of points $(x_0,\ldots,x_m) \in
      \CDxy(x,y)^{m+1} \setminus \Sigma_0$ such that
      \begin{displaymath}
        |x_i - \tilde{x}_i| \le \frac{|x-y|}{kN}
      \end{displaymath}
      for each $i = 0,\ldots,m$.

    \item For any tuple $(x_0,\ldots,x_m) \in \CDxy(x,y)^{m+1} \setminus
      \Sigma_0$ and any tuple $(y_0,\ldots,y_m) \in \CDxy(x,y)^{m+1} \setminus
      \Sigma_0$ and any point $\tilde{z} \in \CDxy(x,y)$
      such that
      \begin{displaymath}
        \left| \tilde{z} - \tfrac 12 (x+y) \right| \le \big(1 - \tfrac 1{kN}\big) |x-y|
      \end{displaymath}
      there exists a point $z \in \CDxy(x,y) \setminus (\Sigma_1(x_0,\ldots,x_m)
      \cup \Sigma_1(y_0,\ldots,y_m))$ such that
      \begin{displaymath}
        |z - \tilde{z}| \le \frac{|x-y|}{kN} \,.
      \end{displaymath}
    \end{itemize}
  \end{rem}

  Fix an orthonormal basis $(e_1,\ldots,e_m)$ of $T_o\Sigma$. For $i =
  1,\ldots,m$ we set
  \begin{align*}
    \tilde{x}_0 &:= x \,,& \tilde{x}_i &:= \tilde{x}_0 + \tfrac{|x-y|}{N}e_i \,,&
    \tilde{y}_0 &:= y &\text{and}&& \tilde{y}_i &:= \tilde{y}_0 + \tfrac{|x-y|}{N}e_i \,.
  \end{align*}
  Remark~\ref{rem:close-good-params} allows us to find
  \begin{displaymath}
    (x_0,\ldots,x_m) \in \CDxy(x,y)^{m+1} \setminus \Sigma_0
    \qquad \text{and} \qquad
    (y_0,\ldots,y_m) \in \CDxy(x,y)^{m+1} \setminus \Sigma_0 \,,
  \end{displaymath}
  such that for each $i = 0, \ldots, m$
  \begin{displaymath}
    |x_i - \tilde{x}_i| \le \frac{|x-y|}{kN}
    \qquad \text{and} \qquad
    |y_i - \tilde{y}_i| \le \frac{|x-y|}{kN} \,,
  \end{displaymath}
  We set
  \begin{align*}
    X &:= \opspan\{ \varphi(x_1) - \varphi(x_0), \ldots, \varphi(x_m) - \varphi(x_0) \} \\
    \text{and} \quad
    Y &:= \opspan\{ \varphi(y_1) - \varphi(y_0), \ldots, \varphi(y_m) - \varphi(y_0) \} \,.
  \end{align*}
  Now we have
  \begin{align}
    \| D\varphi(x) - D\varphi(y) \|
    &\le
    \| D\varphi(x) - D\varphi(x_0) \| + \| D\varphi(x_0) - D\varphi(y_0) \| + \| D\varphi(y_0) - D\varphi(y) \| \notag \\
    \label{est:DphixDphiy}
    &\le
    2\Phi\left(\tfrac{|x-y|}{kN},\CDxy(x,y)\right) + \Cr{ang-deriv} \dgras(T_{\varphi(x_0)}\Sigma,T_{\varphi(y_0)}\Sigma) \,.
  \end{align}
  Using the triangle inequality we may further write
  \begin{equation}
    \label{est:Tx0Ty0}
    \dgras(T_{\varphi(x_0)}\Sigma,T_{\varphi(y_0)}\Sigma) 
    \le \dgras(T_{\varphi(x_0)}\Sigma,X) + \dgras(X,Y) + \dgras(Y,T_{\varphi(y_0)}\Sigma) \,.
  \end{equation}
  
  \mysubsubsection*{Estimates for $\dgras(T_{\varphi(x_0)}\Sigma,X)$ and
    $\dgras(Y,T_{\varphi(y_0)}\Sigma)$}

  The first and the last term on the right-hand side of \eqref{est:Tx0Ty0} can
  be estimated as follows. For each $i=1,\ldots,m$ from the fundamental theorem
  of calculus we have
  \begin{align}
    v_i &:= \varphi(x_i) - \varphi(x_0) = \int_0^1 \tfrac{d}{dt} \left(
      \varphi(x_0 + t(x_i - x_0))
    \right)\ dt \notag \\
    &= \int_0^1 \left( D\varphi(x_0 + t(x_i - x_0)) - D\varphi(x_0) \right) (x_i - x_0) \ dt
    + D\varphi(x_0) (x_i - x_0) \notag \\
    \label{eq:viwisi}
    &=: \sigma_i + w_i \,.
  \end{align}

  From now on let us assume that $\iota$ and $k$ satisfy
  \begin{equation}
    \label{cond:iota-k}
    \iota + \frac 1k \le \Cl{iota-k} = \Cr{iota-k}(m)
    := \frac{1}{2\Cr{dist-ang} (2\Cr{gs-eps} + 24\Cr{gs-del})} \,,
  \end{equation}
  so that we can safely use Proposition~\ref{prop:red-ang} later on.

  Set $u_i := x_i - x_0$. Since $(u_1,\ldots,u_m)$ is a basis of $T_o\Sigma$ and
  $w_i = D\varphi(x_0)u_i$, the tuple $(w_1,\ldots,w_m)$ is a basis of
  $T_{\varphi(x_0)}\Sigma$. Furthermore
  \begin{align}
    \left( 1 - \tfrac 2k \right) \tfrac{|x-y|}{N}
    &\le |u_i| 
    \le \left( 1 + \tfrac 2k \right) \tfrac{|x-y|}{N} \,, \notag \\
    \intertext{hence}
    \label{est:wi-len}
    (1 - 2 \Cr{iota-k}) \tfrac{|x-y|}{N}
    \le \left( 1 - \tfrac 2k \right) \tfrac{|x-y|}{N}
    &\le |w_i| 
    \le (1 + \iota) \left( 1 + \tfrac 2k \right) \tfrac{|x-y|}{N}
    \le (1 + 2 \Cr{iota-k}) \tfrac{|x-y|}{N} \,.
  \end{align}
  Set $\tilde{u}_i := \tilde{x}_i - \tilde{x}_j$. We have $|\tilde{u}_i| = \tfrac
  1N |x-y|$ and $|u_i - \tilde{u}_i| \le \tfrac{2}{kN} |x-y|$, so we obtain
  \begin{align*}
    |\langle u_i, u_j \rangle|
    &\le |\langle u_i - \tilde{u}_i, u_j - \tilde{u}_j \rangle|
    + |\langle \tilde{u}_i, u_j - \tilde{u}_j \rangle|
    + |\langle u_i - \tilde{u}_i, \tilde{u}_j \rangle|
    + |\langle \tilde{u}_i, \tilde{u}_j \rangle| \\
    &\le \left(\tfrac{|x-y|}{N}\right)^2 \left(
      \tfrac{4}{k^2} 
      + 2 \tfrac 2k (1 + \tfrac 2k) 
    \right)
    = \left(\tfrac{|x-y|}{N}\right)^2 \left( \tfrac 4k + \tfrac{12}{k^2} \right) \,.
  \end{align*}
  Consequently
  \begin{align}
    |\langle w_i, w_j \rangle|
    &= |\langle D\varphi(x_0)u_i, D\varphi(x_0)u_j \rangle|
    = |\langle DF_o(x_0)u_i + u_i, DF_o(x_0)u_j + u_j \rangle| \notag \\
    &\le |\langle DF_o(x_0)u_i, DF_o(x_0)u_j \rangle|
    + |\langle u_i, DF_o(x_0)u_j \rangle|
    + |\langle DF_o(x_0)u_i, u_j \rangle|
    + |\langle u_i, u_j \rangle| \notag \\
    &\le \iota^2 |u_i||u_j| + 2 \iota |u_i||u_j| + |\langle u_i, u_j \rangle| \notag \\
    \label{est:wiwj}
    &\le \left(\tfrac{|x-y|}{N}\right)^2 \left(
      (1 + \tfrac 4k + \tfrac 4{k^2}) (\iota^2 + 2 \iota) + \tfrac 4k + \tfrac{12}{k^2}
    \right)
    \le 16 \Cr{iota-k} \left(\tfrac{|x-y|}{N}\right)^2 \,.
  \end{align}
  Estimates \eqref{est:wi-len} and \eqref{est:wiwj} show that $(w_1,\ldots,w_j)$
  is a $\red$-basis of $T_{\varphi(x_0)}\Sigma$ with
  \begin{align*}
    \rho_X &= \tfrac 1N |x-y| \,, \\
    \varepsilon_X &= \varepsilon_X(m) := 2 \Cr{iota-k} \\
    \text{and} \quad
    \delta_X &= \delta_X(m) := 16 \Cr{iota-k} \,.
  \end{align*}
  Moreover we have
  \begin{align*}
    |v_i - w_i| = |\sigma_i| 
    &\le \Phi(|x_i-x_0|, \CDxy(x,y)) |x_i-x_0| \\
    &\le \Phi\left((1 + \tfrac 2k)\tfrac{|x-y|}{N}, \CDxy(x,y)\right) (1 + \tfrac 2k) \tfrac{|x-y|}{N} \,.
  \end{align*}
  To apply Proposition~\ref{prop:red-ang} we need to ensure that $|v_i - w_i| <
  1$. Recalling the definition of $\Cr{eps-rad}$ one sees that $\Cr{eps-rad} <
  \frac 12$, so $|x-y| < 1$ and we have
  \begin{align*}
    \Phi \left( (1 + \tfrac 2k)\tfrac{|x-y|}{N}, \CDxy(x,y) \right) (1 + \tfrac 2k)
    &\le 2 \Phi \left(2\tfrac{|x-y|}{N}, \CDxy(x,y) \right) \\
    &\le 2 \Cr{holder-norm} E^{1/\kappa} (\tfrac 2N)^{\tau} |x-y|^{\tau} 
    < 2 (\tfrac 2N)^{\tau} \Cr{holder-norm} E^{1/\kappa} \,.
  \end{align*}
  Hence, it suffices to impose the following condition on $N$
  \begin{equation}
    \label{cond:N1}
    2 (\tfrac 2N)^{\tau} \Cr{holder-norm}  E^{1/\kappa} \le 1
    \quad \iff \quad
    N \ge 2 (4 \Cr{holder-norm} E^{1/\kappa})^{\frac 1{\tau}} \,,
  \end{equation}
  to reach the estimate
  \begin{equation}
    \label{est:TxX-ang}
    \dgras(T_{\varphi(x_0)}\Sigma,X) 
    \le \Cr{red-ang}(m,\varepsilon_X,\delta_X) (1 + \tfrac 2k) \Phi\left((1 + \tfrac 2k) \tfrac{|x-y|}{N}, \CDxy(x,y)\right) \,.
  \end{equation}
  Replacing $x_i$ by $y_i$ and repeating the same arguments we also obtain
  \begin{equation}
    \label{est:TyY-ang}
    \dgras(T_{\varphi(y_0)}\Sigma,Y) 
    \le \Cr{red-ang}(m,\varepsilon_X,\delta_X) (1 + \tfrac 2k) \Phi\left((1 + \tfrac 2k) \tfrac{|x-y|}{N}, \CDxy(x,y)\right) \,.
  \end{equation}

  \mysubsubsection*{Estimates for $\dgras(X,Y)$}

  Let
  \begin{displaymath}
    G := \CDxy(x,y) \setminus
    \left(
      \Sigma_1(x_0,\ldots,x_m) \cup \Sigma_1(y_0,\ldots,y_m)
    \right) \,.
  \end{displaymath}
  From Remark~\ref{rem:close-good-params} we know that for each point $\tilde{z}
  \in \CDxy(x,y)$ with $|z - \frac 12 (x+y)| \le (1 - \frac 1{kN}) |x-y|$ we can
  find a point $z \in G$ satisfying $|z - \tilde{z}| \le \tfrac{|x-y|}{kN}$. For
  each $i = 1,\ldots,m$ we set
  \begin{displaymath}
    \tilde{z}_0 = y_0
    \qquad \text{and} \qquad
    \tilde{z}_i := \tilde{z}_0 + \tfrac{|x-y|}{4} e_i
  \end{displaymath}
  and we find points $z_0 \in G$, \ldots, $z_m \in G$ such that $|z_i -
  \tilde{z}_i| \le \tfrac{|x-y|}{kN}$.
  Set 
  \begin{align*}
    a_i &:= \varphi(z_i) - \varphi(z_0) \,,& \tilde{a}_i &:= z_i - z_0 \,,\\
    b_i &:= \varphi(\tilde{z}_i) - \varphi(\tilde{z}_0) \,,& \tilde{b}_i &:= \tilde{z}_i - \tilde{z}_0 = \tfrac{|x-y|}{4} e_i \,,
  \end{align*}
  \begin{displaymath}
    Z := \opspan\{a_1, \ldots, a_m\} \,.
  \end{displaymath}
  Using the upper bound on the Lipschitz constant of $\varphi$ and the fact that
  $N \ge 4$ we obtain
  \begin{equation}
    \label{est:ai-len}
    (1 - 2 \Cr{iota-k}) \tfrac{|x-y|}{4} 
    \le (1 - \tfrac 2k) \tfrac{|x-y|}{4} 
    \le |a_i|
    \le (1 + \iota) (1 + \tfrac 2k) \tfrac{|x-y|}{4}
    \le (1 + 2 \Cr{iota-k}) \tfrac{|x-y|}{4}
  \end{equation}
  Note that
  \begin{align*}
    |b_i| &\le (1 + \iota) \tfrac{|x-y|}{4} \,,\\
    |a_i - b_i| &\le 2 (1 + \iota) \tfrac{|x-y|}{kN} \le \tfrac 2k (1 + \iota) \tfrac{|x-y|}{4} \,,\\
    |b_i - \tilde{b}_i| &= |F_o(\tilde{z}_i) - F_o(\tilde{z}_0)| \le \iota \tfrac{|x-y|}{4} \\
    \text{and} \quad 
    |\langle b_i, b_j \rangle| 
    &\le |\langle b_i - \tilde{b}_i, b_j - \tilde{b}_j \rangle| 
    + |\langle b_i, b_j - \tilde{b}_j \rangle| 
    + |\langle b_i - \tilde{b}_i, b_j \rangle| \\
    &\le \left(\tfrac{|x-y|}{4}\right)^2 \left(
      \iota^2 + 2 \iota (1 + \iota)
    \right) \,.
  \end{align*}
  It follows
  \begin{align}
    |\langle a_i, a_j \rangle|
    &\le |\langle a_i - b_i, a_j - b_j \rangle|
    + |\langle a_i, a_j - b_j \rangle|
    + |\langle a_i - b_i, a_j \rangle|
    + |\langle b_i, b_j \rangle| \notag \\
    &\le \left(\tfrac{|x-y|}{4}\right)^2 \left(
      \tfrac 4{k^2} (1 + \iota)^2
      + \tfrac 4k (1 + \iota)^2 (1 + \tfrac 2k)
      + \iota^2 + 2 \iota (1 + \iota)
    \right) \notag \\
    \label{est:aiaj}
    &\le  24 \Cr{iota-k} \left(\tfrac{|x-y|}{4}\right)^2 \,.
  \end{align}
  Estimates \eqref{est:ai-len} and \eqref{est:aiaj} show that $(a_1,\ldots,a_m)$
  is a $\red$-basis of $Z$ with
  \begin{align*}
    \rho_Z &= \tfrac 14 |x-y| \,,\\
    \varepsilon_Z &= \varepsilon_Z(m) := 2 \Cr{iota-k} \\ 
    \text{and} \quad
    \delta_Z &= \delta_Z(m) := 24\Cr{iota-k} \,.
  \end{align*}
  Now we only need to estimate the distances $\dist(a_i, X) = |Q_X (a_i)|$ and
  $\dist(a_i, Y) = |Q_Y (a_i)|$. Set $T :=
  (\varphi(x_0),\ldots,\varphi(x_m),\varphi(z_i))$ and $T_0 :=
  (\varphi(x_0),\ldots,\varphi(x_m))$. We know that $z_i \in G$, so for each
  $i=0,\ldots,m$ we have
  \begin{equation}
    \label{est:K}
    \K(T) = \frac{\HM^{m+1}(\simp T)}{(\diam T)^{m+2}}
    \le \left( \frac{2 E(x,y)}{\omega_{m} \omega_{m(m+1)}} \right)^{\frac 1p}
    \left(\frac{kN}{|x-y|}\right)^{\frac{m(m+2)}{p}} \,.
  \end{equation}
  The measure $\HM^{m+1}(\simp T)$ can be expressed by
  \begin{displaymath}
    \HM^{m+1}(\simp T) = \tfrac 1{m+1} \HM^m(\simp T_0) \dist(\varphi(z_i), \varphi(x_0) + X) \,.
  \end{displaymath}
  Using the above formula and \eqref{est:K} we obtain the estimate
  \begin{equation}
    \label{est:zi-X-dist}
    \dist(\varphi(z_i,) \varphi(x_0) + X) 
    \le \left( \tfrac{2 E(x,y)}{\omega_{m} \omega_{m(m+1)}} \right)^{\frac 1p}
    \tfrac{(m+1)(\diam T)^{m+2}}{\HM^m(\simp T_0)}
    \left(\tfrac{kN}{|x-y|}\right)^{\frac{m(m+2)}{p}} \,.
  \end{equation}
  Set $T_1 = (\tilde{x}_0,\ldots,\tilde{x}_m)$ and $T_2 = (x_0,\ldots,x_m)$.
  Note that 
  \begin{align*}
    T_1 &\subseteq \CBall(\tilde{x}_0, \tfrac{|x-y|}{N}) \,, \\
    \HM^{m-1}(\face_m(T_1)) &= \left( ((m-1)!)^{- \frac 1{m-1}} \tfrac{|x-y|}{N}\right)^{m-1} \\
    \text{and} \quad
    \height_m(T_1) &= \tfrac{|x-y|}{N} \,,
  \end{align*}
  hence $T_1 \in \Reg_{m-1}\big( (\tfrac 1{(m-1)!}) ^{\frac 1{m-1}},
  \frac{|x-y|}{N}\big)$. We also have $\|T_1 - T_2\| \le \frac 1k
  \frac{|x-y|}{N}$, so if we impose
  \begin{equation}
    \label{cond:k-reg}
    \tfrac 1k \le \varsigma_{m-1}\big( (\tfrac 1{(m-1)!})^{\frac 1{m-1}} \big) \,,
  \end{equation}
  then Proposition~\ref{prop:perturb} gives us $T_2 \in \Reg_{m-1}\big(\frac 12
  (\tfrac 1{(m-1)!})^{\frac 1{m-1}}, \frac 32 \frac{|x-y|}{N}\big)$. Therefore
  \begin{align}
    \label{est:T-base-meas}
    \HM^m(\simp T_0)
    &\ge \HM^m(\pi_o(\simp T_0))
    = \HM^m(\simp T_2) \\
    &\ge \tfrac 1m  \left(
      \tfrac 34 (\tfrac 1{(m-1)!})^{\frac 1{m-1}} \frac{|x-y|}{N} 
    \right)^{m}
    := \Cl{mN}(m,N) |x-y|^m \notag \,.
  \end{align}
  Of course we also have
  \begin{equation}
    \label{est:T-diam}
    \diam(T) \le (1 + \iota) \diam \{ x_0, \ldots, x_m, z_i \} 
    \le (1 + \iota) 2 |x-y| 
    \le 4 |x-y| \,.
  \end{equation}
  Combining \eqref{est:T-base-meas} and \eqref{est:T-diam} with
  \eqref{est:zi-X-dist} we get
  \begin{align}
    \dist(\varphi(z_i), \varphi(x_0) + X) 
    &\le \left( \tfrac{2 E(x,y) (kN)^{m(m+2)}}{\omega_{m} \omega_{m(m+1)}} \right)^{\frac 1p}
    \tfrac{(m+1) 4^{m+2}}{\Cr{mN}(m,N)}
    |x-y|^{2-\frac{m(m+2)}{p}} \notag \\
    \label{est:zi-X-dist-final}
    &\le \tfrac 12 \Cl{mpkN} E(x,y)^{\frac 1p} |x-y|^{\alpha} \tfrac 14 |x-y| \,,
  \end{align}
  where
  \begin{displaymath}
    \Cr{mpkN} = \Cr{mpkN}(m,p,k,N) :=
    8 \frac{2^{1/p} (m+1) 4^{m+2}}{(\omega_{m} \omega_{m(m+1)})^{1/p} \Cr{mN}(m,N)}
    (kN)^{\frac{m(m+2)}{p}} \,.
  \end{displaymath}
  Using \eqref{est:zi-X-dist-final} we can write
  \begin{align*}
    |Q_X(a_i)| &\le \dist(\varphi(z_i), \varphi(x_0) + X) + \dist(\varphi(z_0), \varphi(x_0) + X) \\
    &\le \Cr{mpkN} E(x,y)^{\frac 1p} |x-y|^{\alpha} \tfrac 14 |x-y| \,.
  \end{align*}
  Note that we can do exactly the same for $Y$ and obtain
  \begin{align*}
    |Q_Y(a_i)| &\le \dist(\varphi(z_i), \varphi(y_0) + Y) + \dist(\varphi(z_0), \varphi(y_0) + Y) \\
    &\le \Cr{mpkN} E(x,y)^{\frac 1p} |x-y|^{\alpha} \tfrac 14 |x-y| \,.
  \end{align*}
  To apply Proposition~\ref{prop:red-ang} we still need to ensure that
  \begin{displaymath}
    \Cr{mpkN} E(x,y)^{\frac 1p} |x-y|^{\alpha} < 1 \,.
  \end{displaymath}
  Of course $E(x,y) \le E$, so a sufficient condition is
  \begin{displaymath}
    |x-y| < (\Cr{mpkN}^p E)^{\frac{-1}{p - m(m+2)}} = (\Cr{mpkN}^p E)^{-1/\lambda}\,.
  \end{displaymath}
  Let us set 
  \begin{equation}
    \label{def:eps-rad2}
    \Cr{eps-rad2} = \Cr{eps-rad2}(E,m,p,k,N)
    := \min \left\{
      \Cr{eps-rad}, \tfrac 12 (\Cr{mpkN}^p E)^{-1/\lambda}
    \right\} \,.
  \end{equation}

  Now we can use Proposition~\ref{prop:red-ang} reaching the estimates
  \begin{align}
    \label{est:XZ-ang}
    \dgras(X,Z) &\le \Cr{red-ang}(m,\varepsilon_Z,\delta_Z) \Cr{mpkN} E(x,y)^{\frac 1p} |x-y|^{\alpha} \\
    \text{and} \quad
    \label{est:ZY-ang}
    \dgras(Z,Y) &\le \Cr{red-ang}(m,\varepsilon_Z,\delta_Z) \Cr{mpkN} E(x,y)^{\frac 1p} |x-y|^{\alpha} \,.
  \end{align}

  \mysubsubsection*{The iteration}
  Putting the inequalities \eqref{est:DphixDphiy}, \eqref{est:Tx0Ty0},
  \eqref{est:TxX-ang}, \eqref{est:TyY-ang}, \eqref{est:XZ-ang} and
  \eqref{est:ZY-ang} together we acquire
  \begin{align}
    \label{est:Dphi-osc}
    \| D\varphi(x) - D\varphi(y) \|
    &\le 2\Phi\left(\tfrac{|x-y|}{kN},\CDxy(x,y)\right) \\
    &\phantom{=}+ 2 \Cr{ang-deriv} \Cr{red-ang}(m,\varepsilon_X,\delta_X) (1 + \tfrac 2k)
    \Phi\left((1 + \tfrac 2k) \tfrac{|x-y|}{N}, \CDxy(x,y)\right) \notag \\
    &\phantom{=}+ 2 \Cr{ang-deriv} \Cr{red-ang}(m,\varepsilon_Z,\delta_Z) \Cr{mpkN} E(x,y)^{\frac 1p} |x-y|^{\alpha} \notag \\
    &\le \Cl{Dphi-osc1} \Phi\left(\tfrac{2|x-y|}{N}, \CDxy(x,y)\right) 
    + \Cl{Dphi-osc2} E(x,y)^{\frac 1p} |x-y|^{\alpha} \notag \,,
  \end{align}
  where
  \begin{align*}
    \Cr{Dphi-osc1} &= \Cr{Dphi-osc1}(m) 
    := 2 + 4 \Cr{ang-deriv}(m) \Cr{red-ang}(m,\varepsilon_X,\delta_X) \\
    \text{and} \quad
    \Cr{Dphi-osc2} &= \Cr{Dphi-osc2}(m,p,k,N) 
    := 2 \Cr{ang-deriv}(m) \Cr{red-ang}(m,\varepsilon_Z,\delta_Z) \Cr{mpkN}(m,p,k,N) \,.
  \end{align*}
  We define
  \begin{displaymath}
    M_p(a,\rho) := \left(
      \int_{[\varphi(\CDisc(a,\rho))]^{m+2}} \K^p\ d\mu
    \right)^{\frac 1p}
  \end{displaymath}
  Fix some $a \in \CDisc_{\Cr{eps-rad}}$ and a radius $R \in (0,\Cr{eps-rad}]$.
  Taking the supremum on both sides of \eqref{est:Dphi-osc} over all $x,y \in
  \CDisc(a,R)$ satisfying $|x-y| \le r \le R$ we attain the estimate
  \begin{displaymath}
    \Phi(r, \CDisc(a,R))
    \le \Cr{Dphi-osc1} \Phi \left( \tfrac 2N r, \CDisc(a,R+r) \right)
    + \Cr{Dphi-osc2} M_p(a,R+r) r^{\alpha} \,.
  \end{displaymath}
  Choose any $j \in \N$. Iterating the above inequality $j$ times we get
  \begin{displaymath}
    \Phi(r, \CDisc(a,R))
    \le \Cr{Dphi-osc1}^j \Phi \left( (\tfrac 2N)^j r, \CDisc(R + r_j) \right) 
    + \Cr{Dphi-osc2} M_p(a,R + r_j) r^{\alpha} 
    \sum_{l=0}^{j-1} \left( \frac{\Cr{Dphi-osc1}}{N^{\alpha}} \right)^l \,,
  \end{displaymath}
  where $r_j := r \sum_{l=0}^{j-1} N^{-l} \le 2r$. Recall that we know a priori
  that $\varphi$ is a $C^{1,\tau}$ function, so we can estimate the first term on
  the right-hand side by
  \begin{displaymath}
    \Phi \big( \big(\tfrac 2N\big)^j r, \CDisc(a,R + r_j) \big)
    \le \Cr{holder-norm} E^{1/\kappa} \big(\tfrac 2N\big)^{j \tau} r^{\tau} \,.
  \end{displaymath}
  This gives
  \begin{displaymath}
    \Phi(r, \CDisc(a,R))
    \le \Cr{holder-norm} E^{1/\kappa} r^{\tau} \left( \frac{2^{\tau}\Cr{Dphi-osc1}}{N^{\tau}} \right)^j 
    + \Cr{Dphi-osc2} M_p(a,3R) r^{\alpha} \sum_{l=0}^{j-1} \left( \frac{\Cr{Dphi-osc1}}{N^{\alpha}} \right)^l 
  \end{displaymath}
  for each $j \in \N$. To ensure that the first term disappears and that the
  second term converges when $j \to \infty$ we need to know the following
  \begin{equation}
    \label{cond:N2}
    \frac{2^{\tau}\Cr{Dphi-osc1}}{N^{\tau}} < 1
    \quad \text{and} \quad
    \frac{\Cr{Dphi-osc1}}{N^{\alpha}} < 1 \,.
  \end{equation}
  Note that $\Cr{Dphi-osc1}$ depends only on $m$ and does not depend on
  $N$. Hence, we can find big enough $N = N(m,p)$ to ensure both conditions
  \eqref{cond:N1} and \eqref{cond:N2}. Passing with $j$ to the limit $j \to
  \infty$ we obtain the bound
  \begin{displaymath}
    \Phi(r, \CDisc(a,R))
    \le \Cr{Dphi-osc2} M_p(a,3R) \sum_{l=0}^{\infty}
    \left( \tfrac{\Cr{Dphi-osc1}}{N^{\alpha}} \right)^l
    r^{\alpha}
    = \Cr{Dphi-osc2} M_p(a,3R) \frac{N^{\alpha}}{N^{\alpha} - \Cr{Dphi-osc1}} r^{\alpha} \,.
  \end{displaymath}
  Setting
  \begin{displaymath}
    \Cr{alpha-hol-norm} := \Cr{Dphi-osc2} E^{1/p} \frac{N^{\alpha}}{N^{\alpha} - \Cr{Dphi-osc1}} \,,
  \end{displaymath}
  we reach the conclusion
  \begin{displaymath}
    \forall a \in \CDisc_{\Cr{eps-rad2}}\ 
    \forall r \le \Cr{eps-rad2}
    \quad
    \Phi(r, \CDisc(a,\Cr{eps-rad2})) \le \Cr{alpha-hol-norm} r^{\alpha} \,,
  \end{displaymath}
  hence for any $x,y \in \CDisc_{\Cr{eps-rad2}}$, taking $a = \frac{x+y}2$ and
  $R = |x-y|$ we get
  \begin{displaymath}
    \|D\varphi(x) - D\varphi(y)\| \le \Cr{alpha-hol-norm} |x-y|^{\alpha} \,.
  \end{displaymath}
  
  Note that $\iota$ and $k$ satisfying \eqref{cond:iota1}, \eqref{cond:iota-k}
  and \eqref{cond:k-reg} can be chosen depending only on $m$. Hence,
  $\Cr{eps-rad}$ depends only on $E$, $m$ and $p$. Next we can choose $N$
  satisfying \eqref{cond:N1} and \eqref{cond:N2} depending only on $m$ and $p$,
  hence there exists a constant $C = C(m,p)$ such that the H{\"o}lder norm of
  $D\varphi$ is bounded by
  \begin{displaymath}
    \Cr{alpha-hol-norm} = C(m,p) E^{1/p} \,.
  \end{displaymath}
  Finally recalling \eqref{def:eps-rad2} we see that the radius $\Cr{eps-rad2}$
  of the domain of $\varphi$ can be expressed as
  \begin{displaymath}
    \Cr{eps-rad2} = C'(m,p) E^{-1/\lambda} \,,
  \end{displaymath}
  for some constant $C'(m,p)$.
  \hfill$\square$

  \begin{rem}
    Note that we actually proved a bit stronger theorem. Namely, we proved that
    there exists a constant $C = C(m,p)$ such that for each $x,y \in
    \CDisc_{\Cr{eps-rad2}}$ we have
    \begin{displaymath}
      \|D\varphi(x) - D\varphi(y)\| \le C M_p\big( \tfrac{x+y}{2},3|x-y| \big) |x-y|^{\alpha} \,.
    \end{displaymath}
  \end{rem}



\section*{Acknowledgements}
This research has been partially supported by the Polish Ministry of Science
grant no. N~N201~397737 (years 2009-2012). The author was partially supported by
the Polish Ministry of Science grant no. N~N201~611140 (years 2011-2012) and
partially by EU FP6 Marie Curie RTN programme CODY.

\bibliography{menger}{}
\bibliographystyle{hplain}

\end{document}